\documentclass[11pt]{article}
\usepackage{amsmath}
\usepackage{graphicx}
\usepackage{psfrag}
\usepackage{epsf}
\usepackage{enumerate}

\usepackage[sort,nocompress]{cite}

\usepackage{url} 
\usepackage{amsthm}
\usepackage{xpatch}
\usepackage{subcaption}
\usepackage{bm}
\usepackage{tabularx}
\usepackage{longtable}
\usepackage{mathrsfs}
\usepackage{tikz}
\usepackage{color} 
\usepackage{amssymb}
\usepackage{latexsym}
\usepackage{xr}
\usepackage{float}
\usepackage{multirow}
\usepackage{dsfont}
\usepackage{enumitem}

\usepackage[ruled]{algorithm2e}
\usepackage{lscape}
\usepackage{rotating}
\numberwithin{equation}{section}

\usepackage{mathtools}
\mathtoolsset{showonlyrefs}

\usepackage{xcolor}
\definecolor{lightblue}{HTML}{044E9E}
\usepackage[breaklinks=true, hyperfootnotes = true,colorlinks,citecolor=lightblue,urlcolor=lightblue, linkcolor=lightblue]{hyperref}

\newcommand\norm[1]{\left\lVert#1\right\rVert}

\newcommand{\1}{\boldsymbol{1}}
\newcommand{\RR}{\mathbb{R}}
\newcommand{\PP}{\mathbb{P}}
\newcommand{\EE}{\mathbb{E}}

\newcommand{\HH}{\mathbb{H}}
\newcommand{\NN}{\mathbb{N}}

\newcommand{\cls}{\mathcal{S}}
\newcommand{\clc}{\mathcal{C}}
\newcommand{\clr}{\mathcal{R}}

\newcommand{\clh}{\mathcal{H}}

\newcommand{\clt}{\mathcal{T}}
\newcommand{\cll}{\mathcal{L}}
\newcommand{\mfh}{\mathfrak{H}}

\newcommand{\clp}{\mathcal{P}}

\newcommand{\cli}{\mathcal{I}}

\newcommand{\Om}{\Omega}

\newcommand{\Lip}{\operatorname{Lip}}
\newcommand{\op}{\operatorname{op}}

\newcommand {\tr}{\operatorname{Tr}}

\newcommand{\Deriv}{D}
\newcommand{\Id}{\operatorname{Id}}

\definecolor{expcol}{rgb}{1.0,0.0,0.5}
\definecolor{ecol}{rgb}{0.0, 0.0, 1.0}

\makeatletter
\xpatchcmd{\proof}{\@addpunct{.}}{\@addpunct{:}}{}{}
\makeatother

\newtheorem{theorem}{Theorem}[section]
\newtheorem{proposition}{Proposition}[section]
\newtheorem{lemma}{Lemma}[section]
\newtheorem{corollary}{Corollary}[section]

\theoremstyle{definition}
\newtheorem{remark}{Remark}[section]
\newtheorem{example}{Example}[section]

\DeclareFontFamily{U}{mathx}{\hyphenchar\font45}
\DeclareFontShape{U}{mathx}{m}{n}{<-> mathx10}{}
\DeclareSymbolFont{mathx}{U}{mathx}{m}{n}
\DeclareMathAccent{\widebar}{0}{mathx}{"73}

\newcommand{\mockalph}[1]{}

\allowdisplaybreaks

\begin{document}

\def\spacingset#1{\renewcommand{\baselinestretch}%
{#1}\small\normalsize} \spacingset{1}

\newtheorem*{assumptionBIC*}{\assumptionnumber}
\providecommand{\assumptionnumber}{}
\makeatletter
\newenvironment{assumptionBIC}[2]
 {%
  \renewcommand{\assumptionnumber}{Assumption #1#2}%
  \begin{assumptionBIC*}%
  \protected@edef\@currentlabel{#1#2}%
 }
 {%
  \end{assumptionBIC*}
 }
\makeatother


\title{The Fourth-Moment Theorem on Hilbert Spaces}

\author{
Marie-Christine D\"uker \\ Technical University of Munich                             \and
Pavlos Zoubouloglou      \\ University of M\"unster}
\date{\today}

\maketitle

\bigskip

\begin{abstract}
\noindent
In this work, we establish conditions ensuring convergence in distribution of a sequence admitting a Wiener–Itô chaos representation to a Gaussian measure on a separable Hilbert space. Our first main result shows that, assuming convergence of the associated covariance operators in the trace-class norm, a sequence lying in a fixed Wiener–Itô chaos converges in distribution if and only if its fourth weak moments converge to the corresponding Gaussian moments. For general sequences with infinite chaos expansions, we derive analogous sufficient conditions for convergence in distribution. A key ingredient in our approach is a Stein–Malliavin bound formulated with respect to a distance that metrizes weak convergence of probability measures on separable Hilbert spaces. 
The results are infinite-dimensional extensions of the classical real-valued Fourth-Moment Theorem of Nualart and Peccati [Ann. Probab. 33, 177–193 (2005)]. 
Our work builds upon the work by Bourguin and Campese [Electron. J. Probab. 25, 1–30 (2020)] who claimed a Fourth-Moment Theorem in separable Hilbert spaces. However, a recent work by Bassetti, Bourguin, Campese, and Peccati [Stat. Probab. Lett. 233, 110671, (2026)] showed that the distance employed in the former article does not metrize weak convergence of probability measures on separable Hilbert spaces. Consequently, the conditions stated in Bourguin and Campese are not sufficient to recover a valid Fourth-Moment Theorem in the Hilbert-space setting. 
\end{abstract}

\section{Introduction}

The Fourth-Moment Theorem, introduced by Nualart and Peccati in \cite{nualart2005central}, provides a strikingly simple criterion for normal convergence within a fixed Wiener chaos. Let $q \ge 2$ be an integer, and let $F_n = I_q(f_n)$
denote a sequence of multiple Wiener-Itô integrals of order $q$ with respect to a standard Brownian motion, where each kernel $f_n \in L^2(\RR_+^q)$ is symmetric and standardized so that $\EE[F_n^2] = q!  \|f_n\|^2_{L^2(\RR_+^q)} = 1$.
The theorem states that $\{F_n\}_{n\ge 1}$ converges in distribution to a standard normal random variable if and only if the fourth moments converge to those of a standard Gaussian,
\begin{equation} \label{eq:fourth-moment-Pec-Nual}
\EE[F_n^4] \to 3, \quad \text{ as } n \to \infty.
\end{equation} 
This remarkable result reduces the challenging problem of weak convergence within a fixed Wiener chaos to a simple condition on the fourth moments. Since its discovery, the Fourth-Moment Theorem has stimulated a wide range of research, leading to numerous applications and quantitative bounds under various distances that metrize weak convergence of probability measures on $\RR^d$.
\par
An important extension of the classical Fourth-Moment Theorem concerns sequences of random variables taking values in a real separable Hilbert space $\clh$. In \cite{Bou20}, Bourguin and Campese were the first to claim a Fourth-Moment Theorem in such Hilbert spaces. 
Unfortunately, some of the main statements in \cite{Bou20} were later shown to be incorrect. As recently demonstrated in \cite{BasBurCamPec25}, the distance introduced in \cite{Bou20}, denoted here by $\rho_2$, does not metrize weak convergence of probability measures with respect to the topology of the Hilbert space. Consequently, as pointed out in \cite{BasBurCamPec25}, the results in subsequent works \cite{bourguin2024spherical,BouCamDan2024,cammarota2023quantitative,Cap2024,dukZou24,favaro2023quantitative,vidotto2025functional} that establish $\rho_2$-bounds to infer convergence in distribution in the respective Hilbert spaces are inconclusive.\footnote{The work \cite{vidotto2025functional} is currently unpublished; the authors informed us of their intention to revise it to incorporate bounds for the $d_2$ distance presented below, which does metrize weak convergence.}
\par
While \cite{BasBurCamPec25} focuses on highlighting that the distance $\rho_2$ does not metrize weak convergence of probability measures, we focus here on the validity of the infinite-dimensional Fourth-Moment Theorem of \cite{Bou20}. 
We point out why the conditions in \cite{Bou20} do not suffice, and propose conditions for a Fourth-Moment Theorem in infinite dimensions.
To describe the setting, consider a nondegenerate, centered, $\clh$-valued Gaussian random variable $Z$ and a sequence $\{F_n\}_{n\ge 1}$ of $\clh$-valued multiple Wiener-Itô integrals of fixed order $q$, which will be discussed more in Section \ref{sec:prelim} below. Nondegeneracy here means that $\operatorname{Var}(\langle Z, h \rangle_\clh) > 0$ for all nonzero $h \in \clh$. Then, Theorem 3.12 in \cite{Bou20} claims that, if
\begin{equation} \label{eq:BandC-conditions}
    \EE[\|F_n\|_{\clh}^2] \to \EE[\|Z\|_{\clh}^2], \quad \EE[\|F_n\|_{\clh}^4] \to \EE[\|Z\|_{\clh}^4],
\end{equation} 
as $n \to \infty$, then $F_n \xrightarrow{d} Z$, in the topology of $\clh$. However, as demonstrated in the following counterexample, these moment conditions lead to a certain identifiability problem, and are not sufficient to fully characterize the Gaussian limit.

\begin{example} \label{exp1}
Let $\{e_k\}_{k\in\NN}$ be an orthonormal basis of $\clh$. For $i=1,2$, let $Z_i$ be the centered $\clh$-valued Gaussian random variable with covariance operator
\begin{equation}
\clt_{Z_i} = e_i \otimes e_i, \qquad
\text{ i.e., } \quad \clt_{Z_i}(h)=\langle h,e_i\rangle_{\clh} e_i, \text{ for } h\in\clh.
\end{equation}
Then $\tr(\clt_{Z_i})=1$, so that
\begin{equation}
\EE[\|Z_1\|_{\clh}^2]=\EE[\|Z_2\|_{\clh}^2]=1
\quad\text{and}\quad
\EE[\|Z_1\|_{\clh}^4]=\EE[\|Z_2\|_{\clh}^4]=3.
\end{equation}
Nevertheless, $Z_1$ and $Z_2$ have different laws as $\clh$-valued random variables (their laws are supported on orthogonal one-dimensional subspaces  determined by $e_1$ and $e_2$ respectively). Hence, the conditions
$\EE[\|F_n\|_{\clh}^2] \to \EE[\|Z\|_{\clh}^2]$ and
$\EE[\|F_n\|_{\clh}^4] \to \EE[\|Z\|_{\clh}^4]$
do not identify a unique Gaussian limit.
\footnote{Although these random variables are degenerate, we can make them nondegenerate by taking $\widetilde Z_i = N + Z_i$, where $N$ is a centered, nondegenerate Gaussian random variable that is independent of $Z_i$, $i = 1,2$ and such that $\EE \langle N,e_1 \rangle^2 = \EE \langle N,e_2 \rangle^2$. Then, the same considerations hold.}
\end{example}

Our work addresses this issue by identifying two moment conditions that are stronger than the conditions in \eqref{eq:BandC-conditions} and ensure convergence in distribution whenever $F_n$ belongs to a fixed Wiener chaos. 
First, we require convergence of the covariance operators in the trace-class norm. This implies that the Gaussian limit is unique and avoids the issue that is raised in Example \ref{exp1}. Second, we require that,
\begin{equation} \label{eq:con-ours1}
    \EE \left[ \langle F_n, e_i \rangle^4_{\clh}\right] \to \EE \left[ \langle Z, e_i \rangle^4_{\clh}\right],
\end{equation}
for a fixed orthonormal basis $\{e_i\}_{i \in \NN}$ of $\clh$. This condition plays a role analogous to the one of Nualart and Peccati in \eqref{eq:fourth-moment-Pec-Nual}. 
The corresponding result is formalized in our first main Theorem~\ref{thm:abstract-fourth-moment-theorem}: the Fourth-Moment Theorem for Hilbert space–valued multiple Wiener integrals of fixed order.
\par
The foundation for the Fourth-Moment Theorem on Hilbert spaces for fixed chaoses and several follow-up results crucially rely on a new Stein–Malliavin bound expressed in terms of the trace-class norm, presented in Theorem \ref{thm:1}. 
Our second main result (Theorem \ref{thm:abstract-fourth-moment-theorem-infinite-chaos}) provides sufficient conditions for the convergence in distribution of the sequence $\{F_n\}$ when it admits an infinite chaos expansion. To ensure this convergence, we require that the covariance operators of the individual multiple integrals converge in trace-class norm, so that the sum of the limits defines a trace-class covariance operator characterizing the limiting Gaussian random variable. Furthermore, we impose that the contractions of the underlying kernels vanish asymptotically and satisfy a suitable summability condition in the tails. 
As part of the proof, we derive sufficient conditions for the convergence in distribution of a finite chaos expansion (Theorem \ref{cor:abstract-fourth-moment-theorem-finite-chaos}). We then show that this result suffices to also prove a central limit theorem (CLT) for sequences of vectors of fixed chaos order, i.e., sequences on $\clh^K$; see Theorem \ref{thm:abstract-fourth-moment-theorem-finite-chaos-vector}. These results recover the classical Fourth-Moment Theorem (cf. items (i) and (ii) of Theorem 5.2.7 in \cite{nourdin2012normal}) and its multivariate counterpart as direct corollaries.
Finally, Theorem \ref{thm:quantitative-CLT} provides a quantitative Fourth-Moment Theorem, establishing explicit bounds on the proximity of the laws of $F_n$ and $Z$, expressed in terms of the conditions appearing in the fourth moment theorems.
\par
Our approach brings together several technical ideas required to prove a Fourth-Moment Theorem in infinite dimensions. Our convergence in distribution statements are established with respect to the $d_2$-distance introduced by Giné and León in \cite{GinLeo80}, which metrizes weak convergence of probability measures on Hilbert spaces. 
Since there is no canonical notion of higher-order moments for infinite-dimensional random variables \cite{janson2015higher}, it is not a priori clear which norms are appropriate for bounding the $d_2$-distance.
We analyze the interaction between $d_2$-distance and Stein’s method on abstract Wiener spaces, clarifying in particular the relationship between Malliavin, Gross, and Fréchet derivatives in infinite dimensions. We leverage these tools to derive a new Malliavin–Stein bound (Theorem~\ref{thm:1}), which determines the precise mode of operator convergence. 
However, the arising trace-class norm of the difference of covariance operators is in general hard to compute. To address this, our proofs employ operator decompositions that control finite-dimensional blocks, cross-terms, and infinite-dimensional tails \eqref{eq:summands3-trace}, even for operators that are not necessarily positive semidefinite or self-adjoint (\eqref{eq:Fn-matrix-Gamma}--\eqref{eq:qm-qm-quantitative}), making trace-class estimates tractable. 
\par
Finally, we demonstrate our results and their required conditions in three examples. First, we consider the celebrated Cremers--Kadelka theorem in the theory of stochastic processes, which identifies sufficient conditions for the convergence in distribution of processes taking values in $L^2$ spaces. We obtain in Section \ref{subsec:Cremers--Kadelka} a quantitative version of the Cremers--Kadelka theorem via Theorems \ref{thm:abstract-fourth-moment-theorem} and \ref{thm:quantitative-CLT}. 
Second, in Section \ref{subsec:KRR} we apply Theorem \ref{thm:abstract-fourth-moment-theorem-infinite-chaos} to establish a functional central limit theorem for the kernel ridge regression estimator with fixed regularization parameter, under the stated design conditions and independent, identically distributed, centered errors with finite variance. A Gaussian representation of the error law yields an infinite Wiener-chaos expansion, while Theorem \ref{thm:quantitative-CLT} provides quantitative \(d_2\)-bounds.
In Section \ref{subsec:spde-general} we present the formulation of the stochastic heat equation with Dirichlet boundary conditions as an evolution equation in $L^2$. We then use Corollary \ref{cor1} to obtain weak error estimates for the Galerkin approximation to its solution. Moreover, initiating the equation from a value admitting an infinite chaos expansion, we use Theorem \ref{thm:abstract-fourth-moment-theorem-infinite-chaos} to obtain estimates on the $d_2$-distance between the solution at time $t$ and the invariant measure of the equation. 
\par
As mentioned above, several works build upon the results in \cite{Bou20} including one of our own works \cite{dukZou24}.
In light of the issues identified in \cite{BasBurCamPec25}, the statement for a CLT presented in \cite{dukZou24} is  inconclusive. The results given here build the basis to address the issues that arose in our previous article in future research.

\section{Preliminaries} \label{sec:prelim}

We start by introducing notation and recalling several definitions and technical tools used throughout the paper. Section \ref{subsec:notations} collects basic notation. Section \ref{se:Fre-Gross} introduces two different notions of differentiation of functions in infinite dimensions that are  used in this article: the Fréchet and Gross derivatives. A central aspect of our results is using an appropriate distance between probability measures on infinite-dimensional spaces, which is presented in Section \ref{se:d2-distance}. We also recall aspects of Stein’s method for infinite-dimensional Gaussian approximation in abstract Wiener spaces (Section \ref{se:Stein}). Finally, Section \ref{subsec:mall,ou,pseudo} introduces the Malliavin derivative, the Ornstein–Uhlenbeck semigroup, its generator, and its pseudo-inverse.

\subsection{Notations} \label{subsec:notations}
Let $\clh$ denote a separable Hilbert space equipped with the inner product $\langle \cdot, \cdot \rangle_{\clh}$ and norm $\| \cdot \|_{\clh}$.
Let $\{ e_i \}_{i \in \NN}$ be an orthonormal basis of $\clh$. Throughout this section, $V$ and $W$ denote two separable Banach spaces, endowed with respective norms $\|\cdot \|_V$ and $\|\cdot\|_W$. 
As is customary, we denote by $\cll(V:W)$ the space of all bounded linear operators $T : V \to W$. The space $\cll(V:W)$ is equipped with the norm
\begin{equation}
\|T\|_{\op} \doteq \sup_{\|x\|_V \le 1} \|Tx\|_W,
\end{equation}
so that $(\cll(V:W), \| \cdot \|_{\op})$ is a Banach space. When $V=W$, we write simply $\mathcal{L}(V)$ for the space of bounded linear operators on $V$. Moreover, we write $\cll(V^{\otimes k}:W)$ for the space of multilinear operators from $V^{\otimes k}$ to $W$. When needed, we emphasize the spaces by writing $\|T\|_{\op(\cll(V:W))}$.

Let $T: \clh \to \clh$ be a bounded linear operator. For $p \in [1, \infty)$, define the Schatten $p$-norm as
\begin{equation} \label{def:Schatten-p-norm}
\|T\|^p_{\mathcal{S}_{p}(\clh)} 
\doteq \tr_{\clh}(|T|^p) 
= \sum_{i=1}^{\infty} \langle |T|^p e_i, e_i \rangle_{\clh},
\quad |T| \doteq \sqrt{T^{*}T}.
\end{equation}
Let $s_1(T) \ge s_2(T) \ge \cdots \ge s_n(T) \ge \cdots \ge 0$ denote the singular values of $T$, i.e., the eigenvalues of the Hermitian operator $|T| \doteq \sqrt{T^{*}T}$. Then, 
\begin{equation} \label{def:Schatten-p-norm-singular}
    \|T\|_{\mathcal{S}_{p}(\clh)}
    =
    \left( \sum_{i=1}^\infty s_i^p(T) \right)^{\frac{1}{p}}.
\end{equation}
The space of Schatten $p$-norm operators is denoted by $\mathcal{S}_{p}(\clh)$.
Two important subclasses of Schatten operators are the trace-class operators ($p=1$) and the Hilbert–Schmidt operators ($p=2$). We refer to Chapter 3 in \cite{conway2000course} and also \cite{schatten1960norm} for more details.

If $T \ge 0$ and self-adjoint, then $\|T\|_{\mathcal{S}_1(\clh)} = \tr_{\clh}(T) = \sum_{i=1}^{\infty} \langle T e_i, e_i \rangle_{\clh}$. Moreover, the following holds for the three norms defined on $\cll(\clh)$
\begin{equation} \label{eq:bounde_HS_trace_norm}
\|T\|_{\op} \le \|T\|_{\mathcal{S}_2(\clh)} \le \|T\|_{\mathcal{S}_1(\clh)} .
\end{equation}

We denote by $\Lip_b^1(V)$ the space of uniformly Lipschitz functions $h: V \to \RR$, endowed with the norm
$\|h\|_{\Lip_b^1} =
\| h \|_{\Lip} + |h(0)| < \infty$, where 
\begin{equation*}
\|h\|_{\Lip} = \sup_{x \neq y \in V} \frac{|h(x)-h(y)|}{\|x-y\|_{V}} \le 1.
\end{equation*}

We use $x^*$ to denote an element of the dual space $V^*$. When $V$ is a Hilbert space, we use $\langle x, x^* \rangle_{V, V^*} \doteq x^*(x)$ for any $x \in V$ and $x^* \in V^*$ to denote the $V$-$V^*$ dual pairing.

We deal with direct sums (or Cartesian products) of $\clh$, denoted by $\clh^K, \; K \in \NN$. For a vector $x \in \clh^K$, we write $x^{(k)}$ to denote the $k-$th coordinate of $x$. The space $\clh^K$ is equipped with the induced inner product and its associated norm, defined as, for $x, y \in \clh^K$, 
\[
\langle x, y \rangle_{\clh^K} = \sum_{k=1}^K \langle x^{(k)}, y^{(k)} \rangle_\clh, \quad \| x \|_{\clh^K} = \sqrt{\sum_{k=1}^K \|x^{(k)} \|^2_{\clh}}.
\]

\subsection{Fréchet and Gross derivatives}
\label{se:Fre-Gross}

In this section, we define two notions of differentiation in infinite-dimensional spaces, the Fréchet and Gross derivatives. For the presentation we follow the excellent exposition of \cite{BigFerForZan24}.  Here, $H_0$ denotes a Hilbert space that is continuously embedded in $V$, with associated norm $\| \cdot \|_{H_0}$ and inner product $\langle \cdot, \cdot \rangle_{H_0}$. Later in the subsection, we provide concrete examples for $H_0$.

We say that a function $f: V\to W$ is \text{$H_0$-Fréchet differentiable} at $x\in V$ if there exists a bounded linear operator $L_x \in \mathcal{L}(H_0: W)$
such that
\begin{equation}
\lim_{\|h\|_{H_0}\to 0} \frac{\|f(x+h)-f(x)-L_x h\|_{W}}{\|h\|_{H_0}} = 0.
\end{equation}
If it exists, the operator $L_x$ is unique, called the \emph{$H_0$-Fréchet derivative} of $f$ at $x$, and is denoted by $D_{H_0}f(x) \doteq L_x$. We say that $f$ is $H_0$-Fréchet differentiable if it is $H_0$-Fréchet differentiable at every point $x\in V$.

Furthermore, we say that $f$ is twice $H_0$-Fréchet differentiable at $x\in {V}$ if (i) $f$ is $H_0$-Fréchet differentiable, and (ii) the map $D_{H_0}f : V \to \mathcal{L}(H_0: W)$ is also $H_0$-Fréchet differentiable.
In that case, the unique second-order $H_0$-Fréchet derivative of $f$ at $x$ is identified with the bilinear form
\begin{equation}
D^{2}_{H_0}f(x) \in \mathcal{L}(H_0^{2}: W).
\end{equation}

Higher-order differentiation of $f$ at some $x \in V$ along the subspace $H_0$ is then defined analogously by induction.

In the sequel, $V$ is often a Hilbert space and $W = \RR$. In this case, the Riesz representation theorem says that $D^k_{H_0} f(x)$, as an element of $\cll(H_0^{k}:\RR)$, can be identified with a unique element in $\cll(H_0^{(k-1)}:H_0)$, which we denote by $\nabla^k_{H_0} f(x)$, i.e., 
\begin{equation} \label{eq:se:Fre-Gross-R-rep}
    D^k_{H_0} f(x)(h_1,\dots,h_k) = \langle \nabla^k_{H_0} f(x)(h_1,\dots,h_{k-1}),h_k  \rangle_{H_0}, \quad h_1,\dots,h_k \in {H_0}.
\end{equation}
These definitions are clarified in the recent survey \cite{BigFerForZan24}; see Definition 2.10.

We can now define both Fréchet and Gross derivatives under the unified framework above.

\textit{Fréchet differentiability:} In the above framework, the usual notion of Fréchet differentiability corresponds to $H_0 = V$. In this case, we write $D^k_F f(x)$ instead of $D^k_{V} f(x)$ (for $k=1,2$) to denote the Fréchet derivative.

\textit{Gross differentiability:} Let $(i,\HH,V)$ be an abstract Wiener space (see Section 3.9 in \cite{Bogachev1998}), where $\HH$ is the Cameron-Martin space, and let $\gamma$ be the associated centered, and nondegenerate Gaussian measure on $V$. Then, the Gross derivative corresponds to $H_0 = \HH$. Whenever the abstract Wiener space $(i,\HH,V)$ is easily inferred from context, we write as expected $D_\HH f(x)$ to denote the Gross derivative. 

If, for $H_0 = \HH $ or $H_0 = V = \clh$, \( \nabla^2_{H_0} f(x) \) is a trace-class operator in $\cll(H_0:H_0)$, we define the Laplacian and Gross Laplacian respectively by,
\begin{equation} \label{eq:laplace-def}
   \Delta_\clh f(x) \doteq \tr_\clh (\nabla^2_\clh f(x)), \quad \Delta_\HH f(x) \doteq \tr_\HH (\nabla^2_\HH f(x)). 
\end{equation}

\textit{Relation Between Fréchet, Gross Differentiability:}
If \( f \) is Fréchet differentiable at \( x \in V \), then it is automatically \( \HH \)-differentiable at \( x \), and the \( \HH \)-derivative is the restriction of the Fréchet derivative to \( \HH \), see Proposition 2.8 in \cite{BigFerForZan24}.
\( \HH \)-differentiability does not imply Fréchet differentiability. Indeed, many functionals on Gaussian spaces are not Fréchet differentiable due to irregular behavior in directions orthogonal to \( \HH \), but they are \( \HH \)-differentiable, see Remark \ref{rmk:frechet-vs-gross-dif}\eqref{rmk:frechet-vs-gross-dif-item1} below. If $f$ is twice Fréchet differentiable, then $D_\HH f(x), x \in \clh,$ can be viewed as an element in $\clh^*$ and $D_\HH f(x) = D_F f(x)$; see p.\ 1241 in \cite{shih2011stein}. Many useful relations are derived from first principles in the paper \cite{BigFerForZan24}, see, e.g., Theorem 2.21 therein.

\begin{table}[h!] 
\renewcommand{\arraystretch}{1.2}
\small
\centering
\begin{tabularx}{\textwidth}{|l|X|X|X|X|}
\hline
\textbf{Concept} & \textbf{Our Notation} & \textbf{Section 3.4 in \cite{da2002second}}  & \textbf{\cite{shih2011stein}} & \textbf{Section 2.1 in \cite{BigFerForZan24}} \\
\hline
Fr\'echet Derivative ($H_0 = \clh$) & $D_F$ & $D$ & $'$ & $D$ \\
\hline
Gross Derivative ($H_0 = \HH$) & $D_\HH$ & $D_G$ & $D$ (context-dependent) & $D_{\HH}$ \\
\hline
Fr\'echet Gradient ($H_0 = \clh$) & $\nabla_F$ & $D$ & $'$ & $\nabla$ \\
\hline
Gross Gradient ($H_0 = \HH$) & $\nabla_\HH$ & $D_G$ & $D$ (context-dependent) & $\nabla_{\HH}$ \\
\hline
Laplacian ($H_0 = \clh$) & $\tr_\clh, \Delta_\clh$ & $\tr$ & Not available & Not available \\
\hline
Gross Laplacian ($H_0 = \HH$) & $\tr_\HH, \Delta_\HH$ & $\tr_G, \Delta_G$& $\tr_G, \Delta_H$ & Not available \\
\hline
\end{tabularx}
\caption{Notation comparison for Fr\'echet and Gross derivatives and gradients in commonly referenced works.}
\label{table1}
\end{table}

\subsection{The \texorpdfstring{$d_2$}{d_2}-distance} \label{se:d2-distance}

In this section, we present a metric on $\clp(\clh)$, the space of probability measures on $\clh$, that metrizes weak convergence. This metric was introduced by Gin\'{e} and Le\'{o}n in \cite{GinLeo80}. The recent write-up \cite{BasBurCamPec25} illuminates the differences between this distance and the metric $\rho_2$ that was employed in \cite{Bou20}.

Recall that, for an (everywhere) Fréchet differentiable function $h:\clh \to \RR$, the $k$-th (Fréchet) derivative $\Deriv^k_F h$ can be identified with a map from $\clh$ to $\cll(\clh^{k}:\RR)$. We denote by $\clc_b^k(\clh)$ the space of bounded, $\RR$-valued functions on $\clh$, admitting $k$ Fréchet derivatives, i.e., $h \in \clc_b^k(\clh)$ if 
\begin{equation} \label{eq:norm-def-cb2}
\| h \|_{\clc_b^k(\clh)} =  \sup_{x \in \clh} \left(  \sum_{j=0}^k \| \Deriv_F^j h(x) \|_{\op(\cll(\clh^{\otimes j}:\RR))} \right) < \infty,
\end{equation}
under the convention that $\| D^0_F h(x) \|_{\cll(\clh^0: \RR)} = \sup_{x \in \clh} |h(x)|$.  
Then, the $d_j$ metric on $\clp(\clh)$ is defined as, for $\mu, \nu \in \clp(\clh)$,
\begin{equation} \label{eq:d2metric}
d_j(\mu,\nu) \doteq \sup_{h \in \clc^j_b(\clh), \|h\|_{\clc_b^j(\clh)} \le 1} \left|\int_{\clh} h(x) (\mu(dx) - \nu(dx)) \right|,
\quad
\text{ for }
\quad
j \geq 1.
\end{equation}
By Theorem 2.4 in \cite{GinLeo80}, for $j \ge 1$, $d_j$ (and in particular, $d_2$) metrizes weak convergence in $\clp(\clh)$. For two random variables $X, Z$ in $\clh$, we write $d_2(X,Z)$, meaning $d_2(\cll(X),\cll(Z))$, where $\cll(\cdot)$ denotes the law of a random variable. In terms of the random variables $X,Z$, the $d_2$ distance is recast as
\[
d_2(X,Z) \doteq \sup_{h \in \clc^2_b(\clh), \|h\|_{\clc_b^2(\clh)} \le 1} \left|\EE [h(X)] - \EE [h(Z)] \right|.
\]

\subsection{Stein's method in abstract Wiener spaces} \label{se:Stein}

Let $(i, \HH, V)$ be an abstract Wiener space as in Section \ref{se:Fre-Gross}, with an associated Gaussian measure $\gamma_Z$ on $V$. We write $\gamma_Z$ to emphasize that $Z$ here denotes a $V$-valued random variable with law $\gamma_Z$. Let $h: V \to \RR$ be a measurable function in the test class $\Lip_b^1(V)$.

The \emph{Stein operator} associated with $\gamma_Z$ acts on sufficiently smooth functions $f: V \to \RR$ as
\begin{equation} \label{eq:Stein-operator}
\mathcal{A} f(x) \doteq \langle x, D_\HH f(x) \rangle_{V,V^*} - \Delta_\HH f(x),
\end{equation}
where $D_\HH f(x)$ and $\Delta_\HH f(x) = \tr_\HH(D_\HH^2 f(x))$ were introduced in Section \ref{se:Fre-Gross}. 
For a fixed function $h$, the Stein operator gives rise to the so-called \textit{Stein equation} corresponding to $h$, namely the function $f_h: V \to \RR$ that solves the following equation
\begin{equation} \label{eq:Stein-equation}
\mathcal{A} f_h(x) = h(x) - \int_V h   d\gamma_Z, \qquad x \in V.
\end{equation}
When $h\in\Lip_b^1(V)$, existence of a solution to the Stein equation follows from Shih \cite{shih2011stein}. With $\mathcal{A}$ as in \eqref{eq:Stein-operator}, we use the semigroup solution
\begin{equation} \label{eq:Stein-solution}
    f_h(x) = \int_0^\infty \big( P_t h(x) - \EE[h(Z)] \big)\,dt,
\end{equation} 
where $\{P_t\}_{t \ge 0}$ is the semigroup associated to the Ornstein-Uhlenbeck process and is defined by Mehler's formula as follows.
\begin{equation} \label{eq:Pt-shih}
    P_t h(x) \doteq \int_V h \left(e^{-t} x + \sqrt{1-e^{-2t}}  y \right) \gamma_Z(dy), 
    \qquad t \ge 0.
\end{equation} 

\subsection{Malliavin derivative for Hilbert space-valued random variables} \label{subsec:mall,ou,pseudo}

We collect in this section some results on the Malliavin derivative, the OU semigroup, its generator, and the induced pseudo-inverse of Hilbert space-valued random variables. We follow here the careful treatment of \cite{vidotto2025functional}, and we refer to this article for more details. For $\RR$-valued random variables, these results are classical and can be found in, e.g., Chapter 2 in \cite{nourdin2012normal}.

Let $(\Omega, \mathcal{F}, \mathbb{P})$ be a probability space, and let $
W = \{ W(h) : h \in \mfh\}$ be an isonormal Gaussian process defined on this probability space. Here, $\mfh$ is a real separable Hilbert space, equipped with the norm $\| \cdot \|_{\mfh}$ and an inner product $\langle \cdot, \cdot \rangle_{\mfh}$, and denote by $\{h_i\}_{i \in \NN}$ an orthonormal basis of $\mfh$. By definition, $W$ is a centered Gaussian family satisfying  $\EE[W(h) W(g)] = \langle h, g \rangle_{\mfh}$, for $h,g \in \mfh$.

When $\clh = \RR$, classical Wiener-Itô chaos decomposition (e.g., Corollary 2.7.8 in \cite{nourdin2012normal}) asserts that every square integrable random variable $G \in L^2((\Omega, \sigma \{W\}, \PP) : \RR)$ that is measurable with respect to $\sigma \{W\}$ admits the decomposition
\begin{equation} \label{eq:eq:chaos-decomp-real}
    G = \EE [G] + \sum_{n =1}^\infty I_n(g_n),
\end{equation}
with $g_n \in \mfh^{\odot n}$ suitable symmetric kernels and $I_n$ the usual multiple Wiener integral of order $n$ with regard to $W$.

In analogy, any $F \in L^2(\Omega: \mathcal{H})$ that is measurable with respect to $W$ admits a Wiener-Itô chaos expansion
\begin{equation} \label{eq:chaos-decomp-F}
    F = \EE[F] + \sum_{n =1}^\infty \cli_n(f_n),
\end{equation}
where $f_n \in \mfh^{\odot n} \otimes \mathcal{H}$ are symmetric kernels, and $\cli_n(f_n) \in L^2(\Om:\clh)$ denotes the $\clh$-valued, $n$-th multiple Wiener-Itô integral satisfying $\EE[\cli_n(f_n)] = 0$. Here $\mfh^{\odot n} \subset \mfh^{\otimes n}$ denotes the $n$-th symmetric tensor product of $\mfh$; see Chapter 1 in \cite{Nua06book} and Chapter 2 in \cite{nourdin2012normal}. The multiple Wiener integrals $I_n$ and $\cli_n$ (as well as their respective kernels) are related through the relationship
\begin{equation}
    \cli_n = I_n \otimes \Id_\clh, \quad  \mfh^{\odot n} \otimes \mathcal{H} \ni f_n = \sum_{i=1}^\infty \langle f_{n}, e_i \rangle_{\clh} \otimes e_i \mapsto \cli_n (f_n) = \sum_{i=1}^\infty I_n(\langle f_{n}, e_i \rangle_{\clh}) e_i.
\end{equation}
For careful definitions of these objects in general Hilbert spaces and proofs of the chaos decomposition claim, see Theorem 4.2 in \cite{dukZou24} or Lemma 2.2 in \cite{vidotto2025functional} (and note that the definitions therein coincide). Henceforth, we assume that the random variables are centered.

We write $D_M$ for the usual Malliavin derivative to distinguish it from the Fr\'echet derivative $D_F$ and the Gross derivative $D_\HH$.
The Malliavin derivative $D_M$ can then be defined in terms of the chaos expansion \eqref{eq:chaos-decomp-F} as follows
\begin{equation} 
    D_M F \doteq \sum_{n \ge 1} n   \cli_{n-1}(f_n) \in L^2(\Omega: \mfh \otimes \mathcal{H}),
\end{equation}
where $F$ belongs to the domain of $D_M$ defined as
\begin{equation}
\mathbb{D}^{1,2}(\mathcal{H}) \doteq \Big\{ F \in L^2(\Omega: \mathcal{H})  \text{ such that }  F \text{ is as in }   \eqref{eq:chaos-decomp-F}    \text{ and }   \sum_{n \ge 1} n \ n! \| f_n\|^2_{\mfh^{\otimes n} \otimes \clh } < \infty \Big\}.
\end{equation}
Analogous definitions can be given for the higher-order Malliavin derivatives $D_M^k$, as well as their respective domains $\mathbb{D}^{k,2}(\mathcal{H})$. However, we only use $k = 1$ in the present article.

Cylindrical $\mathcal{H}$-valued random variables are dense in $\mathbb{D}^{1,2}(\mathcal{H})$ (see, e.g., Section 2 in \cite{leon1998stochastic}), ensuring that this definition agrees with the classical derivative on smooth cylindrical functionals.

The divergence operator $\delta$ is defined as the adjoint of the Malliavin derivative $D_M$. More precisely, let $u \in L^2(\Omega :  \mfh \otimes \clh)$ be in the domain of $\delta$ (denoted $\operatorname{dom}(\delta)$) if there exists a constant $C$ depending on $u$ such that, 
\[
| \EE \langle D_M F, u \rangle_{\mfh \otimes \clh} | \le C \| F \|_{L^2(\Om:\clh)}
\]
for all $F \in \mathbb{D}^{1,2}(\clh)$. Then, for all $(F,u) \in \mathbb{D}^{1,2}(\mathcal{H}) \times \operatorname{dom}(\delta)$, the random variable $\delta(u) \in L^2(\Omega: \clh)$ is defined by the duality relation
\begin{equation} \label{eq:2.24inPoincare}
\EE[\langle F, \delta(u)\rangle_\clh] = \EE[\langle D_M F, u\rangle_{\mfh \otimes \clh}]
\quad \text{for all } F \in \mathbb{D}^{1,2}(\mathcal{H});
\end{equation}
see Sections 2.4 and 2.6 in \cite{nourdin2012normal} or Section 2.2.3 in \cite{vidotto2025functional} for details. 

The Ornstein-Uhlenbeck semigroup $\{ P_t : t \ge 0\}$ can also be defined to act on Wiener chaos as
\begin{equation} \label{eq:Pt-def-random-variable}
    P_t F \doteq \EE[F] + \sum_{n\ge1} e^{-n t} \cli_n(f_n),
\end{equation}
whenever $F$ admits the representation in \eqref{eq:chaos-decomp-F}. It will be clear from context whether we use the definition of the semigroup given in \eqref{eq:Pt-def-random-variable} (suitable for random variables), or the one in \eqref{eq:Pt-shih} (suitable for deterministic functions). 

Associated to the semigroup $\{P_t\}$ in \eqref{eq:Pt-def-random-variable} is its infinitesimal generator $L$, which acts on random variables $F$ as in \eqref{eq:chaos-decomp-F} by
\begin{equation}
L F \doteq \sum_{n\ge1} (-n) \cli_n(f_n),
\end{equation}
with domain 
\begin{equation} 
\mathrm{dom}(L) \doteq \Big\{F \in L^2(\Omega: \mathcal{H})  \text{ such that } F \text{ is as in } \eqref{eq:chaos-decomp-F}    \text{ and }   \sum_{n \ge 1} n^2 n! \| f_n\|^2_{\mfh^{\otimes n} \otimes \clh } < \infty \Big\}.
\end{equation}
The following identity connects the generator, the adjoint operator, and the Malliavin derivative
\begin{equation} \label{eq:L-def}
    L = - \delta D_M.
\end{equation}
Finally, for random variables $F$ as in \eqref{eq:chaos-decomp-F}, the pseudo-inverse $L^{-1}$ of the generator is defined by
\begin{equation} \label{eq:L-inv-def}
L^{-1} F \doteq - \sum_{n\ge1} \frac{1}{n} \cli_n(f_n),
\end{equation}
so that $L L^{-1} F = F - \EE F$.

\section{Abstract bound and Fourth-Moment Theorems}

\begin{theorem}
  \label{thm:1}
  Let $Z$ be a centered Gaussian random variable on $\clh$ with covariance operator $\clt_Z$. Then, for all centered $F \in \mathbb{D}^{1,2}(\clh)$, it holds that
  \begin{align} 
    d_2 (F,Z)
    &\leq   
    \frac{1}{2}
    \left\lVert 
    \langle D_M F, - D_M L^{-1}F \rangle_{\mfh}  - \clt_Z 
    \right\rVert_{L^1(\Omega: \mathcal{S}_1(\clh))},
  \end{align}
  where the operator $L^{-1}$ was defined in \eqref{eq:L-inv-def}.
\end{theorem}

\begin{proof}
We first prove the claim under the additional assumption that $Z$ is nondegenerate.
Denote by $\gamma_Z$ the Gaussian measure on $\clh$ associated with the random variable $Z$, and recall that $\mathbb{H}$ is the Hilbert space associated to the abstract Wiener space formulation of $\gamma_Z$ on $\clh$. Then, with explanations given below, 
\begin{align} 
    d_{2}(F,Z)
    &=
    \sup_{h \in \clc^2_b(\clh), \|h\|_{\clc_b^2(\clh)} \leq 1} 
    | \EE [ h(F) ] - \EE [h(Z)] |
    \nonumber
    \\ &\le
    \sup_{h \in \clc^2_b(\clh), \|h\|_{\clc_b^2(\clh)} \leq 1} 
    \left|
    \EE \left( 
    \left\langle F,\Deriv_\HH f_h(F) \right\rangle_{\clh,\clh^{\ast}} 
    - \Delta_{\HH} f_h(F) 
    \right) 
    \right|
    \label{eq:le1-1}
    \\ &=
     \sup_{h \in \clc^2_b(\clh), \|h\|_{\clc_b^2(\clh)} \leq 1} 
     \left|
    \EE \left( 
    \left\langle F,\Deriv_\HH f_h(F) \right\rangle_{\clh,\clh^{\ast}} 
    - \tr_\HH \nabla^2_\HH f_h(F) 
    \right) 
    \right|,
    \label{eq:le1-2}
\end{align}
where $f_h$ was defined in \eqref{eq:Stein-solution}, and is twice Fréchet differentiable by Lemma \ref{lem:4}. This also implies that it is twice Gross differentiable; see Proposition 2.8 in \cite{BigFerForZan24}.
In the calculations above, \eqref{eq:le1-1} follows by the inclusion $\clc_b^2(\clh) \subset \Lip_b^1(\clh)$ and Theorem 4.10 in \cite{shih2011stein} (but note the adjusted notation here), and recall the notation of the $\clh$-$\clh^*$ dual pairing.
Furthermore, \eqref{eq:le1-2} is due to the definition of the Gross Laplacian; see \eqref{eq:laplace-def}.
For the first term in \eqref{eq:le1-2}, $D_\HH f_h(F)$ is an element in $\clh^*$ and, again by Proposition 2.8 in \cite{BigFerForZan24},
\begin{align} 
    \EE \left( 
    \left\langle F,\Deriv_\HH f_h(F) \right\rangle_{\clh,\clh^{\ast}} 
    \right)  
    = 
    \EE \left( 
    \left\langle F, D_F f_h(F) \right\rangle_{\clh,\clh^{\ast}}  
    \right);
    \label{eq:910-2}
\end{align}
see also p.\ 1241 in \cite{shih2011stein}, but note again the differences in notation that were highlighted in Table \ref{table1}. Moreover,
\begin{align}
    \EE \left(
    \left\langle F,\Deriv_F f_h(F) \right\rangle_{\clh,\clh^*}
    \right)
    &= 
    \EE \left( 
    \left\langle F,\nabla_F f_h(F) \right\rangle_{\clh}  
    \right)
    \label{eq:910-4-1} 
    \\
    &=
    \EE \left(
    \left\langle L L^{-1}F, \nabla_F f_h(F) \right\rangle_{\clh}
    \right) 
    \nonumber
    \\ &=
    \EE \left( 
    \tr_\clh \left(
    \langle -D_M L^{-1} F, D_M \nabla_F f_h(F) \rangle_{\mfh}
    \right) 
    \right) 
    \label{eq:910-4}
    \\ &=
    \EE \left( 
    \tr_\clh \left( 
    \nabla_F^2 f_h(F)
    \langle D_M F, - D_M L^{-1}F \rangle_{\mfh} 
    \right) 
    \right), 
    \label{eq:910-5} 
\end{align}
where \eqref{eq:910-4-1} follows by the Riesz representation theorem; see \eqref{eq:se:Fre-Gross-R-rep}.
Moreover, \eqref{eq:910-4} and \eqref{eq:910-5} are respectively due to Lemma \ref{le:analogue-Th-2.6} \ref{item3:Dirichlet-Th-2.6} and \ref{item2:Dirichlet-Th-2.6} and since $\nabla_F f_h(F)$ is Fréchet differentiable with bounded second derivative; see Lemma \ref{lem:4}. The boundedness of $\nabla_F f_h$ and $\nabla^2_F f_h$ in their respective operator norms also yields the Malliavin differentiability $\nabla_F f_h(F) \in \mathbb{D}^{1,2}(\clh)$.

Note further that
\begin{equation}
    \tr_\HH \left( \nabla^2_\HH f_h(F) \right)
    =
    \tr_\clh \left( \nabla_F^2 f_h(F) \clt_Z \right)
    \label{eq:910-3}
\end{equation}
from (3.4.4) in \cite{da2002second}. This identity shows how $\clt_Z$ appears in the statement.

Combining \eqref{eq:910-2}, \eqref{eq:910-5}, and \eqref{eq:910-3} yields, for any $h\in\clc_b^2(\clh)$,
\begin{equation}
\EE\left[\langle F,\nabla_F f_h(F)\rangle_\clh
-\tr_\HH\nabla_\HH^2 f_h(F)\right] =\EE\left[\tr_\clh\left(
\nabla_F^2 f_h(F)
\left(\langle D_MF,-D_ML^{-1}F\rangle_\mfh-\clt_Z\right)
\right)\right].\label{eq:188}
\end{equation}

Since, for $h_1,h_2 \in \clh$,
\begin{equation} 
    \label{eq:910-6} 
    \langle \nabla^2_F f_h(F) [h_1],h_2 \rangle_{\clh} 
    = 
    D^2_F f_h(F) [h_1,h_2],
\end{equation}
we have $\| \nabla^2_F f_h(F) \|_{\op(\cll(\clh:\clh))} 
=
\| D^2_F f_h(F) \|_{\op(\cll(\clh^{\otimes 2}: \RR))}$.

The two Malliavin derivatives can be identified with square-integrable Hilbert--Schmidt operators from $\clh$ to $\mfh$. The product of the adjoint of the second with the first is therefore integrable in trace-class norm, which implies that $\langle D_M F, - D_M L^{-1}F \rangle_{\mfh}$ is trace-class. Since, in addition, $\nabla_F^2 f_h(F)$ is a bounded operator and $\clt_Z$ is trace-class, 
\begin{align}
    &
    \left|
    \EE\left[ 
    \tr_\clh
    \left(
    \nabla_F^2 f_h(F) 
    \left( 
    \langle D_M F, - D_M L^{-1}F \rangle_{\mfh} 
    - \clt_Z 
    \right)
    \right)
    \right]
    \right|
    \nonumber
    \\
    &\leq
    \EE \left(
    \left\lVert \nabla_F^2 f_h(F) \right\rVert_{\op(\cll(\clh:\clh))}
    \left\| 
    \langle D_M F, - D_M L^{-1}F \rangle_{\mfh} 
    - \clt_Z 
    \right\|_{\cls_1(\clh)} 
    \right)
    \label{eq:189}
    \\
    &\leq
    \sup_{u \in \clh}
    \left\lVert \Deriv_F^2 f_h(u) \right\rVert_{\op(\cll(\clh^{\otimes 2}:\RR))}
    \EE 
    \left\| 
    \langle D_M F, - D_M L^{-1}F \rangle_{\mfh} 
    - \clt_Z 
    \right\|_{\cls_1(\clh)}
    \nonumber
    \\
    &\leq
    \frac{1}{2}
    \left\lVert
    \langle D_M F, - D_M L^{-1}F \rangle_{\mfh} 
    - \clt_Z
    \right\rVert_{L^1(\Omega:\mathcal{S}_1(\clh))},
    \label{eq:191}
\end{align}
where \eqref{eq:189} follows by Theorem 18.11 (e) in \cite{conway2000course} and \eqref{eq:191} is a consequence of \eqref{eq:8} in Lemma \ref{lem:4}. Combining \eqref{eq:le1-2}, \eqref{eq:910-2}, \eqref{eq:188}, and \eqref{eq:191} proves the desired bound whenever $Z$ is nondegenerate.

It is left to remove the nondegeneracy assumption on $Z$. Let $Q \in \mathcal{S}_1(\clh)$ be a positive, self-adjoint, and nondegenerate operator.
Let $G$ be a centered Gaussian random variable on $\clh$ with covariance operator $Q$, independent of $Z$, and, for $\varepsilon>0$, set
\begin{equation}
    Z_\varepsilon
    \doteq
    Z+\sqrt{\varepsilon}\,G.
    \quad
    \text{ with }
    \quad
    \clt_{Z_\varepsilon}
    =
    \clt_Z+\varepsilon Q,
    \label{eq:regularized-covariance}
\end{equation}
such that $Z_\varepsilon$ is a centered Gaussian random variable with covariance operator $\clt_{Z_\varepsilon}$.
Moreover, for every nonzero $h\in\clh$,
\begin{align}
    \left\langle 
    \clt_{Z_\varepsilon} h,h
    \right\rangle_{\clh}
    =
    \left\langle 
    \clt_Z h,h
    \right\rangle_{\clh}
    +
    \varepsilon
    \left\langle Qh,h\right\rangle_{\clh}
    \geq
    \varepsilon
    \left\langle Qh,h\right\rangle_{\clh}
    >0,
    \label{eq:regularized-nondegenerate}
\end{align}
where the last inequality follows from the positivity and nondegeneracy of $Q$. Hence, $Z_\varepsilon$ is nondegenerate for every $\varepsilon>0$.
We may therefore apply the bound \eqref{eq:191} established above with $Z$ replaced by $Z_\varepsilon$. Together with \eqref{eq:regularized-covariance} and the triangle inequality in $\mathcal{S}_1(\clh)$, this gives
\begin{align}
    d_2(F,Z_\varepsilon)
    &\leq
    \frac{1}{2}
    \left\|
    \langle D_M F,-D_ML^{-1}F\rangle_{\mfh}
    -\clt_Z-\varepsilon Q
    \right\|_{L^1(\Omega:\mathcal{S}_1(\clh))}
    \nonumber
    \\
    &\leq
    \frac{1}{2}
    \left\|
    \langle D_M F,-D_ML^{-1}F\rangle_{\mfh}
    -\clt_Z
    \right\|_{L^1(\Omega:\mathcal{S}_1(\clh))}
    +
    \tr_\clh(Q) \frac{\varepsilon}{2},
    \label{eq:regularized-Stein-bound}
\end{align}
where in the last line we used that $Q$ is deterministic, nonnegative, and 
$\|Q\|_{\mathcal{S}_1(\clh)}=\tr_\clh(Q)$.

On the other hand, by the definition of $d_2$ and since every 
$h\in\clc_b^2(\clh)$ with $\|h\|_{\clc_b^2(\clh)}\leq1$ has Lipschitz constant bounded by one,
\begin{align}
    d_2(Z_\varepsilon,Z)
    \leq
    \EE\|Z_\varepsilon-Z\|_{\clh}
    =
    \sqrt{\varepsilon}\,
    \EE\|G\|_{\clh}
    \leq
    \sqrt{\varepsilon}\,
    \left(\EE\|G\|_{\clh}^2\right)^{1/2}
    =
    \sqrt{\varepsilon}\,
    \left(\tr_\clh(Q)\right)^{1/2}.
    \label{eq:regularized-d2}
\end{align}
Combining the triangle inequality for $d_2$ and \eqref{eq:regularized-Stein-bound}--\eqref{eq:regularized-d2},
\begin{align}
    d_2(F,Z)
    &\leq
    d_2(F,Z_\varepsilon)
    +
    d_2(Z_\varepsilon,Z)
    \nonumber
    \\
    &\leq
    \frac{1}{2}
    \left\|
    \langle D_M F,-D_ML^{-1}F\rangle_{\mfh}
    -\clt_Z
    \right\|_{L^1(\Omega:\mathcal{S}_1(\clh))}
    +
    \frac{\varepsilon}{2}\tr_\clh(Q)
    +\sqrt{\varepsilon}\,\bigl(\tr_\clh(Q)\bigr)^{1/2}.
    \label{eq:regularized-final}
\end{align}
Letting $\varepsilon\downarrow0$ yields the conclusion.
\end{proof}

The following remark 
explains how Theorem \ref{thm:1} relates to the
classical scalar Malliavin--Stein bound, as well as the bound that was obtained in \cite{Bou20}.

\begin{remark}\label{rmk:scalar-Schatten}
\begin{enumerate}
    \item For centered, real-valued $F\in\mathbb D^{1,2}$ and $N\sim\mathcal N(0,1)$, the classical Malliavin--Stein bound (cf. Theorem 5.13 in \cite{nourdin2012normal}) is given by
\[
d_W(F,N)\leq 
\EE\left|1-\langle D_MF,-D_ML^{-1}F\rangle_\mfh\right|,
\]
where $d_W$ denotes the Wasserstein distance defined using real-valued test functions with Lipschitz constant at most one. Theorem \ref{thm:1} specialized to $\clh=\RR$ then implies the corresponding estimate for the $d_2$-distance with constant $1/2$. If $F=I_p(f)$, then $-L^{-1}F=F/p$, and hence
\[
d_2(F,N)\leq\frac12\,
\EE\left|1-\frac1p\|D_MF\|_\mfh^2\right|.
\]
Thus the factor $1/p$ is already incorporated through the pseudo-inverse of the Ornstein--Uhlenbeck generator.

\item More generally, write $\Gamma(F,G)=\langle D_MF,D_MG\rangle_\mfh$, as in Lemma \ref{le:analogue-Th-2.6}. Let $r,r'\in[1,\infty]$ satisfy $1/r+1/r'=1$, with $1/\infty=0$. In this remark, $\cls_\infty(\clh)=\cll(\clh:\clh)$ is equipped with the operator norm. For $h\in\clc_b^2(\clh)$ whose Hessian is uniformly bounded in $\cls_r(\clh)$, the argument for \eqref{eq:D2fh-form} gives
\[
\sup_{x\in\clh}\|\nabla_F^2 f_h(x)\|_{\cls_r(\clh)}
\leq\frac12\sup_{x\in\clh}\|\nabla_F^2 h(x)\|_{\cls_r(\clh)}.
\]
Hence, combining the regularization argument as in the proof of Theorem \ref{thm:1}, \eqref{eq:188}, and Schatten H\"older duality, we obtain
\begin{equation}\label{eq:Schatten-pairing-bound}
\begin{split}
\left|\EE[h(F)]-\EE[h(Z)]\right|
&\leq\frac12\sup_{x\in\clh}\|\nabla_F^2 h(x)\|_{\cls_r(\clh)} 
\|\Gamma(F,-L^{-1}F)-\clt_Z\|_{L^1(\Omega:\cls_{r'}(\clh))}.
\end{split}
\end{equation}
Taking $r=\infty$ gives the operator-norm/trace-class pairing used here, while choosing $r=2$ together with Cauchy--Schwarz gives the Hilbert--Schmidt pairing of \cite{Bou20} in terms of $L^2(\Omega:\cls_2(\clh))$. 

Although \eqref{eq:Schatten-pairing-bound} yields a family of bounds parametrized by $r \in [1,\infty]$, the important insight of \cite{BasBurCamPec25} was that, to obtain a distance that metrizes weak convergence in separable infinite-dimensional Hilbert spaces $\mathcal{H}$, one must necessarily take test functions $h$ whose derivatives in \eqref{eq:Schatten-pairing-bound} with $r = \infty$ is finite, i.e., $\sup_{x\in\clh}\|\nabla_F^2 h(x)\|_{\cls_\infty(\clh)} < \infty$. 
\end{enumerate}
\end{remark}

The following corollary reproves Proposition 2 of \cite{BasBurCamPec25} for the case $p=\infty$, $q=1$.

\begin{corollary}
\label{cor1}
  Let $Z_1,Z_2$ be centered, Gaussian random variables on $\clh$ with covariance operators $\clt_{Z_1},\clt_{Z_2}$. Then, 
  \begin{align} 
    d_2 (Z_1,Z_2)
    &\leq   
    \frac{1}{2}
    \left\lVert \clt_{Z_1}  - \clt_{Z_2} \right\rVert_{\mathcal{S}_1(\clh)}.
\end{align}   
\end{corollary}

\begin{proof}
    By Theorem \ref{thm:1}, we have
    \begin{align} 
    d_2 (Z_1,Z_2)
    &\leq   
    \frac{1}{2}
    \left\lVert \langle D_M Z_1, - D_M L^{-1}Z_1 \rangle_{\mfh}  - \clt_{Z_2} \right\rVert_{L^1(\Omega:\mathcal{S}_1(\clh))}.
\end{align}
Then the statement of the corollary can be inferred by showing that for all $u,v \in \clh$, 
\begin{equation}
    \langle \langle D_M Z_1, - D_M L^{-1}Z_1 \rangle_{\mfh}
    u, v \rangle_{\clh} = \langle \clt_{Z_1} u, v \rangle_{\clh}.
\end{equation}
Since $Z_1$ is Gaussian, it admits a Wiener-It\^o chaos representation $Z_1 = \mathcal{I}_{1}(f)$ where $f \in \mfh \otimes \clh$ so that $D_M Z_1 = D_M \mathcal{I}_{1}(f)= f$. 
Then, recalling \eqref{eq:L-inv-def},
\begin{equation} \label{eq1-cor1}
\begin{aligned}
    \langle \langle D_M Z_1, - D_M L^{-1}Z_1 \rangle_{\mfh}
    u, v \rangle_{\clh}
    &=
    \langle 
    \langle D_M \mathcal{I}_{1}(f), u \rangle_{\clh}, 
    \langle - D_M L^{-1}\mathcal{I}_{1}(f), v \rangle_{\clh} \rangle_{\mfh} 
    \\&=
    \langle 
    \langle f, u \rangle_{\clh}, 
    \langle f , v \rangle_{\clh}
    \rangle_{\mfh}.
\end{aligned}
\end{equation}
On the other hand, by the isometry property, we have
\begin{equation} \label{eq2-cor1}
\begin{aligned}
    \langle \clt_{Z_1} u, v \rangle_{\clh} 
    &= 
    \EE [ \langle Z_1 , u \rangle_{\clh} 
    \langle Z_1, v \rangle_{\clh} ]
    =
    \EE [ \langle \mathcal{I}_{1}(f) , u \rangle_{\clh} 
    \langle \mathcal{I}_{1}(f), v \rangle_{\clh} ]
    \\&=
    \EE [ I_{1}(\langle f , u \rangle_{\clh} )
    I_{1}(\langle f, v \rangle_{\clh} ) ]
    =
    \langle 
    \langle f, u \rangle_{\clh}, 
    \langle f , v \rangle_{\clh}
    \rangle_{\mfh}.
\end{aligned}
\end{equation}
The conclusion follows by combining \eqref{eq1-cor1}, \eqref{eq2-cor1}, and noticing that both $\clt_{Z_1}$ and $\clt_{Z_2}$ are deterministic.
\end{proof}

\subsection{The case of a fixed chaos expansion}

For the following theorem, we use the notion of an $\clh$-valued multiple Wiener integral. For a rigorous construction, we refer to Chapter 2 in \cite{nourdin2012normal}, and \cite{dukZou24}.

We say that $\{F_n\}_{n\ge 1}$ is a sequence of $\clh$-valued multiple integrals living in a fixed chaos, if there is a $p \ge 1$ such that $\{F_n\}$ admit the expansion
\begin{equation} \label{eq:seq-multiple-fixed-chaos}
    F_n = \cli_{p}(f_{n,p}), \quad  f_{n,p} \in \mfh^{\odot p} \otimes \clh, n\ge 1.
\end{equation}

\begin{theorem}[Infinite-dimensional Fourth-Moment Theorem -- Fixed Chaos]
\label{thm:abstract-fourth-moment-theorem}
Let $Z$ be a centered, Gaussian random variable on $\clh$ with covariance operator $\clt_Z$. Let $\{ F_n \}_{n\ge 1}$ be a sequence of random variables living in a fixed chaos, with expansion given in \eqref{eq:seq-multiple-fixed-chaos} for some $p \geq 2$, and respective covariance operators $\clt_{F_n}$. Suppose $\| \clt_{F_n} - \clt_Z\|_{\mathcal{S}_1(\clh)} \to 0$. Then, as $n \to \infty$, the following statements are equivalent
    \begin{enumerate}[label=(\roman*)]
    \item $F_n \xrightarrow{d} Z$.
    \label{item1:thm:abstract-fourth-moment-theorem}
    \item $\EE \left[ \langle F_n, e_i \rangle^4_{\clh}\right] \to \EE \left[ \langle Z, e_i \rangle^4_{\clh}\right]$ for all orthonormal bases $\{e_i\}_{i \in \NN}$ of $\clh$.
    \label{item2-all:thm:abstract-fourth-moment-theorem}
    \item The $4$-th weak moments of $\{F_n\}$ converge to the $4$-th weak moments of $Z$, i.e., 
    \begin{equation}
        \EE ( x^*_1(F_n)\dots x^*_4(F_n)) \to \EE ( x^*_1(Z)\dots x^*_4(Z)), 
    \end{equation}
    for all $x_1^*,\dots,x_4^* \in \clh^*$.
    \label{item3:thm:abstract-fourth-moment-theorem}
    \item $\EE \left[ \langle F_n, e_i \rangle^4_{\clh}\right] \to \EE \left[ \langle Z, e_i \rangle^4_{\clh}\right]$ for some orthonormal basis $\{e_i\}_{i \in \NN}$ of $\clh$.
    \label{item2:thm:abstract-fourth-moment-theorem}
    \item For some orthonormal basis $\{e_i\}_{i \in \NN}$, and with $f_{n,p,i} = \langle f_{n,p}, e_i \rangle_{\clh}$, it holds that for all $i\in \NN$, and $r=1,\dots,p-1$, as $n \to \infty$, \label{item5:thm:abstract-fourth-moment-theorem-contr}
    \[
    \|  f_{n,p,i}  \otimes_r  f_{n,p,i}  \|_{\mfh^{\otimes (2p - 2r)}} \to 0.
    \]
    \end{enumerate}
\end{theorem}

\begin{proof}
The equivalence of \ref{item2-all:thm:abstract-fourth-moment-theorem} and \ref{item3:thm:abstract-fourth-moment-theorem} is proved in Lemma \ref{lemma:equiv-weak-moment-all-bases} below. That \ref{item2-all:thm:abstract-fourth-moment-theorem} implies \ref{item2:thm:abstract-fourth-moment-theorem} is immediate. So it remains to show that \ref{item1:thm:abstract-fourth-moment-theorem} implies \ref{item2-all:thm:abstract-fourth-moment-theorem}, that \ref{item2:thm:abstract-fourth-moment-theorem} implies \ref{item5:thm:abstract-fourth-moment-theorem-contr}, and that \ref{item5:thm:abstract-fourth-moment-theorem-contr} implies \ref{item1:thm:abstract-fourth-moment-theorem}.

\textit{Proof of \ref{item1:thm:abstract-fourth-moment-theorem} implies \ref{item2-all:thm:abstract-fourth-moment-theorem}:}
Fix an orthonormal basis $\{e_i\}_{i \in \NN}$ and note that, since the map $x \mapsto \langle x, e_i \rangle_{\clh}, \; x \in \clh$, is continuous, \ref{item1:thm:abstract-fourth-moment-theorem} and the continuous mapping theorem imply that $\langle F_n, e_i \rangle_{\clh} \xrightarrow{d} \langle Z, e_i \rangle_{\clh}$ for all $i \geq 1$. 
Moreover, $\| \clt_{F_n} - \clt_{Z} \|_{\mathcal{S}_1(\clh)} \to 0$ implies that $\EE \|F_n\|^2_{\clh} \to \EE \|Z\|^2_{\clh}$, which in turn says that $\sup_{n\ge 1} \EE \|F_n\|^2_{\clh} < \infty$. 
By hypercontractivity of the Wiener chaos, for all $k \ge 2$, there is a constant $c(k)>0$ such that $\EE \left[ \langle F_n, e_i \rangle_{\clh}^k \right] \leq c(k) \left( \EE \left[ \langle F_n, e_i \rangle_{\clh}^2\right] \right)^{k/2}$; see Theorem 2.7.2 in \cite{nourdin2012normal}.
Then, with $F_{n,i} \doteq \langle F_n, e_i \rangle_{\clh}$, for arbitrary $\varepsilon > 0$, there is a constant $c>0$, so that, by hypercontractivity, 
\begin{equation*}
    \EE F_{n,i}^{4+\varepsilon} \le c (\EE F_{n,i}^2)^{\frac{4+\varepsilon}{2}} \le c \left(\sum_{i=1}^{\infty} \EE F_{n,i}^2\right)^{\frac{4+\varepsilon}{2}}  \le c \left(\sup_{n\ge 1} \EE \|F_n\|^2 _{\clh}\right)^{\frac{4+\varepsilon}{2}}  < \infty,
\end{equation*}
which implies that $\{F_{n,i}^4\}_{n\ge 1}$ is uniformly integrable. Then, Theorem 3.5 in \cite{billingsley1999convergence} gives convergence of the moments, namely that $\EE \left[ \langle F_n, e_i \rangle_{\clh}^4 \right] \to \EE \left[ \langle Z, e_i \rangle_{\clh}^4 \right]$, for all $i \ge 1$.

\textit{Proof of \ref{item2:thm:abstract-fourth-moment-theorem} implies \ref{item5:thm:abstract-fourth-moment-theorem-contr}:} Since $\langle F_n, e_i \rangle_{\clh}$ is an $\RR$-valued random variable with $F_{n,i} = I_p(f_{n,p,i})$ (cf. Lemma 2.2(i) in \cite{vidotto2025functional}), equation (5.2.6) in \cite{nourdin2012normal} implies that 
\begin{equation*}
    \EE F_{n,i}^4 - 3(\EE F_{n,i}^2)^2 =  \sum_{r=1}^{p-1}  (p!)^2 \binom{p}{r}^2 \left( \| f_{n,p,i}  \otimes_r f_{n,p,i} \|^2_{\mfh^{\otimes (2p-2r)}} + \binom{2p-2r}{p-r} \| f_{n,p,i}  \widetilde \otimes_r f_{n,p,i} \|^2_{\mfh^{\otimes (2p-2r)}} \right) \to 0,
\end{equation*}
for all $i \in \NN$, where the convergence holds by our assumption. The notation $\widetilde \otimes$ here stands for the symmetric tensorization; we refer to Appendix B of \cite{nourdin2012normal} for a careful treatment.

\textit{Proof of \ref{item5:thm:abstract-fourth-moment-theorem-contr} implies \ref{item1:thm:abstract-fourth-moment-theorem}:} 
To prove that \ref{item5:thm:abstract-fourth-moment-theorem-contr} implies \ref{item1:thm:abstract-fourth-moment-theorem}, we employ Theorem \ref{thm:1} (since $F_n \in \mathbb{D}^{1,2}$ trivially) so that
\begin{equation} \label{eq:applyTH31}
d_2(F_n,Z)
 \le \tfrac12 \EE\|\Gamma(F_n, -L^{-1} F_n)-\clt_Z\|_{\mathcal{S}_1(\clh)},
\qquad
\Gamma(F,G) \doteq \langle D_M F, D_M G\rangle_{\mfh}.
\end{equation}

Note that, since $F_n$ is in the $p$-th Wiener chaos, by Lemma \ref{le:analogue-Th-2.6} \ref{item2:lem:Gamma-covariance} and \ref{item1:lem:Gamma-covariance} below, we have 
\begin{equation} \label{eq:positivity-exp}
 \Gamma(F_n, -L^{-1} F_n) \geq 0,
 \hspace{0.2cm}
 \mathbb{E}\big[ \Gamma(F_n, -L^{-1} F_n)\big] = \mathcal{T}_{F_n},
\end{equation}
that is, $\Gamma(F_n, -L^{-1} F_n)$ is a nonnegative (self-adjoint) trace-class operator on $\mathcal H$.

By triangle inequality, the trace-class norm in \eqref{eq:applyTH31} can be further bounded by
\begin{equation} \label{eq:two-summands-trace-cov}
\EE\| \Gamma(F_n, -L^{-1} F_n)-\clt_Z\|_{\mathcal{S}_1(\clh)}
 \le \EE\| \Gamma(F_n, -L^{-1} F_n)-\clt_{F_n}\|_{\mathcal{S}_1(\clh)}  + \|\clt_{F_n}-\clt_Z\|_{\mathcal{S}_1(\clh)}.
\end{equation}
The second summand in \eqref{eq:two-summands-trace-cov} vanishes, as $n\to \infty$, by assumption. 

For the first summand of \eqref{eq:two-summands-trace-cov}, fix the orthonormal basis $\{e_i\}_{i \in \NN}$ of $\clh$ along which the convergences in \ref{item5:thm:abstract-fourth-moment-theorem-contr} hold, and let $P_m$ be the orthogonal projection operator onto $E_m \doteq \mathrm{span}\{e_1,\dots,e_m\}$. To be more precise, define the following quantities
\begin{equation} \label{def:pmqmgn}
  P_m : \clh \to \clh \; \text{with } P_m(x) \doteq \sum_{i=1}^m \langle x , e_i \rangle_{\clh} e_i, \;  \quad Q_m \doteq \Id_\clh-P_m, \quad  G_n \doteq  \Gamma(F_n, -L^{-1} F_n)-\clt_{F_n}.  
\end{equation}
Then, for fixed $m\in\NN$, 
\begin{equation} \label{eq:summands3-trace}
\EE \| G_n \|_{\mathcal{S}_1(\clh)} \leq \EE \| P_m G_n P_m \|_{\mathcal{S}_1(\clh)} +  2\EE \| Q_m G_n P_m \|_{\mathcal{S}_1(\clh)} +  \EE \| Q_m G_n Q_m \|_{\mathcal{S}_1(\clh)},
\end{equation}
where we used that the trace-class norm is $^*$-invariant (i.e., $\| A \|_{\mathcal{S}_1(\clh)} = \| A^* \|_{\mathcal{S}_1(\clh)}$) by Theorem 18.11 (f) in \cite{conway2000course}. 
We consider the three summands on the right-hand side of \eqref{eq:summands3-trace} separately.

For the first summand in \eqref{eq:summands3-trace}, note that since $P_m$ is the orthogonal projection onto 
$E_m$, the range of the operator 
$A \doteq P_m G_n P_m$ is a subset of $E_m$ and $A|_{E_m^\perp} = 0$, hence 
$\mathrm{rank}(A) \le m$. Recall the definition of the Schatten $p$-norm in terms of singular values from \eqref{def:Schatten-p-norm-singular} with $p=1$. 
Let $s_1(A) \ge \dots \ge s_r(A) > 0$ denote the nonzero singular values of $A$ with $r = \mathrm{rank}(A) \le m$.
Then,
\begin{equation}
\|A\|_{\mathcal{S}_1(\clh)} = \sum_{k=1}^{r} s_k(A), 
\qquad
\|A\|_{\mathcal{S}_2(\clh)} = \Big( \sum_{k=1}^{r} s_k(A)^2 \Big)^{1/2}.
\end{equation}
By Cauchy-Schwarz in $\RR^r$,
\begin{equation} \label{eq:Trace-HS-ineq}
\|A\|_{\mathcal{S}_1(\clh)} 
\le \sqrt{r}   \|A\|_{\mathcal{S}_2(\clh)}.
\end{equation}
Then, relative to the orthonormal basis $\{e_1,\dots,e_m\}$ of $E_m$ and by \eqref{eq:Trace-HS-ineq},
\begin{equation} \label{eq:finite-term}
\EE \| P_m G_n P_m \|_{\mathcal{S}_1(\clh)} 
\leq
\sqrt{m}
\sqrt{\EE \| P_m G_n P_m \|^2_{\mathcal{S}_2(\clh)}}
=
\sqrt{m} \sqrt{
\sum_{i,j=1}^m \EE \langle G_n e_i, e_j \rangle^2_{\clh} }.
\end{equation}
Note that, for these estimates, $G_n$ needs not be a positive semidefinite operator. 

Moreover, by denoting $F_{n,i} \doteq \langle F_n, e_i \rangle_{\clh}$ and $f_{n,i} \doteq \langle f_n, e_i \rangle_{\clh}$, the following bound holds
\begin{align}
\hspace{-0.3cm}
    \sum_{i,j=1}^m \EE \langle G_n e_i, e_j \rangle^2_{\clh} &=  \sum_{i,j=1}^m \EE \left( \langle  \Gamma(F_n, -L^{-1} F_n) e_i, e_j \rangle_{\clh} - \langle \clt_{F_n} e_i, e_j \rangle_{\clh}   \right)^2 \\
    &\le  \sum_{i,j=1}^m \left( \langle \clt_{F_n} e_i, e_j \rangle_{\clh} - \EE (F_{n,i} F_{n,j} )  \right)^2 \\
    &\qquad + \frac{1}{2} \sum_{i,j=1}^m \sum_{r =1}^{p-1} c_p^2(r) \left(  \| f_{n,p,i} \otimes_{p-r} f_{n,p,i} \|_{\mfh^{\otimes 2r}}^2 + \| f_{n,p,j} \otimes_{p-r} f_{n,p,j} \|_{\mfh^{\otimes 2r}}^2 \right)  \label{eq:summands4-trace-10}  \\
    &= \frac{1}{2} \sum_{i,j=1}^m \sum_{r =1}^{p-1} c_p^2(r) \left(  \| f_{n,p,i} \otimes_{p-r} f_{n,p,i} \|_{\mfh^{\otimes 2r}}^2 + \| f_{n,p,j} \otimes_{p-r} f_{n,p,j} \|_{\mfh^{\otimes 2r}}^2 \right),
    \label{eq:summands4-trace-1} 
\end{align}
where $c_p(r)$ are suitable constants given in \eqref{eq:remainder-terms-def-quant-3} below. Here \eqref{eq:summands4-trace-10} follows from (6.2.1) of Lemma 6.2.1 in \cite{nourdin2012normal}, with, in the notation used therein, $\alpha =  \langle \clt_{F_n} e_i, e_j \rangle_{\clh}$, $p = q$, and $f = f_{n,p,i}, g = f_{n,p,j}$. 

By assumption, we have $\| f_{n,p,i} \otimes_{r} f_{n,p,i} \|_{\mfh^{\otimes (2p-2r)}} \to 0$, for all $i \in \NN$ and $r=1,\dots,p-1$, as $n \to \infty$.
Then, by combining \eqref{eq:finite-term} and \eqref{eq:summands4-trace-1}, we get
\begin{equation} \label{eq:conv-Pn-Pn}
    \EE \| P_m G_n P_m \|_{\mathcal{S}_1(\clh)} \to 0,
\end{equation}
as $n\to \infty$, for all fixed $m \in \NN$.

Note that, since $F_n$ is in a fixed chaos, $\Gamma$ is a positive semidefinite operator by \eqref{eq:positivity-exp}. Therefore, for the third term of \eqref{eq:summands3-trace},
\begin{equation}
\begin{aligned}
    \EE \| Q_m G_n Q_m \|_{\mathcal{S}_1(\clh)}
    &\leq
    \EE \| Q_m  \Gamma(F_n, -L^{-1} F_n) Q_m \|_{\mathcal{S}_1(\clh)} +  \| Q_m \EE  \Gamma(F_n, -L^{-1} F_n) Q_m \|_{\mathcal{S}_1(\clh)}
    \\&=
    2 \| Q_m \EE  \Gamma(F_n, -L^{-1} F_n) Q_m \|_{\mathcal{S}_1(\clh)},
\end{aligned}
\end{equation}
where the equality follows, since by \eqref{eq:positivity-exp}, i.e., the positivity of $\Gamma(F_n, -L^{-1} F_n)$, and Tonelli's theorem
\begin{equation} \label{eq:first316-1}
\begin{split}
\EE \| Q_m  \Gamma(F_n, -L^{-1} F_n) Q_m \|_{\mathcal{S}_1(\clh)}
&=
\sum_{i=1}^{\infty} \langle Q_m \EE  \Gamma(F_n, -L^{-1} F_n) Q_m e_i, e_i \rangle_{\clh}\\
&=
\| Q_m \EE  \Gamma(F_n, -L^{-1} F_n) Q_m \|_{\mathcal{S}_1(\clh)}.
\end{split}
\end{equation}
Then, for fixed $m$ and as $n \to \infty$,
\begin{equation} \label{eq:first316-2}
\begin{split}
   \lim_{n \to \infty} \| Q_m \EE  \Gamma(F_n, -L^{-1} F_n) Q_m \|_{\mathcal{S}_1(\clh)}
&= \lim_{n \to \infty} 
\sum_{i=1}^{\infty} \langle Q_m \EE  \Gamma(F_n, -L^{-1} F_n) Q_m e_i, e_i \rangle_{\clh} \\
&= \lim_{n \to \infty} 
\sum_{i = m+1}^\infty \langle \clt_{F_n} e_i,e_i\rangle_{\clh}
=
\sum_{i = m+1}^\infty \langle \clt_{Z} e_i,e_i\rangle_{\clh}, 
\end{split}
\end{equation}
where the convergence is due to 
\begin{equation} \label{eq:tail-convergence-trace}
    \| Q_m \big( \clt_{F_n} - \clt_Z\big) Q_m \|_{\mathcal{S}_1(\clh)}
    \leq
    \| Q_m \|^2_{\op} \|\clt_{F_n}-\clt_Z\|_{\mathcal{S}_1(\clh)}
    =
    \|\clt_{F_n}-\clt_Z\|_{\mathcal{S}_1(\clh)} \to 0,
\end{equation}
as $n \to \infty$; the inequality follows by Theorem 18.11 (g) in \cite{conway2000course} and the second relation is due to $\| Q_m \|_{\op} = 1$.
Finally, the tail in \eqref{eq:first316-2} vanishes as $m\to\infty$ since $\sum_{i=1}^{\infty} \langle \clt_{Z} e_i,e_i\rangle_{\clh} = \| \clt_Z \|_{\mathcal{S}_1(\clh)} = \EE \| Z \|_{\clh}^2 < \infty$.

For the second summand in \eqref{eq:summands3-trace}, i.e., the cross terms, we get
\begin{equation} \label{eq:start-cross}
    \EE \| Q_m G_n P_m \|_{\mathcal{S}_1(\clh)}
    \leq 
    \EE \| Q_m  \Gamma(F_n, -L^{-1} F_n) P_m \|_{\mathcal{S}_1(\clh)} + \| Q_m \clt_{F_n} P_m \|_{\mathcal{S}_1(\clh)}.
\end{equation}
For the second summand on the right hand side of \eqref{eq:start-cross}, with explanations given below,
\begin{align}
\|P_m \clt_{F_n} Q_m\|_{\mathcal{S}_1(\clh)}
&=
\| (P_m \clt_{F_n}^{1/2})(\clt_{F_n}^{1/2} Q_m)\|_{\mathcal{S}_1(\clh)}
\nonumber
\\&\le 
\|P_m \clt_{F_n}^{1/2}\|_{\mathcal{S}_2(\clh)} \|\clt_{F_n}^{1/2} Q_m\|_{\mathcal{S}_2(\clh)}
\label{al:norm-ine-TR-HS-1}
\\&= 
\| | P_m \clt_{F_n}^{1/2} |^2 \|_{\mathcal{S}_1(\clh)}^{\frac{1}{2}} 
\| | \clt_{F_n}^{1/2} Q_m |^2 \|_{\mathcal{S}_1(\clh)}^{\frac{1}{2}},
\label{al:norm-ine-TR-HS-2}
\end{align}
where we refer to the proof of Proposition 18.8 ((d) implies (a)) in \cite{conway2000course} for \eqref{al:norm-ine-TR-HS-1}
and \eqref{al:norm-ine-TR-HS-2} follows by Definition 18.4 of \cite{conway2000course}.

Furthermore, by assumption and \eqref{eq:tail-convergence-trace} (one can argue similarly for the truncated part since $\| P_m \|_{\op}=1$), 
\begin{align}
\| | P_m \clt_{F_n}^{1/2} |^2 \|_{\mathcal{S}_1(\clh)} &=
\sum_{i=1}^{m}\langle \clt_{F_n} e_i,e_i\rangle_{\clh} \to \sum_{i=1}^{m}\langle \clt_{Z} e_i,e_i\rangle_{\clh},
\\
\| | \clt_{F_n}^{1/2} Q_m |^2 \|_{\mathcal{S}_1(\clh)}
&=
\sum_{i = m+1}^\infty \langle \clt_{F_n} e_i,e_i\rangle_{\clh} 
\to 
\sum_{i = m+1}^\infty \langle \clt_{Z} e_i,e_i\rangle_{\clh},
\end{align}
so together with \eqref{al:norm-ine-TR-HS-2}, we get, for fixed $m$, 
\begin{equation}\label{eq:cross-block}
\begin{split}
\lim_{n\to\infty}\|P_m \clt_{F_n} Q_m\|_{\mathcal{S}_1(\clh)}
\ &\le\
\Big(\sum_{i=1}^{m}\langle \clt_{Z} e_i,e_i\rangle_{\clh}\Big)^{1/2}\Big(\sum_{i = m+1}^\infty \langle \clt_{Z} e_i,e_i\rangle_{\clh}\Big)^{1/2} \\
& \le \| \clt_Z \|_{\mathcal{S}_1(\clh)}^{1/2} \Big(\sum_{i = m+1}^\infty \langle \clt_{Z} e_i,e_i\rangle_{\clh}\Big)^{1/2},
\end{split}
\end{equation}
which converges to $0$ as $m \to \infty$.
Similarly, for the first summand in \eqref{eq:start-cross}, and since $ \Gamma(F_n, -L^{-1} F_n)$ is nonnegative by \eqref{eq:positivity-exp}, we get
\begin{align}
&
\EE \| Q_m  \Gamma(F_n, -L^{-1} F_n) P_m \|_{\mathcal{S}_1(\clh)}
\\&=
\EE \| (P_m  \Gamma(F_n, -L^{-1} F_n)^{1/2})( \Gamma(F_n, -L^{-1} F_n)^{1/2} Q_m)\|_{\mathcal{S}_1(\clh)}
\nonumber
\\&  \le \EE
\|P_m  \Gamma(F_n, -L^{-1} F_n)^{1/2}\|_{\mathcal{S}_2(\clh)} \| \Gamma(F_n, -L^{-1} F_n)^{1/2} Q_m\|_{\mathcal{S}_2(\clh)}
\label{al:norm-ine-TR-HS-1-gam}
\\&= 
\EE \| | P_m  \Gamma(F_n, -L^{-1} F_n)^{1/2} |^2 \|_{\mathcal{S}_1(\clh)}^{\frac{1}{2}}  
\| |  \Gamma(F_n, -L^{-1} F_n)^{1/2} Q_m |^2 \|_{\mathcal{S}_1(\clh)}^{\frac{1}{2}} 
\nonumber
\\&\le \left( \EE \| | P_m  \Gamma(F_n, -L^{-1} F_n)^{1/2} |^2 \|_{\mathcal{S}_1(\clh)} \right) ^{\frac{1}{2}}   \left(  \EE
\| |  \Gamma(F_n, -L^{-1} F_n)^{1/2} Q_m |^2 \|_{\mathcal{S}_1(\clh)}\right)^{\frac{1}{2}},
\label{al:norm-ine-TR-HS-3-gam}
\end{align}
where the last line follows from Cauchy-Schwarz. Furthermore,
\begin{align}
\EE \| | P_m  \Gamma(F_n, -L^{-1} F_n)^{1/2} |^2 \|_{\mathcal{S}_1(\clh)} &=
\sum_{i=1}^{m}\langle \clt_{F_n} e_i,e_i\rangle_{\clh} \to \sum_{i=1}^{m}\langle \clt_{Z} e_i,e_i\rangle_{\clh},
\\
\EE \| |  \Gamma(F_n, -L^{-1} F_n)^{1/2} Q_m |^2 \|_{\mathcal{S}_1(\clh)}
&=
\sum_{i = m+1}^\infty \langle \clt_{F_n} e_i,e_i\rangle_{\clh} 
\to 
\sum_{i = m+1}^\infty \langle \clt_{Z} e_i,e_i\rangle_{\clh},
\end{align}
so together with \eqref{al:norm-ine-TR-HS-3-gam}, we get, for fixed $m$, 
\begin{equation}\label{eq:cross-block-gam}
\limsup_{n\to\infty} \EE \| Q_m  \Gamma(F_n, -L^{-1} F_n) P_m \|_{\mathcal{S}_1(\clh)}
\ \le\
\| \clt_{Z} \|_{\mathcal{S}_1(\clh)}^{1/2}\Big(\sum_{i = m+1}^\infty\langle \clt_{Z} e_i,e_i\rangle_{\clh}\Big)^{1/2}. 
\end{equation}
Therefore, combining \eqref{eq:applyTH31}, \eqref{eq:two-summands-trace-cov} \eqref{eq:summands3-trace}, \eqref{eq:conv-Pn-Pn}, \eqref{eq:first316-2}, \eqref{eq:start-cross}, \eqref{eq:cross-block} and \eqref{eq:cross-block-gam}, letting $n\to\infty$ and then $m \to \infty$, we obtain $d_2(F_n,Z)\to 0$, i.e., $F_n \xrightarrow{d} Z$, as $n \to \infty$.
\end{proof}

\begin{remark} \label{rmk1}
     \begin{enumerate}
\item 
    The result of Theorem \ref{thm:abstract-fourth-moment-theorem} is consistent with the multivariate Fourth-Moment Theorem \cite[Theorem 6.2.3]{nourdin2012normal} in the following way: The fourth moment condition in the infinite-dimensional setting is only needed for a fixed orthonormal basis. Likewise, in the multivariate setting, the fourth-moment convergence only needs to be satisfied for each component of the vector sequence $\{F_n\}$.

    However, on the level of second moments, we require convergence of the covariance operators in trace-class norm. In the finite-dimensional setting the variances and cross-covariances need to converge to those of the limiting variable. While in finite dimensions, trace-class convergence and convergence of variances and cross-covariances are equivalent, the behavior in infinite dimensions is more nuanced. In particular, convergence of cross-covariances even along all orthonormal bases does not suffice to imply convergence of the covariance operators in trace-class norm. 

    We emphasize that this comparison is based on the multivariate Fourth-Moment Theorem (Theorem 6.2.3 in \cite{nourdin2012normal}) assuming that the chaos order is the same across dimensions.
        \item 
        Note that condition \ref{item2-all:thm:abstract-fourth-moment-theorem} of Theorem \ref{thm:abstract-fourth-moment-theorem} can alternatively be written as 
        $\EE \left[ \langle F_n, u \rangle^4_{\clh}\right] \to \EE \left[ \langle Z, u \rangle^4_{\clh}\right]$, as $n \to \infty$, for every unit vector $u \in \clh$.
        
\item     The condition $\|\clt_{F_n} - \clt_Z \|_{\mathcal{S}_1(\clh)} \to 0$ required herein is stronger than the assumption $\tr(\clt_{F_n} - \clt_{Z}) \to 0$ required in \cite{Bou20}, or its equivalent formulation $\EE \|F_n\|^2_{\clh} \to \EE \|Z\|^2_{\clh}$. This stronger assumption is needed due to the identifiability issue illustrated through Example \ref{exp1}.

\item \label{rmk1-Gaussian} Note that for centered Gaussian measures $\{\mu_n\}_{n\ge 1}$ and $\mu$ on $\clh$, the weak convergence $\mu_n \to \mu$ in the topology of $\clh$ is equivalent to the convergence in trace-class norm of their covariance operators. This can be seen by combining Example 3.8.15 in \cite{Bogachev1998} or Proposition 1 in \cite{BasBurCamPec25} with 
Theorem 2 in \cite{kubrusly1986convergence}. This is not true for an arbitrary (non Gaussian) sequence.

        \item   The statement of Theorem \ref{thm:abstract-fourth-moment-theorem} leads naturally to postulating the following question: If the convergence $\| \clt_{F_n} - \clt_Z\|_{\mathcal{S}_1(\clh)} \to 0$ is replaced by the weaker convergence in some Schatten $\tau$-norm, \eqref{def:Schatten-p-norm}, i.e.,
        \begin{equation}
            \| \clt_{F_n} - \clt_Z\|_{\cls_\tau(\clh)} \to 0, \quad \text{for some }\tau \in (1,\infty),
        \end{equation}
        could we still obtain an equivalence of items \ref{item1:thm:abstract-fourth-moment-theorem}--\ref{item5:thm:abstract-fourth-moment-theorem-contr}? The answer to this is negative, as the following counterexample shows.

        Fix some $\tau > 1$. We construct a Gaussian sequence $\{F_n\}_{n\ge 1}$ such that, $\| \clt_{F_n} - \clt_Z\|_{\cls_\tau(\clh)} \to 0$, $\| \clt_{F_n} - \clt_Z\|_{\mathcal{S}_1(\clh)} = 1$ for all $n\ge 1$, item \ref{item2:thm:abstract-fourth-moment-theorem} is satisfied, but $F_n$ does not converge in distribution to a Gaussian element $Z$.

        Fix $\clh = \ell^2(\NN)$, select a basis $\{e_k \}_{k \in \NN}$ and choose a sequence $\{\lambda_k \}_{k \in \NN}$ such that $\lambda_k > 0$ for all $k \in \NN$ and $\sum_{i=1}^\infty \lambda_i < \infty$. Define the covariance operator $\langle \clt_Z e_i , e_j \rangle_{\clh} = \lambda_i \delta_{i,j},$ and denote the centered Gaussian random variable on $\ell^2$ with covariance operator $\clt_Z$ by $Z$. 

        Let $\alpha_k \doteq k^{-\gamma}$, where $\gamma \in \left( 1 - (\tau-1)/\tau, 1\right)$, $S_n \doteq \sum_{k=1}^n \alpha_k$, and $s_{n,k} \doteq \frac{\alpha_k}{S_n},$ if $k \le n$, and $0$ otherwise. Define
        \begin{equation}
            Y_n \doteq \sum_{k=1}^{n} \xi_k \sqrt{s_{n,k}} e_k,
        \end{equation}
        where $\{\xi_i\}_{i \in \NN}$ are $\RR$-valued, i.i.d. standard normal random variables that are independent of $Z$. Then, we have that $\langle \clt_{Y_n} e_i, e_j \rangle_{\clh} = \delta_{i,j} s_{n,i} \1_{\{i \le n\}}$. Let $F_n = Z + Y_n$, which is a centered Gaussian random variable with covariance operator $\clt_{F_n} = \clt_Z + \clt_{Y_n}$. In particular, for all $n$, $F_n$ lives in the Wiener chaos of order one.

        Note that, by construction,
        \begin{equation}
        \begin{split}
            \| \clt_{F_n} - \clt_Z \|_{\cls_1 (\clh)} &= \| \clt_{Y_n}  \|_{\cls_1 (\clh)} = \sum_{k=1}^n s_{n,k} = 1, \\
            \| \clt_{F_n} - \clt_Z \|^\tau_{\cls_\tau (\clh)} &= \| \clt_{Y_n}  \|^p_{\cls_\tau (\clh)} = \sum_{k=1}^n s^\tau_{n,k} = \frac{\sum_{k=1}^n k^{-\tau \gamma}}{ \left( \sum_{k=1}^n \alpha_k \right)^\tau} \to 0,
        \end{split}
        \end{equation}
        where the convergence in the second line follows since $-\tau \gamma < -1$ for $\tau \in (1,\infty)$, hence the numerator converges to a finite value while the denominator diverges. Moreover, 
        \[
        \EE \| F_n \|^2_{\clh} = \EE \sum_{i=1}^\infty \langle Z, e_i \rangle_\clh^2 + \EE \sum_{i=1}^\infty \langle Y_n, e_i \rangle_\clh^2 = \EE \|Z\|_\clh^2 + \sum_{k=1}^n s_{n,k} = \EE \|Z\|_\clh^2  +1.
        \]
        By a well-known result for the convergence in distribution of Gaussian measures (see Remark \ref{rmk1} \eqref{rmk1-Gaussian}), this implies that $F_n$ does not converge to $Z$ in distribution. Finally, we see that, for all $i \in \NN$,
        \begin{equation}
        \begin{split}
        \EE \langle F_n, e_i \rangle^4_\clh &= \EE \langle Z, e_i \rangle^4_\clh + 6 \EE \langle Z, e_i \rangle^2_\clh \EE \langle Y_n, e_i \rangle^2_\clh + \EE \langle Y_n, e_i \rangle^4_\clh \\
        &= 3\lambda_i^2 + 6 \lambda_i s_{n,i} + 3s_{n,i}^2 \to \EE \langle Z, e_i \rangle^4_\clh,
        \end{split}
        \end{equation}
        since $s_{n,i} \to 0$ as $n \to \infty$. 
        This completes the claim. Finally, we remark that analogous constructions can be done for a sequence $\{F_n\}_{n\ge 1}$ that belongs to a fixed Wiener chaos of order different than one, by choosing the sequence $\{ \xi_n \}_{n\ge 1}$ accordingly.
    \end{enumerate}
\end{remark}

\subsection{The case of a finite chaos expansion}

The following theorem presents a version of the Fourth-Moment Theorem for $\clh$-valued random variables admitting the finite chaos expansion given by 
\begin{equation} \label{eq:finite-chaos-expansions}
F_n = \sum_{r=1}^N \cli_{r}(f_{n,r}), \quad f_{n,r} \in \mfh^{\odot r} \otimes \clh, \; N \in \NN.
\end{equation}

The result below can be proven under an equivalence of vanishing contractions, converging fourth moments, and convergence in distribution, in analogy with the statement of Theorem \ref{thm:abstract-fourth-moment-theorem}. For simplicity, we only state the most important implication based on vanishing contractions.

\begin{theorem}[Infinite-dimensional Fourth-Moment Theorem -- Finite Chaos Expansion] \label{cor:abstract-fourth-moment-theorem-finite-chaos}
Let $\{ F_n \}_{n\ge 1}$ be an $\clh$-valued sequence admitting the chaos expansion given in \eqref{eq:finite-chaos-expansions} for some $N \ge 2$, with respective covariance operators $\clt_{F_n}$. Let $Z$ be a centered Gaussian, $\clh$-valued random variable with covariance operator $\clt_{Z}$. Assume that $\| \clt_{F_n} - \clt_Z \|_{\mathcal S_1(\clh)} \to 0$. Fix an orthonormal basis $\{e_i\}_{i \in \NN}$ of $\clh$ and let $f_{n,r,i} \doteq \langle f_{n,r},e_i \rangle_\clh$. Suppose that for $i\in \NN$ and $r=2,\dots,N$,
\begin{equation} \label{eq:contr-finite-chaos-assum}
  \|  f_{n,r,i}  \otimes_m  f_{n,r,i}  \|_{\mfh^{\otimes (2r - 2m)}} \to 0
\end{equation}
for all $m=1,\dots,r-1$ as $n \to \infty$.
Then $F_n \xrightarrow{d} Z$.
\end{theorem}

\begin{proof}
Recall $F_n$ from \eqref{eq:finite-chaos-expansions} and the notation
\begin{equation}
    \Gamma(F,G) = \langle D_M F, D_M G \rangle_{\mathfrak{H}}, \quad F,G \in \mathbb{D}^{1,2}.
\end{equation}
By Theorem \ref{thm:1} since $F_n \in \mathbb{D}^{1,2}$ and by the triangle inequality, we have 
\begin{equation}\label{eq:two-terms}
\begin{aligned}
d_2(F_n,Z) &\le \tfrac12 \EE\|  \Gamma(F_n, -L^{-1} F_n) - \clt_Z \|_{\mathcal S_1(\clh)} \\&\le
\frac{1}{2}\EE\|  \Gamma(F_n, -L^{-1} F_n) - \clt_{F_n} \|_{\mathcal S_1(\clh)} + \frac{1}{2} \| \clt_{F_n} - \clt_Z \|_{\mathcal S_1(\clh)}.
\end{aligned}
\end{equation}
The second term on the right side of \eqref{eq:two-terms} vanishes by assumption. It suffices to show that
\begin{equation}\label{eq:key-vanish}
\EE\|  \Gamma(F_n, -L^{-1} F_n) - \clt_{F_n} \|_{\mathcal S_1(\clh)} \to 0.
\end{equation}
Let $P_m, Q_m, G_n$ be defined as in \eqref{def:pmqmgn}.
By decomposing $G_n$ along $P_m$ and $Q_m$, one has
\begin{equation}\label{eq:3-blocks}
\begin{split}
\EE\|G_n\|_{\mathcal S_1(\clh)}
&\leq\EE\|P_mG_nP_m\|_{\mathcal S_1(\clh)}
+\EE\|P_mG_nQ_m\|_{\mathcal S_1(\clh)}\\
&\qquad+\EE\|Q_mG_nP_m\|_{\mathcal S_1(\clh)}
+\EE\|Q_mG_nQ_m\|_{\mathcal S_1(\clh)}.
\end{split}
\end{equation}
Unlike in the fixed-chaos case, $G_n$ need not be self-adjoint, so we retain the two off-diagonal blocks separately.

For the finite-dimensional block, since $\mathrm{rank}(P_m G_n P_m) \le m$ and from \eqref{eq:Trace-HS-ineq}, we have
\begin{equation}\label{eq:finite-SHS}
\EE\|P_m G_n P_m\|_{\mathcal S_1(\clh)}
\le
\sqrt m \Big( \sum_{i,j=1}^{m} \EE \langle G_n e_i, e_j\rangle_{\clh}^2 \Big)^{1/2}.
\end{equation}
Write $F_{n,i}=\sum_{r=1}^N I_r(f_{n,r,i})$.
Then, by Lemma \ref{le:analogue-Th-2.6} \ref{item1:lem:Gamma-covariance} below, 
\begin{align}
&
\sum_{i,j=1}^{m}\EE\langle G_n e_i,e_j\rangle_{\clh}^2 
\\&=
\sum_{i,j=1}^{m}
\EE \left(
\Gamma(F_{n,i}, -L^{-1}F_{n,j}) - \EE \Gamma(F_{n,i}, -L^{-1}F_{n,j})
\right)^2
\\&\le
N^2 
\sum_{i,j=1}^{m}
\sum_{r,s=1}^N 
\EE
\Big(\Gamma(I_r(f_{n,r,i}), -L^{-1}I_s(f_{n,s,j}))
- \EE\Gamma(I_r(f_{n,r,i}), -L^{-1}I_s(f_{n,s,j})) \Big)^2
\label{eq:sum-ij}\\
&\le  N^2 
\sum_{i,j=1}^{m} \Bigg( \sum_{r=1}^N \EE
\Big( \Gamma(I_r(f_{n,r,i}), -L^{-1}I_r(f_{n,r,j}))
- \EE[I_r(f_{n,r,i}) I_r(f_{n,r,j})] \Big)^2   \\
&\hspace{1.5cm}  \qquad +
\sum_{\substack{r,s=1 \\ r \neq s}}^N 
\EE
\Big( \Gamma(I_r(f_{n,r,i}), -L^{-1} I_s(f_{n,s,j}) ) \Big)^2   \Bigg) \label{eq:9800}
\\
&\le
N^2 
 \sum_{i,j=1}^m \Bigg(\frac{1}{2} \sum_{r=1}^N \sum_{\chi =1}^{r-1} c_r(\chi)^2 \left(  \| f_{n,r,i} \otimes_{r-\chi} f_{n,r,i} \|_{\mfh^{\otimes 2\chi}}^2 + \| f_{n,r,j} \otimes_{r-\chi} f_{n,r,j} \|_{\mfh^{\otimes 2\chi}}^2 \right)
\\&\hspace{1cm}+ 
\sum_{\substack{r,s=1 \\ r \neq s}}^N
\Big(   c(r,s) \| f_{n,r,i} \|^2_{\mfh^{\otimes r}} \| f_{n,s,j} \otimes_{s-r}  f_{n,s,j} \|_{\mfh^{\otimes 2 r}} \1_{\{r < s\}} \\
&\hspace{2cm} + \frac{r^2}{s^2}c(s,r) \| f_{n,s,j} \|^2_{\mfh^{\otimes s}} \| f_{n,r,i} \otimes_{r-s}  f_{n,r,i} \|_{\mfh^{\otimes 2 s}} \1_{\{s < r\}}  
\\ &\hspace{1cm}+  \sum_{\chi =1}^{r \wedge s-1} c(r,s,\chi) \left(  \| f_{n,r,i} \otimes_{r-\chi}  f_{n,r,i} \|_{\mfh^{\otimes 2\chi}}^2 + \| f_{n,s,j} \otimes_{s-\chi}  f_{n,s,j} \|_{\mfh^{\otimes 2 \chi}}^2 \right) \Big) \Bigg).
\label{eq:9801}
\end{align}
In the estimates above, \eqref{eq:sum-ij} follows by Cauchy–Schwarz inequality. The terms in \eqref{eq:9800} follow by splitting the sum into $r=s$ and $r\neq s$, and recalling that
\[
\EE\Gamma(I_r(f_{n,r,i}),-L^{-1}I_s(f_{n,s,j}))
=\delta_{r,s}\EE[I_r(f_{n,r,i})I_r(f_{n,r,j})].
\]
In the last inequality, we obtain four summands given across the last four lines in \eqref{eq:9801}. The first one is analogous to \eqref{eq:summands4-trace-1} and follows by equation (6.2.1) in \cite{nourdin2012normal} with $\alpha=\EE[I_r(f_{n,r,i})I_r(f_{n,r,j})]$. The second, third, and fourth summands are consequences of (6.2.2) in \cite{nourdin2012normal} with $\alpha = 0$. For the third summand, where $s<r$, the normalization must be kept in the second argument:
\[
\Gamma(I_r(f),-L^{-1}I_s(g))
=\frac rs\,\Gamma(I_s(g),-L^{-1}I_r(f)).
\]
This accounts for the factor $r^2/s^2$ and, likewise, for the ordered coefficient $c(r,s,\chi)$ in the last summand. The constants $c_r(\chi),c(r,s),$ and $c(r,s,\chi)$ are defined in \eqref{eq:remainder-terms-def-quant-3} below.

By the assumption in \eqref{eq:contr-finite-chaos-assum}, all these contractions vanish. Hence for fixed $m$
\begin{equation}\label{eq:finite-vanish}
\sum_{i,j=1}^{m} \EE \langle G_n e_i, e_j\rangle_{\clh}^2 \to 0
\quad\text{and}\quad
\EE\|P_m G_n P_m\|_{\mathcal S_1(\clh)} \to 0.
\end{equation}

For the tail block in \eqref{eq:3-blocks}, note that $\Gamma(F_n,-L^{-1}F_n)$ need not be self-adjoint or nonnegative (since $\Gamma(F_n, -L^{-1} F_n)^* = \Gamma(-L^{-1} F_n, F_n)$). To obtain tractable calculations for its trace-class norm, we recast it as follows
\begin{multline}
    \| \Gamma(F_n, -L^{-1} F_n) \|_{\cls_1(\clh)} = \tr_\clh |\Gamma(F_n, -L^{-1} F_n)| \\
    = \tr_{\clh^2} \left|  \begin{pmatrix}
        0 & 0 \\ 
        0 & \Id_\clh
    \end{pmatrix}
    \begin{pmatrix}
    \Gamma(F_n, F_n) & \Gamma(F_n, -L^{-1}F_n)^* \\
   \Gamma(F_n, -L^{-1}F_n) & \Gamma(-L^{-1}F_n, -L^{-1}F_n)
    \end{pmatrix}
    \begin{pmatrix}
        0 & \Id_\clh \\ 
        0 & 0
    \end{pmatrix}\right|  \label{eq:Fn-matrix-Gamma}.
\end{multline}
For the equality in \eqref{eq:Fn-matrix-Gamma}, we used that 
\[
\tr_{\clh^2} \left| \begin{pmatrix}
   0 & 0 \\
    0 &  \Gamma(F_n, -L^{-1} F_n)
\end{pmatrix} \right| = \tr_{\clh^2}  \begin{pmatrix}
    0 & 0 \\
    0 & \left| \Gamma(F_n, -L^{-1} F_n)  \right|
\end{pmatrix}  = \| \Gamma(F_n, -L^{-1} F_n) \|_{\cls_1(\clh)}.
\]
Consequently, Theorem 18.11 (e) in \cite{conway2000course} and \eqref{eq:Fn-matrix-Gamma} imply
\begin{multline}
     \| \Gamma(F_n, -L^{-1} F_n) \|_{\cls_1(\clh)} = \tr_\clh |\Gamma(F_n, -L^{-1} F_n)| \\\le \norm{ \begin{pmatrix}
        0 & 0 \\ 
        0 & \Id_\clh
    \end{pmatrix} }_{\op} \norm{ \begin{pmatrix}
        0 & \Id_\clh \\ 
        0 & 0
    \end{pmatrix} }_{\op}   \norm{ \begin{pmatrix}
        \Gamma(F_n,  F_n) & \Gamma(F_n, -L^{-1} F_n)^* \\
        \Gamma(F_n, -L^{-1} F_n) & \Gamma(-L^{-1} F_n, -L^{-1} F_n)
    \end{pmatrix} }_{\cls_1(\clh^2)} \\
    \le  \norm{ \begin{pmatrix}
        \Gamma(F_n,  F_n) & \Gamma(-L^{-1} F_n,  F_n) \\
        \Gamma(F_n, -L^{-1} F_n) & \Gamma(-L^{-1} F_n, -L^{-1} F_n)
    \end{pmatrix} }_{\cls_1(\clh^2)}. \label{eq:fn-matrix}
\end{multline}
The key observation is that the operator (in matrix form) we constructed in the last line of \eqref{eq:fn-matrix} is a self-adjoint (cf. with item \ref{item1:Dirichlet-Th-2.6} of Lemma \ref{le:analogue-Th-2.6}) and nonnegative operator, since, for $u_1, u_2 \in \clh$,
\begin{align}
&
\left\langle
\begin{pmatrix}
    \Gamma(F_n, F_n) & \Gamma(-L^{-1}F_n, F_n) \\
   \Gamma(F_n, -L^{-1}F_n) & \Gamma(-L^{-1}F_n, -L^{-1}F_n)
\end{pmatrix}
\begin{pmatrix} u_1 \\ u_2 \end{pmatrix},
\begin{pmatrix} u_1 \\ u_2 \end{pmatrix}
\right\rangle_{\clh^2}
\\
&= \langle \Gamma(F_n, F_n) u_1, u_1 \rangle_{\clh}
+ \langle \Gamma(F_n, -L^{-1}F_n) u_1, u_2 \rangle_{\clh} \\
&\qquad + \langle \Gamma(-L^{-1}F_n, F_n) u_2, u_1 \rangle_{\clh}
+ \langle \Gamma(-L^{-1}F_n, -L^{-1}F_n) u_2, u_2 \rangle_{\clh}
\\
&= \big\langle \langle D_M F_n, u_1 \rangle_{\clh},
            \langle D_M F_n, u_1 \rangle_{\clh} \big\rangle_{\mfh}
+ \big\langle \langle D_M(-L^{-1}F_n), u_2 \rangle_{\clh},
            \langle D_M F_n, u_1 \rangle_{\clh} \big\rangle_{\mfh}
\\
&\qquad + \big\langle \langle D_M F_n, u_1 \rangle_{\clh},
            \langle D_M(-L^{-1}F_n), u_2 \rangle_{\clh} \big\rangle_{\mfh}
+ \big\langle \langle D_M(-L^{-1}F_n), u_2 \rangle_{\clh},
            \langle D_M(-L^{-1}F_n), u_2 \rangle_{\clh} \big\rangle_{\mfh}
\\
&= \big\| \langle D_M F_n, u_1 \rangle_{\clh}
       + \langle D_M(-L^{-1}F_n), u_2 \rangle_{\clh}
   \big\|_{\mfh}^{2} \ge 0.
\end{align}
Hence, since $\Gamma(F,F)$ are self-adjoint and a.s. positive semidefinite operators, by Lemma \ref{le:analogue-Th-2.6} \ref{item1:Dirichlet-Th-2.6}, we have the explicit decomposition
\begin{equation} \label{eq:matrix-norm-fn-gn}
     \norm{ \begin{pmatrix}
        \Gamma(F_n,  F_n) & \Gamma(F_n, -L^{-1} F_n)^* \\
        \Gamma(F_n, -L^{-1} F_n) & \Gamma(-L^{-1} F_n, -L^{-1} F_n)
    \end{pmatrix} }_{\cls_1(\clh^2)} = \tr_\clh(\Gamma(F_n,  F_n)) + \tr_\clh(\Gamma(-L^{-1} F_n, -L^{-1} F_n)).
\end{equation}
Then, after applying triangle inequality, the same calculations leading to \eqref{eq:matrix-norm-fn-gn} give
\begin{align}
    &
    \EE \| Q_m G_n Q_m \|_{\cls_1(\clh)} \le \EE \| Q_m \Gamma(F_n, -L^{-1} F_n )Q_m \|_{\cls_1(\clh)} +  \| Q_m \clt_{F_n} Q_m \|_{\cls_1(\clh)} \\
    &\le \EE  \tr_\clh( Q_m \Gamma(F_n,  F_n) Q_m ) + \EE \tr_\clh( Q_m \Gamma(-L^{-1} F_n, -L^{-1} F_n) Q_m) +   \| Q_m \clt_{F_n} Q_m \|_{\cls_1(\clh)}. \label{eq:321}
\end{align}
We consider the summands in \eqref{eq:321} separately. 
For the first one, we write
\begin{equation} \label{eq:322}
\begin{split}
        \EE  \tr( Q_m \Gamma(F_n,  F_n) Q_m ) 
        &=  \sum_{i= m+1}^\infty \EE \langle \Gamma(F_n,  F_n)e_i, e_i \rangle_\clh \\
        &= \sum_{i= m+1}^\infty  \sum_{r_1, r_2 =1}^N r_1 r_2 \EE [ \langle \langle  \cli_{r_1-1}(f_{n,r_1}), e_i \rangle_\clh , \cli_{r_2-1}(f_{n,r_2}), e_i \rangle_\clh  \rangle_\mfh ] \\
        &= \sum_{i= m+1}^\infty  \sum_{r_1,r_2 =1}^N r_1 r_2 \EE [ \langle I_{r_1-1}(f_{n,r_1,i}) ,I_{r_2-1}(f_{n,r_2,i})  \rangle_\mfh ] \\
        &= \sum_{i= m+1}^\infty  \sum_{r =1}^N  r^2 (r-1)!  \| f_{n,r,i}  \|^2_{\mfh^{\otimes r}} \\
        &\le N \sum_{i= m+1}^\infty  \sum_{r =1}^N   r!  \| f_{n,r,i}  \|^2_{\mfh^{\otimes r}}  \\
        &= N \sum_{i= m+1}^\infty  \langle \EE \Gamma(F_n,  -L^{-1} F_n)e_i, e_i \rangle_\clh  \\
        &= N  \| Q_m \clt_{F_n}  Q_m \|_{\cls_1(\clh)},
\end{split}
\end{equation}
where the fourth equality follows from Theorem 2.7.10 of \cite{nourdin2012normal}; see also the calculation below (6.2.2) therein. Analogous considerations give
\begin{equation} \label{eq:323}
\begin{split}
        \EE  \tr( Q_m \Gamma(-L^{-1} F_n, -L^{-1} F_n) Q_m ) 
        &= \sum_{i= m+1}^\infty  \sum_{r_1, r_2 =1}^N  \EE [ \langle \langle  \cli_{r_1-1}(f_{n,r_1}), e_i \rangle_\clh , \cli_{r_2-1}(f_{n,r_2}), e_i \rangle_\clh  \rangle_\mfh ] \\
        &= \sum_{i= m+1}^\infty  \sum_{r =1}^N   (r-1)!  \| f_{n,r,i}  \|^2_{\mfh^{\otimes r}} \\
        &\le  \sum_{i= m+1}^\infty  \sum_{r =1}^N   r!  \| f_{n,r,i}  \|^2_{\mfh^{\otimes r}}  \\
        &=  \sum_{i= m+1}^\infty  \langle \EE \Gamma(F_n,  -L^{-1} F_n)e_i, e_i \rangle_\clh  \\
        &=   \| Q_m \clt_{F_n}  Q_m \|_{\cls_1(\clh)}.
\end{split}
\end{equation}
Combining \eqref{eq:321}, \eqref{eq:322}, and \eqref{eq:323},
\begin{multline} \label{eq:qm-qm-quantitative}
    \limsup_{n \to \infty} \EE \| Q_m G_n Q_m \|_{\cls_1(\clh)} \le (N+2) \| Q_m \clt_{F_n}  Q_m \|_{\cls_1(\clh)} = \lim_{n \to \infty} (N+2) \sum_{i=m+1}^\infty \EE \langle \clt_{F_n} e_i, e_i \rangle_\clh \\
    = (N+2)\sum_{i=m+1}^\infty \langle \clt_Z e_i, e_i \rangle_\clh ,
\end{multline}
where the convergence follows from $\| \clt_{F_n} - \clt_Z \|_{\cls_1(\clh)} \to 0$ and the last quantity vanishes as $m \to \infty$ since $\clt_Z \in \cls_1(\clh)$.

For the cross-terms in \eqref{eq:3-blocks}, since $\clt_{F_n}$ is self-adjoint and
nonnegative, the same calculation as in \eqref{al:norm-ine-TR-HS-2} gives
\begin{equation}\label{eq:cross-cov}
\|P_m\clt_{F_n}Q_m\|_{\cls_1(\clh)}
=\|Q_m\clt_{F_n}P_m\|_{\cls_1(\clh)}
\leq\sqrt{\tr_\clh(P_m\clt_{F_n}P_m)\tr_\clh(Q_m\clt_{F_n}Q_m)}.
\end{equation}
However, note that $\Gamma(F_n,-L^{-1}F_n)$ now is neither self-adjoint nor positive,
so we cannot infer that $\Gamma(F_n,-L^{-1}F_n)^{1/2}$ exists. Define $U_n,V_n:\clh\to\mfh$ by
\[
U_nu \doteq N^{-1/4}\langle D_MF_n,u\rangle_\clh,
\qquad
V_nu \doteq N^{1/4}\langle-D_ML^{-1}F_n,u\rangle_\clh,
\qquad u \in \clh .
\]
With our convention for $\Gamma$, \eqref{eq:lemma-gamma-u-v} yields
\begin{equation}\label{eq:UV-factorization}
\Gamma(F_n,-L^{-1}F_n)=V_n^*U_n .
\end{equation}

We next compute the Hilbert--Schmidt norms of the two factors. Let $R$ denote
either $P_m$ or $Q_m$; both are self-adjoint projections that are diagonal in the
basis $\{e_i\}_{i\in\NN}$, so that $Re_i=e_i$ for $i\in J_R$ and $Re_i=0$ otherwise,
where $J_R=\{1,\dots,m\}$ if $R=P_m$ and $J_R=\{m+1,m+2,\dots\}$ if $R=Q_m$. Writing
$F_{n,i}\doteq\langle F_n,e_i\rangle_\clh=\sum_{r=1}^NI_r(f_{n,r,i})$, we have
$\langle D_MF_n,e_i\rangle_\clh=D_MF_{n,i}$, whence
\begin{equation}\label{eq:UV-HS-coordinates}
\|U_nR\|_{\cls_2(\clh:\mfh)}^2
=\frac1{\sqrt N}\sum_{i\in J_R}\|D_MF_{n,i}\|_\mfh^2,
\qquad
\|V_nR\|_{\cls_2(\clh:\mfh)}^2
=\sqrt N\sum_{i\in J_R}\bigl\|D_M(-L^{-1}F_{n,i})\bigr\|_\mfh^2 .
\end{equation}
Since $D_MF_{n,i}=\sum_{r=1}^N rI_{r-1}(f_{n,r,i})$ and
$-D_ML^{-1}F_{n,i}=\sum_{r=1}^N I_{r-1}(f_{n,r,i})$, the orthogonality of the chaoses
and the isometry property of the multiple integrals give, as in
\eqref{eq:322}--\eqref{eq:323},
\[
\EE\|D_MF_{n,i}\|_\mfh^2=\sum_{r=1}^N r^2(r-1)!\,\|f_{n,r,i}\|_{\mfh^{\otimes r}}^2
=\sum_{r=1}^N r\cdot r!\,\|f_{n,r,i}\|_{\mfh^{\otimes r}}^2,
\]
\[
\EE\bigl\|D_M(-L^{-1}F_{n,i})\bigr\|_\mfh^2=\sum_{r=1}^N (r-1)!\,\|f_{n,r,i}\|_{\mfh^{\otimes r}}^2
=\sum_{r=1}^N \frac1r\, r!\,\|f_{n,r,i}\|_{\mfh^{\otimes r}}^2 .
\]
Summing over $i\in J_R$ and using
$\EE\|R\cli_r(f_{n,r})\|_\clh^2=\sum_{i\in J_R}r!\|f_{n,r,i}\|_{\mfh^{\otimes r}}^2$,
the two identities in \eqref{eq:UV-HS-coordinates} yield
\begin{equation}\label{eq:balanced-derivative-energy}
\begin{split}
\EE\bigl(\|U_nR\|_{\cls_2(\clh:\mfh)}^2+\|V_nR\|_{\cls_2(\clh:\mfh)}^2\bigr)
&=\sum_{r=1}^N\left(\frac r{\sqrt N}+\frac{\sqrt N}{r}\right)
\EE\|R\cli_r(f_{n,r})\|_\clh^2\\
&\leq\left(\sqrt N+\frac1{\sqrt N}\right)\sum_{r=1}^N\EE\|R\cli_r(f_{n,r})\|_\clh^2\\
&=\left(\sqrt N+\frac1{\sqrt N}\right)\tr_\clh(R\clt_{F_n}R).
\end{split}
\end{equation}
Here the inequality follows from $r+N/r\leq N+1$ for $1\leq r\leq N$.
The Hilbert--Schmidt product inequality and Cauchy--Schwarz, first for the two products and then in probability, imply
\begin{align*}
&\EE\|P_m\Gamma(F_n,-L^{-1}F_n)Q_m\|_{\cls_1(\clh)}
+\EE\|Q_m\Gamma(F_n,-L^{-1}F_n)P_m\|_{\cls_1(\clh)}\\
&\quad\leq
\EE\bigl[\|V_nP_m\|_{\cls_2(\clh:\mfh)}\|U_nQ_m\|_{\cls_2(\clh:\mfh)}
+\|U_nP_m\|_{\cls_2(\clh:\mfh)}\|V_nQ_m\|_{\cls_2(\clh:\mfh)}\bigr]\\
&\quad\leq
\Bigl(
\EE\bigl(\|U_nP_m\|_{\cls_2(\clh:\mfh)}^2+\|V_nP_m\|_{\cls_2(\clh:\mfh)}^2\bigr)\Bigr)^{1/2}
\Bigl(\EE\bigl(\|U_nQ_m\|_{\cls_2(\clh:\mfh)}^2+\|V_nQ_m\|_{\cls_2(\clh:\mfh)}^2\bigr)\Bigr)^{1/2}\\
&\quad\leq\left(\sqrt N+\frac1{\sqrt N}\right)\sqrt{\tr_\clh(P_m\clt_{F_n}P_m)\tr_\clh(Q_m\clt_{F_n}Q_m)}.
\end{align*}
Together with \eqref{eq:cross-cov} and the definition of $G_n$, this yields
\begin{equation}\label{al:norm-ine-TR-HS-3-gam-sum}
\begin{split}
&\frac12\left(\EE\|P_mG_nQ_m\|_{\cls_1(\clh)}
+\EE\|Q_mG_nP_m\|_{\cls_1(\clh)}\right)
\\
&
\leq\left(1+\frac{\sqrt N+N^{-1/2}}2\right)\sqrt{\tr_\clh(P_m\clt_{F_n}P_m)\tr_\clh(Q_m\clt_{F_n}Q_m)}
\\&
\leq\sqrt{N+3}\,\sqrt{\tr_\clh(P_m\clt_{F_n}P_m)\tr_\clh(Q_m\clt_{F_n}Q_m)},
\end{split}
\end{equation}
where we used $\sqrt N+N^{-1/2}+2\leq\sqrt N+3\leq2\sqrt{N+3}$ for $N\geq1$. 

By trace-class convergence of the covariance operators, for every fixed $m$,
\[
\tr_\clh(P_m\clt_{F_n}P_m)\to \tr_\clh(P_m\clt_ZP_m),
\qquad
\tr_\clh(Q_m\clt_{F_n}Q_m)\to \tr_\clh(Q_m\clt_ZQ_m).
\]
Moreover, $\sup_{n,m}\tr_\clh(P_m\clt_{F_n}P_m)\leq\sup_n\tr_\clh(\clt_{F_n})<\infty$, and $\tr_\clh(Q_m\clt_ZQ_m)\to0$ as $m\to\infty$. Combining \eqref{eq:finite-vanish}, \eqref{eq:qm-qm-quantitative}, and \eqref{al:norm-ine-TR-HS-3-gam-sum} with the four-block bound \eqref{eq:3-blocks}, and letting first $n\to\infty$ and then $m\to\infty$, proves \eqref{eq:key-vanish}. The conclusion follows from \eqref{eq:two-terms}.
\end{proof}

\subsection{The case of infinite chaos expansion}

In this section, we consider the most general case. It is known that any centered sequence of random variables $F_n \in L^2(\Om:\clh)$ that is measurable with regard to a Gaussian process admits the chaos decomposition
\begin{equation} \label{eq:infinite-chaos-expansion}
    F_n = \sum_{r=1}^\infty \cli_r (f_{n,r}) = \sum_{r=1}^\infty (I_r \otimes \operatorname{Id}_{\clh})(f_{n,r}), \quad f_{n,r} = \sum_{i=1}^\infty f_{n,r,i} \otimes e_i \in \mfh^{\odot r} \otimes \clh;
\end{equation}
see Section 4 in \cite{dukZou24}. This representation allows us to derive sufficient conditions for convergence in distribution of the sequence $\{F_n\}_{n\ge 1}$. We give the statement in terms of  contractions, but note that different but equivalent formulations are possible; see Remark \ref{remark:equivalent-conditions-general}.

\begin{theorem}[Infinite-dimensional Fourth-Moment Theorem - Infinite Chaos Expansion]
    \label{thm:abstract-fourth-moment-theorem-infinite-chaos}
Let $\{ F_n \}_{n\ge 1}$ be an $\clh$-valued sequence admitting the representation given in \eqref{eq:infinite-chaos-expansion}, with respective covariance operators $\clt_{F_n}$. Let $\{Z_r: r \ge 1\}$ be a family of centered $\clh$-valued normal random variables with respective covariance operators $\clt_{Z_r}$. 
Fix an orthonormal basis $\{e_i\}_{i \in \NN}$ of $\clh$ and set $f_{n,r,i} \doteq \langle f_{n,r},e_i \rangle_\clh$.
Suppose that the following conditions are satisfied:
\begin{enumerate}[label=(\roman*)]
    \item for $r \ge 1$, $\| \clt_{\mathcal{I}_r(f_{n,r})} - \clt_{Z_r}\|_{\mathcal{S}_1(\clh)} \to 0$. \label{item1:thm-infinite-expansions-covar-oper}
    \item Letting $\clt_Z \doteq  \sum_{r=1}^\infty \clt_{Z_r}$, we have that $\| \clt_{Z} \|_{\mathcal{S}_1(\clh)} < \infty$.
    \label{item1.1:thm-infinite-expansions-trace-sec-mom}
   \item For $i\in \NN$ and $r \ge 2$,
    \[
    \lim_{n \to \infty} \|  f_{n,r,i}  \otimes_m  f_{n,r,i}  \|_{\mfh^{\otimes 2(r - m)}} = 0,
    \]
    for all $m=1,\dots,r-1$.  \label{item2:thm-infinite-expansions-contr}
    \item It holds that \label{item3:thm-infinite-expansions-tail}
    \begin{equation}
        \lim_{N \to \infty} \sup_{n \ge 1}  \sum_{r=N+1}^\infty \sum_{i=1}^\infty r! \|f_{n,r,i}\|_{\mfh^{\otimes r}}^2 = 0.
    \end{equation}
\end{enumerate}
Then, $F_n \xrightarrow{d} Z$, as $n \to \infty$, where $Z$ is the centered $\clh$-valued, normal random variable, with covariance operator $\clt_{Z}$.
\end{theorem}

\begin{proof}
Let 
\[ 
F_{n,N} \doteq \sum_{r=1}^N \cli_r (f_{n,r}) = \sum_{r=1}^N (I_r \otimes \operatorname{Id}_{\clh})(f_{n,r}), \quad f_{n,r} = \sum_{i=1}^\infty f_{n,r,i} \otimes e_i \in \mfh^{\odot r} \otimes \clh, \quad N \in \NN.
\]
  We apply the triangle inequality and Theorem \ref{thm:1} to obtain
   \begin{equation} \label{eq:infinite-chaos-exp-1}
    \begin{split}
         d_2(F_n,Z) &\le d_2(F_n, F_{n,N}) + d_2(F_{n,N}, Z) \\
         &\le d_2(F_n,F_{n,N}) + \frac{1}{2} \EE \| \Gamma(F_{n,N}, -L^{-1} F_{n,N}) - \clt_{Z} \|_{\cls_1(\clh)} \\
         &\le d_2(F_n,F_{n,N}) + \frac{1}{2} \EE \| \Gamma(F_{n,N}, -L^{-1} F_{n,N}) - \clt_{Z_N} \|_{\cls_1(\clh)} + \frac{1}{2} \|\clt_{Z_N} - \clt_Z \|_{\cls_1(\clh)},
    \end{split}
   \end{equation}
    where for all $N \in \NN$, $Z_N$ is the (possibly degenerate) centered Gaussian element in $\clh$ with covariance operator $\sum_{r=1}^N \clt_{Z_r}$. Note that, since $F_{n,N}$ admits a finite chaos expansion, it holds that for all $N$, $F_{n,N} \in \mathbb{D}^{1,2}(\clh)$. Since $F_{n,N}$ has a finite chaos expansion, under conditions \ref{item1:thm-infinite-expansions-covar-oper} and \ref{item2:thm-infinite-expansions-contr},  Theorem \ref{cor:abstract-fourth-moment-theorem-finite-chaos} says that, for each fixed $N$, 
    \begin{equation} \label{eq:infinite-chaos-exp-2}
        \EE \| \Gamma(F_{n,N}, -L^{-1} F_{n,N}) - \clt_{Z_N} \|_{\cls_1(\clh)} \to 0,
    \end{equation}
    as $n \to \infty$. Moreover, since both $Z_N$ and $Z$ are Gaussian random variables,
    \begin{equation} \label{eq:infinite-chaos-exp-3}
         \|\clt_Z - \clt_{Z_N} \|_{\mathcal{S}_1(\clh)} = \sum_{r=N+1}^\infty \sum_{i=1}^\infty \langle \clt_{Z_r} e_i, e_i \rangle_{\clh} \to 0,
    \end{equation}
    where the equality follows from positivity and trace summability, since by condition \ref{item1.1:thm-infinite-expansions-trace-sec-mom}, we have
    \[
    \sum_{i=1}^\infty \langle \clt_{Z} e_i, e_i \rangle_{\clh}= \sum_{r=1}^\infty \sum_{i=1}^\infty \langle \clt_{Z_r} e_i, e_i \rangle_{\clh} < \infty.
    \]
    It remains to prove that $d_2(F_{n,N},F_n) \to 0$. For this, we write 
    \begin{align} 
    &d_2(F_{n,N},F_n) 
    = 
    \sup_{h \in \clc_b^2(\clh), \|h\|_{\clc_b^2(\clh)} \le 1} |\EE [h(F_n)] - \EE [h(F_{n,N})]| 
    \le \sup_{h \in \Lip^1_b(\clh)}  
    |\EE [h(F_n)] - \EE [h(F_{n,N})] |
    \\&\le 
    \EE \|F_n - F_{n,N} \|_\clh \le \sqrt{\EE \|F_n - F_{n,N} \|_\clh^2} = \sqrt{\EE \norm{ \sum_{r=N+1}^\infty \cli_r(f_{n,r}) }_\clh^2} 
    \label{eq:infinite-chaos-exp-4-1}
    \\&= 
    \sqrt{\sum_{r=N+1}^\infty  \EE \| \cli_r(f_{n,r}) \|_\clh^2} \le   \sqrt{\sum_{r=N+1}^\infty  r! \| f_{n,r} \|_{\mfh^{\otimes r} \otimes \clh}^2} 
    \le 
    \sqrt{\sup_{n \ge 1} \sum_{i =1}^\infty \sum_{r=N+1}^\infty  r! \| f_{n,r,i} \|_{\mfh^{\otimes r}}^2} \label{eq:infinite-chaos-exp-4}.
    \end{align}
    In the calculations above, the second inequality in \eqref{eq:infinite-chaos-exp-4-1} follows by Cauchy-Schwarz. In \eqref{eq:infinite-chaos-exp-4}, the equality follows by the orthogonality of the multiple integrals, the first inequality by their isometry property (for these two see, e.g., display (2.10) in \cite{vidotto2025functional}), and the last inequality by Parseval's identity.

    Combining \eqref{eq:infinite-chaos-exp-1}, \eqref{eq:infinite-chaos-exp-2}, \eqref{eq:infinite-chaos-exp-3}, \eqref{eq:infinite-chaos-exp-4} and condition \ref{item3:thm-infinite-expansions-tail}, the conclusion follows.
\end{proof}

\begin{remark} \label{remark:equivalent-conditions-general}
    Condition \ref{item2:thm-infinite-expansions-contr} of Theorem \ref{thm:abstract-fourth-moment-theorem-infinite-chaos} can be replaced by equivalent conditions, e.g., with regard to the convergence of the fourth moments as given in item \ref{item2:thm:abstract-fourth-moment-theorem} of Theorem \ref{thm:abstract-fourth-moment-theorem}; see Remark 6.3.2 of \cite{nourdin2012normal} for the analogous fact in finite dimensions.
\end{remark}

\subsection{Quantitative bounds}

In this section, we provide bounds that quantify the rate of convergence to a normal random variable with regard to the $d_2$ distance. Although these bounds are cumbersome, they can become fairly explicit in specific applications.

Define the terms

\begin{equation} \label{eq:remainder-terms-def-quant-1}
\begin{split}
    \clr_{1,N} 
    &\doteq        
    \sqrt{\sup_{n \ge 1}  \sum_{i=1}^\infty \sum_{r=N+1}^\infty r! \|f_{n,r,i}\|_{\mfh^{\otimes r}}^2}   
    \\
    \clr_{2,n,N} 
    &\doteq  
    \frac{1}{2} \sum_{r = 1}^N  \|  \clt_{\cli_r(f_{n,r})} - \clt_{Z_r}\|_{\mathcal{S}_1(\clh)}     
    \\
    \clr_{3,n,m,N} 
    &\doteq \sqrt{m} N \sqrt{ 
    \sum_{i,j=1}^m  \left( \sum_{l=1}^N  \gamma_{n,i,j}^{(l)} + \sum_{l_1 \neq l_2=1}^N  \gamma_{n,i,j}^{(l_1,l_2)}  \right) }  
    \\
    \clr_{4,n,m,N} 
    &\doteq \frac{N+2}{2} \sum_{r=1}^N  
    \sum_{j=m+1}^{\infty} \langle \clt_{\cli_r(f_{n,r})} e_j,e_j \rangle_{\clh}    \\ \clr_{5,n,m,N}  &\doteq  \| \clt_{F_n} \|_{\mathcal{S}_1(\clh)}^{1/2} \left( (N+3)
\sum_{r=1}^N 
\sum_{j=m+1}^{\infty} \langle \clt_{\cli_r(f_{n,r})} e_j,e_j \rangle_{\clh} 
\right)^{1/2}  \\
    \clr_{6,N} &\doteq \frac{1}{2} \sum_{r=N+1}^\infty \sum_{i=1}^\infty \langle \clt_{Z_r} e_i, e_i \rangle_{\clh},
\end{split}
\end{equation}
with
\begin{equation}\label{eq:remainder-terms-def-quant-2}
\begin{aligned} 
    \gamma_{n,i,j}^{(l)} 
    &\doteq \frac{1}{2}  \sum_{\chi=1}^{l-1} c_l(\chi)^2 \left(  
    \| f_{n,l,i} \otimes_{l-\chi} f_{n,l,i} \|_{\mfh^{\otimes 2 \chi}}^2 + \| f_{n,l,j} \otimes_{l-\chi} f_{n,l,j} \|_{\mfh^{\otimes 2\chi}}^2 \right),
    \\
    \gamma_{n,i,j}^{(l_1,l_2)} 
    &\doteq 
    c(l_1,l_2) \| f_{n,l_1,i} \|^2_{\mfh^{\otimes l_1}} 
    \| f_{n,l_2,j} \otimes_{l_2-l_1}  f_{n,l_2,j} \|_{\mfh^{\otimes 2 l_1}} \1_{\{l_1 < l_2\}} 
    \\&\hspace{0.5cm}+ 
    \frac{l_1^2}{l_2^2}c(l_2,l_1) \| f_{n,l_2,j} \|^2_{\mfh^{\otimes l_2 }} 
    \| f_{n,l_1,i} \otimes_{l_1-l_2}  f_{n,l_1,i} \|_{\mfh^{\otimes 2 l_2}} \1_{\{l_2 < l_1\}}
    \\&\hspace{0.5cm}+
    \sum_{\chi =1}^{l_1 \wedge l_2 -1} c(l_1,l_2,\chi) 
    \left(  \| f_{n,l_1,i} \otimes_{l_1-\chi}  f_{n,l_1,i} \|_{\mfh^{\otimes 2\chi}}^2 
    + \| f_{n,l_2,j} \otimes_{l_2-\chi}  f_{n,l_2,j} \|_{\mfh^{\otimes 2 \chi}}^2 \right),
\end{aligned}
\end{equation}
where we have used the constants
\begin{equation} \label{eq:remainder-terms-def-quant-3}
\begin{split}
    c_{p}(r) 
    &\doteq p (r-1)! \binom{p-1}{r-1}^2 \sqrt{(2p-2r)!}, 
    \\
    c(p,q) 
    &\doteq p!^{2} \binom{q-1}{p-1}^{2} (q - p)!, \\
    c(p,q,\chi) &\doteq 
    \frac{p^{2}}{2}
    ( \chi - 1)!^{2}
    \binom{p-1}{\chi-1}^{2}
    \binom{q-1}{\chi-1}^{2}
    (p+q-2\chi)!,
    \qquad 1 \le \chi \le p\wedge q -1.
\end{split}
\end{equation}

\begin{theorem} \label{thm:quantitative-CLT}
Let $\{F_n\}_{n\ge 1}$ be as in \eqref{eq:infinite-chaos-expansion}. Then, for all $n \geq 1$,
    \begin{equation} \label{thm:eq-quantitative-bound}
        d_2(F_n, Z) \le \inf_{N \ge 1} \left[ \clr_{1,N} + \clr_{6,N} + \clr_{2,n,N} +  \inf_{m \ge 1} \left(   \clr_{3,n,m,N} + \clr_{4,n,m,N} + \clr_{5,n,m,N}    \right) \right],
    \end{equation}
    where the terms above were defined in \eqref{eq:remainder-terms-def-quant-1}, \eqref{eq:remainder-terms-def-quant-2}, and \eqref{eq:remainder-terms-def-quant-3}.
\end{theorem}

\begin{proof}
Recalling \eqref{eq:infinite-chaos-exp-1}, the first and third summand on the last line of \eqref{eq:infinite-chaos-exp-1} can be bounded by \eqref{eq:infinite-chaos-exp-4} and \eqref{eq:infinite-chaos-exp-3} in terms of $\clr_{1,N}, \clr_{6,N}$. It remains to justify the bound on $\frac{1}{2} \EE \| \Gamma(F_{n,N}, -L^{-1} F_{n,N}) - \clt_{Z_N} \|_{\cls_1(\clh)}$ in \eqref{eq:infinite-chaos-exp-1}.
        
Since, for each $n, N$, the random variable $F_{n,N}$ admits the representation given in \eqref{eq:finite-chaos-expansions}, the main tool for deriving a bound for $\frac{1}{2} \EE \| \Gamma(F_{n,N}, -L^{-1} F_{n,N}) - \clt_{Z_N} \|_{\cls_1(\clh)}$ is Theorem \ref{cor:abstract-fourth-moment-theorem-finite-chaos}. We consider first the bound in \eqref{eq:two-terms}, which gives
\begin{equation} \label{eq:quant-bound-fnN-ZN}
    \frac{1}{2} \EE \| \Gamma(F_{n,N}, -L^{-1} F_{n,N}) - \clt_{Z_N} \|_{\cls_1(\clh)} \le \frac{1}{2} \EE \| \Gamma(F_{n,N}, -L^{-1} F_{n,N}) - \clt_{F_{n,N}} \|_{\cls_1(\clh)} + \frac{1}{2}\| \clt_{F_{n,N}} - \clt_{Z_N} \|_{\cls_1(\clh)}.
\end{equation}
The second term on the right side of \eqref{eq:quant-bound-fnN-ZN} is bounded by $\clr_{2,n,N}$ by triangle inequality and the observation that $\clt_{F_{n,N}} = \sum_{r=1}^N \clt_{\cli_r(f_{n,r})}$, which is due to the uncorrelatedness of the multiple integrals.

We now justify the bound on the first term on the right side of \eqref{eq:quant-bound-fnN-ZN}. Apply one half of the four-block bound \eqref{eq:3-blocks}, with $F_n$ replaced by $F_{n,N}$. The finite-dimensional block is bounded by $\clr_{3,n,m,N}$ using \eqref{eq:finite-SHS} and \eqref{eq:9801}. The tail block is bounded by $\clr_{4,n,m,N}$ using \eqref{eq:qm-qm-quantitative}.

The two off-diagonal blocks are estimated jointly by \eqref{al:norm-ine-TR-HS-3-gam-sum}. In this application,
\begin{align*}
\tr_\clh(P_m\clt_{F_{n,N}}P_m)
&
\leq\tr_\clh(\clt_{F_{n,N}})
\leq\|\clt_{F_n}\|_{\cls_1(\clh)},\\
\tr_\clh(Q_m\clt_{F_{n,N}}Q_m)
&
=\sum_{r=1}^N\sum_{j=m+1}^{\infty}
\langle\clt_{\cli_r(f_{n,r})}e_j,e_j\rangle_\clh,
\end{align*}
by orthogonality of the chaoses. Hence their joint contribution, including the factor $1/2$, is bounded by $\clr_{5,n,m,N}$ as defined in \eqref{eq:remainder-terms-def-quant-1}. This concludes the proof.
\end{proof}

\begin{remark}\label{remark-quanti}
\begin{enumerate}
    \item The bound in \eqref{thm:eq-quantitative-bound} can be significantly  simplified when $F_{n}$ assumes a simpler form. In particular, when $F_{n}$ admits the finite chaos expansion of \eqref{eq:finite-chaos-expansions}, then the corresponding bound is given by
    \[
    d_2(F_n, Z) \le  \clr_{2,n,N}+ \inf_{m \ge 1} \left[ \sum_{i=3}^5 \clr_{i,n,m,N} \right],
    \]
    where since $N$ is fixed, the dependence on $N$ can be dropped. In analogy, when $F_n$ belongs to a chaos of fixed order (i.e., admits the representation \eqref{eq:seq-multiple-fixed-chaos}), then an even simpler bound can be obtained. This bound will contain a term analogous to $\clr_{2,n,N}$ (with a single summand $p$) and terms analogous to $\clr_{i,n,m,M}, i = 3,4,5$, but without the summations over $N$, the coefficients with regard to $N$, and without the term corresponding to $\gamma_{n,i,j}^{(l_1,l_2)}$.
    \label{remark-quanti-1}
\item \label{remark-quanti-fixed-chaos-weighted}
We comment here on an alternative bound on the $d_2$-distance that arises when  $F_n=\cli_p(f_n)$ admits the single chaos expansion of \eqref{eq:seq-multiple-fixed-chaos} and can lead to potentially sharper bounds when the coordinate contractions have sufficient summable decay.
Let
$Z$ be a centered Gaussian random variable on $\clh$ with covariance
operator $\clt_Z$. Define
\begin{equation}\label{eq:fixed-chaos-weighted-eta}
\eta_{n,j}
\doteq
\sum_{r=1}^{p-1}c_p(r)^2
\big\|f_{n,p,j}\otimes_{p-r}f_{n,p,j}\big\|_{\mfh^{\otimes 2r}}^2,
\end{equation}
where $c_p(r)$ is given in
\eqref{eq:remainder-terms-def-quant-3}.
For each $n\geq1$ such that
$\sum_{j=1}^{\infty}\eta_{n,j}^{1/3}<\infty$, we show that
\begin{equation}\label{eq:fixed-chaos-weighted-bound}
d_2(F_n,Z)
\leq
\frac12\|\clt_{F_n}-\clt_Z\|_{\cls_1(\clh)}
+
\frac12
\left(\sum_{j=1}^{\infty}\eta_{n,j}^{1/3}\right)^{3/2}.
\end{equation}
This estimate uses the same coordinatewise self-contractions as
\eqref{eq:remainder-terms-def-quant-2}, but does not require an
optimization over a spatial cutoff.

To prove \eqref{eq:fixed-chaos-weighted-bound}, recall $G_n$ from \eqref{def:pmqmgn}.
The product formula for multiple Wiener integrals, the fact that
symmetrization cannot increase the tensor norm, and Cauchy--Schwarz
give
\begin{equation}\label{eq:fixed-chaos-weighted-entries}
\begin{aligned}
\EE\big|\langle G_ne_i,e_j\rangle_\clh\big|^2
&\leq
\sum_{r=1}^{p-1}c_p(r)^2
\|f_{n,p,i}\otimes_r f_{n,p,j}\|_{\mfh^{\otimes(2p-2r)}}^2
\\
&\leq
\sum_{r=1}^{p-1}c_p(r)^2
\|f_{n,p,i}\otimes_{p-r}f_{n,p,i}\|_{\mfh^{\otimes 2r}}
\|f_{n,p,j}\otimes_{p-r}f_{n,p,j}\|_{\mfh^{\otimes 2r}}
\leq \sqrt{\eta_{n,i}\eta_{n,j}}.
\end{aligned}
\end{equation}
Recall $P_m$ from \eqref{def:pmqmgn}.
For $\delta>0$, define an invertible diagonal operator $D$ on
$P_m\clh$ by $De_j=(\eta_{n,j}+\delta)^{1/12}e_j$, for $1\leq j\leq m$.
Applying H\"older's inequality for Schatten norms with $1/4,1/4,$ and $1/2$, followed by
Cauchy--Schwarz and \eqref{eq:fixed-chaos-weighted-entries}, yields
\begin{align}
\EE\|P_mG_nP_m\|_{\cls_1(\clh)}
&=
\EE\| D\bigl(D^{-1}P_mG_nP_mD^{-1}\bigr)D \|_{\cls_1(\clh)} \nonumber
\\
&\leq
\|D\|_{\cls_4(P_m\clh)}^2
\left(
\EE\|D^{-1}P_mG_nP_mD^{-1}\|_{\cls_2(P_m\clh)}^2
\right)^{1/2} \nonumber
\\
&\leq
\left( \sum_{j=1}^m(\eta_{n,j}+\delta)^{1/3} \right)^{1/2}
\left(
\sum_{i,j=1}^m
\frac{\sqrt{\eta_{n,i}\eta_{n,j}}}
{(\eta_{n,i}+\delta)^{1/6}(\eta_{n,j}+\delta)^{1/6}}
\right)^{1/2} \nonumber
\\&\leq 
\left( \sum_{j=1}^m(\eta_{n,j}+\delta)^{1/3} \right)^{3/2}
\to
\left(\sum_{j=1}^m\eta_{n,j}^{1/3}\right)^{3/2}, \label{eq:98uj9id}
\end{align}
as $\delta \to 0$.
By Lemma \ref{le:analogue-Th-2.6} \ref{item1:lem:Gamma-covariance}--\ref{item2:lem:Gamma-covariance}, the operator
$\Gamma(F_n,-L^{-1}F_n)$ is nonnegative, belongs to
$L^1(\Omega:\cls_1(\clh))$, and has expectation $\clt_{F_n}$. Hence,
\[
    \EE\|G_n\|_{\cls_1(\clh)}
    \leq
    \EE\|\Gamma(F_n,-L^{-1}F_n)\|_{\cls_1(\clh)}
    +\|\clt_{F_n}\|_{\cls_1(\clh)}
    =2\operatorname{Tr}(\clt_{F_n})<\infty.
\]
Since $G_n$ is trace class almost surely and $P_m\to I$ in the sense that $\|P_mx-x\|_\clh\to0$ for every $x\in\clh$, 
Proposition 2.1 in \cite{WIDOM19761} yields $\lim_{m \to \infty} \|G_n-P_mG_nP_m\|_{\cls_1(\clh)} = 0$, a.s.. Then, applying Fatou's lemma together with \eqref{eq:98uj9id} yields
\begin{align*}
    \EE\|G_n\|_{\cls_1(\clh)}
    \leq
    \liminf_{m\to\infty}
    \EE\|P_mG_nP_m\|_{\cls_1(\clh)}
    \leq
    \lim_{m\to\infty}
    \left(\sum_{j=1}^m\eta_{n,j}^{1/3}\right)^{3/2}
    =
    \left(\sum_{j=1}^{\infty}\eta_{n,j}^{1/3}\right)^{3/2}.
\end{align*}
Finally, Theorem \ref{thm:1} and the identity $\Gamma(F_n,-L^{-1}F_n)-\clt_Z
    =G_n+(\clt_{F_n}-\clt_Z)$
give, by the triangle inequality in $\cls_1(\clh)$,
\begin{align*}
    d_2(F_n,Z)
    \leq
    \frac12\EE\|G_n\|_{\cls_1(\clh)}
    +\frac12\|\clt_{F_n}-\clt_Z\|_{\cls_1(\clh)}
    \leq
    \frac12\|\clt_{F_n}-\clt_Z\|_{\cls_1(\clh)}
    +\frac12
    \left(\sum_{j=1}^{\infty}\eta_{n,j}^{1/3}\right)^{3/2},
\end{align*}
which proves \eqref{eq:fixed-chaos-weighted-bound}.

    \item The bound of \eqref{thm:eq-quantitative-bound} leads to a quantitative CLT under the assumptions of Theorem \ref{thm:abstract-fourth-moment-theorem-infinite-chaos}. To see this, first select $N$ large enough so that $\clr_{1,N} + \clr_{6,N} < \varepsilon/3$, then select $m$ large enough so that $\left(    \clr_{4,n,m,N}+ \clr_{5,n,m,N} \right) < \varepsilon/3$, and finally choose $n$ large enough so that $ \clr_{2,n,N}+ \clr_{3,n,m,N} < \varepsilon/3$, leading to the estimate $d_2(F_{n}, Z) \le \varepsilon$.
    \label{remark-quanti-2}
\end{enumerate}
\end{remark}

\subsection{The case of vector-valued multiple integrals} \label{subsec:abstract-vector}

In this section, we extend the Fourth-Moment Theorem to the case of a vector of $\clh$-valued random variables, each belonging to a (possibly different) fixed chaos and admitting the representation in \eqref{eq:seq-multiple-fixed-chaos}. For the sake of simplicity, we omit the analogues of items \ref{item2-all:thm:abstract-fourth-moment-theorem} and \ref{item3:thm:abstract-fourth-moment-theorem} of Theorem \ref{thm:abstract-fourth-moment-theorem} for the results in this section. One can also consider a version of the following theorem in which each of the coordinates admits an infinite chaos expansion, as well as a quantitative version of that. We will not include these versions, as they are very technical but would offer little additional value.

We say that a sequence of vectors $\{F_n\}_{n \geq 1}$ on $\clh^K, \; K \ge 2$, belongs to a fixed Wiener chaos if it admits the representation
\begin{equation} \label{eq:vector-fixed-chaos-repres}
    F_n = \left( \cli_{r_1}(f_{n,1}),\dots,\cli_{r_K}(f_{n,K})\right), \quad   f_{n,k} \in \mfh^{\odot r_k} \otimes \clh, \quad k = 1,\dots K, \quad r_1,\dots,r_K \in \NN.
\end{equation}

\begin{theorem}[Infinite-dimensional Fourth-Moment Theorem -- Vector Case]
\label{thm:abstract-fourth-moment-theorem-finite-chaos-vector}
Let $Z$ be a centered Gaussian random variable on $\clh^K$ with respective covariance operator $\clt_{Z}$, and let $\{ F_n \}_{n\ge 1}$ be an $\clh^K$-valued sequence admitting the representation given in \eqref{eq:vector-fixed-chaos-repres}, with respective covariance $\clt_{F_n}$. Suppose $\| \clt_{F_n} - \clt_Z\|_{\mathcal{S}_1(\clh^K)} \to 0$. Then, as $n \to \infty$, the following statements are equivalent
    \begin{enumerate}[label=(\roman*)]
    \item $F_n \xrightarrow{d} Z$, in the topology of $\clh^K$.
    \label{item1:thm:vector-case}
    \item     \label{item2:thm:vector-case}
For some orthonormal basis $\{e_i\}_{i \in \NN}$ of $\clh$, let $f_{n,k,i} = \langle f_{n,k},e_i \rangle_\clh$ and $Z_{k,i} = \langle Z^{(k)}, e_i \rangle_\clh \in \RR$. Then, for all $k=1,\dots,K$ such that $r_k \ge 2$,
    \begin{equation} \label{eq:conv-moments-vector}
        \EE \left[ (I_{r_k}(f_{n,k,i}))^4 \right] \to \EE \left[  Z_{k,i}^4 \right], \quad    i \in \NN,
    \end{equation}
    as $n \to \infty$.
    \item For some orthonormal basis $\{e_i\}_{i \in \NN}$, let $f_{n,k,i} = \langle f_{n,k},e_i \rangle_\clh$. Then, it holds that for all $i\in \NN$ and $k = 1,\dots, K$ such that $r_k \ge 2$, \label{item3:thm:vector-case}
    \[
    \lim_{n \to \infty} \|  f_{n,k,i}  \otimes_m  f_{n,k,i}  \|_{\mfh^{\otimes 2(r_k - m)}} = 0,
    \]
    for all $m=1,\dots,r_k-1$.
    \end{enumerate}
\end{theorem}

\begin{proof}
We obtain this result by means of Theorem \ref{cor:abstract-fourth-moment-theorem-finite-chaos}. We start by matching the setting of Theorem \ref{cor:abstract-fourth-moment-theorem-finite-chaos} with the one of the present result.

Let $\widetilde{\clh}$ be the Hilbert space $\clh \otimes \RR^K$ and denote by $\{u_1,\dots,u_K\}$ the canonical orthonormal basis of $\RR^K$. Define the unitary map
\[
U:\clh^K\to\widetilde{\clh},\qquad 
U(x^{(1)},\dots,x^{(K)})=\sum_{k=1}^K x^{(k)}\otimes u_k .
\]
Set $\widetilde F_n = U(F_n)$ and $\widetilde Z = U(Z)$. The operator $U$ is an isometric isomorphism since it is unitary, hence continuous and bijective, with $U^{-1}$ also continuous.
Therefore, convergence of $F_n$ to $Z$ is equivalent to convergence of $\widetilde F_n$ to $\widetilde Z$. Take $x\in \clh^K$, then, for any $h\in \widetilde{\clh}$,
\[
\begin{aligned}
((Ux)\otimes(Ux))h
&= \langle h, Ux\rangle_{\widetilde{\clh}} \,Ux 
= \langle U^{*}h, x\rangle_{\clh^K} \,Ux 
= U\bigl(\langle U^{*}h, x\rangle_{\clh^K} \,x\bigr) \\
&= U\bigl((x\otimes x)(U^{*}h)\bigr) 
= \bigl(U(x\otimes x)U^{*}\bigr)h .
\end{aligned}
\]
Then, we can use the identification
\begin{multline}
    \clt_{\widetilde F_n} - \clt_{\widetilde Z} 
    = \EE \left[ \widetilde F_n \otimes \widetilde F_n - \widetilde Z \otimes \widetilde Z    \right] 
    = \EE \left( U \left(  F_n \otimes F_n -  Z \otimes  Z    \right) U^* \right) 
    = U \EE \left[  F_n \otimes F_n -  Z \otimes  Z    \right] U^* ,
\end{multline}
where the expectation and the operation $U$ in the third equality can be exchanged by elementary properties of the Bochner integral, since the latter is a bounded linear operator. Set $T_n \doteq  \clt_{F_n} - \clt_Z$ and $\widetilde T_n \doteq  \clt_{\widetilde F_n} - \clt_{\widetilde Z}$. 
From the previous identity we have $\widetilde T_n =  U T_n U^*$. Then
\begin{align*}
|\widetilde T_n|
&= \bigl((\widetilde T_n)^*\widetilde T_n\bigr)^{1/2} = \bigl(U T_n^* U^* U T_n U^*\bigr)^{1/2} = \bigl( U T_n^*T_n U^*\bigr)^{1/2} = U |T_n| U^*,
\end{align*}
with the last line following from a standard spectral decomposition argument. 
Using the unitary invariance of the trace, we obtain
\[
\|\clt_{\widetilde F_n}-\clt_{\widetilde Z}\|_{\mathcal S_1(\widetilde{\clh})}
= \operatorname{Tr}|\widetilde T_n|
= \operatorname{Tr}|T_n|
= \|\clt_{F_n}-\clt_{Z}\|_{\mathcal S_1(\clh^K)}\to 0.
\]

    \textit{Proof of \ref{item1:thm:vector-case} implies \ref{item2:thm:vector-case}:} The fact that $F_n \xrightarrow{d} Z$ in $\clh^K$ implies that
    \begin{equation*}
        \cli_{r_k} (f_{n,k}) \xrightarrow{d} Z^{(k)}
    \end{equation*}
    in $\clh$ as $n \to \infty$, and, moreover, we have convergence of the respective covariance operators in the trace-class norm. Then, \eqref{eq:conv-moments-vector} follows for all $i \in \NN$ and for $k =1,\dots,K$ from the implication \ref{item1:thm:abstract-fourth-moment-theorem} implies \ref{item2:thm:abstract-fourth-moment-theorem} of Theorem \ref{thm:abstract-fourth-moment-theorem}.

    \textit{Proof of \ref{item2:thm:vector-case} implies \ref{item3:thm:vector-case}:} This follows directly from the componentwise application of the implication \ref{item2:thm:abstract-fourth-moment-theorem} to \ref{item5:thm:abstract-fourth-moment-theorem-contr} of Theorem \ref{thm:abstract-fourth-moment-theorem}.

\textit{Proof of \ref{item3:thm:vector-case} implies \ref{item1:thm:vector-case}:} Since $
F_n = \bigl(\cli_{r_1}(f_{n,1}),\dots,\cli_{r_K}(f_{n,K})\bigr)
$, then,
\[
\widetilde F_n 
= U(F_n)
= \sum_{k=1}^K \cli_{r_k}(f_{n,k})\otimes u_k .
\]
Let $N = \max_{1\le k\le K} r_k$. For each $r$ in $\{1,\dots,N\}$ define the kernel
\[
g_{n,r} \in \mfh^{\odot r}\otimes\widetilde{\clh}
\qquad\text{by}\qquad
g_{n,r}=\sum_{k: r_k=r} f_{n,k}\otimes u_k
\]
and set $g_{n,r}=0$ if no index $k$ satisfies $r_k=r$. Denote the multiple Wiener integral on $\widetilde{\clh}$ by $\widetilde{\cli}_r = \cli_r\otimes \Id_{\RR^K} = I_r \otimes \Id_\clh \otimes \Id_{\RR^K}$. Hence
\[
\widetilde{\cli}_r(g_{n,r})
= \sum_{k: r_k=r} \cli_r(f_{n,k})\otimes u_k, \qquad \widetilde F_n = \sum_{r=1}^N \widetilde{\cli}_r(g_{n,r}),
\]
and the latter is a finite chaos expansion on $\widetilde{\clh}$.

\par
Fix the orthonormal basis $\{\widetilde e_{i,k}\}_{i\in\NN,\,k=1,\dots,K}$ of $\widetilde{\clh}$ defined by
\[
\widetilde e_{i,k} = e_i\otimes u_k ,
\]
where $\{e_i\}$ is the basis appearing in assumption \ref{item3:thm:vector-case}. For each $r$, $i$ and $k$ define the kernel
\[
g_{n,r,i,k} = \langle g_{n,r}, \widetilde e_{i,k} \rangle_{\widetilde{\clh}} \in \mfh^{\odot r}.
\]
Using the definition of $g_{n,r}$ and orthonormality of $\{u_k\}$, one obtains
\begin{equation} \label{eq:gnrik-fnki}
    g_{n,r,i,k}
= \sum_{l : r_l = r} \langle f_{n,l} \otimes u_l , e_i \otimes u_k \rangle_{\clh \otimes \RR^K} =
\begin{cases}
\langle f_{n,k}, e_i\rangle_{\clh} = f_{n,k,i} &\text{if } r_k = r,\\[1mm]
0 &\text{if } r_k \ne r .
\end{cases}
\end{equation}
\par
Assumption \ref{item3:thm:vector-case} states that for every $k$ with $r_k \ge 2$, for every $i$ and for every $m$ in $\{1,\dots,r_k-1\}$,
\[
\| f_{n,k,i} \otimes_m f_{n,k,i} \|_{\mfh^{\otimes(2r_k-2m)}} \to 0 .
\]
In view of \eqref{eq:gnrik-fnki}, it follows that for each $r$ in $\{2,\dots,N\}$, each $m$ in $\{1,\dots,r-1\}$, and every pair $(i,k)$,
\[
\| g_{n,r,i,k} \otimes_m g_{n,r,i,k} \|_{\mfh^{\otimes(2r-2m)}} \to 0 .
\]
This is exactly the contraction condition required in Theorem \ref{cor:abstract-fourth-moment-theorem-finite-chaos} applied to the $\widetilde{\clh}$-valued chaos expansion of $\widetilde F_n$.

\par
All the hypotheses of Theorem \ref{cor:abstract-fourth-moment-theorem-finite-chaos} are satisfied for the pair $\widetilde F_n$ and $\widetilde Z$ in the Hilbert space $\widetilde{\clh}$. Hence, $\widetilde F_n \xrightarrow{d} \widetilde Z$ in $\widetilde{\clh}$.
Since $U$ is unitary, this is equivalent to $F_n \xrightarrow{d} Z$ in $\clh^K.$
This completes the proof of \ref{item3:thm:vector-case} implies \ref{item1:thm:vector-case}.
\end{proof}

\begin{remark}
\begin{enumerate}
    \item Theorem \ref{thm:abstract-fourth-moment-theorem-finite-chaos-vector} here with $\clh= \RR$ should be compared with the multivariate Fourth-Moment Theorem of \cite{PeccatiTudor2005}. In particular, our Conditions \ref{item1:thm:vector-case}, \ref{item2:thm:vector-case}, and \ref{item3:thm:vector-case} are analogous to respectively Conditions (iii), (v), and (i) in Theorem 1 of \cite{PeccatiTudor2005}.
\end{enumerate}
\end{remark}

\section{Applications}

Section \ref{subsec:Cremers--Kadelka} presents an application to the celebrated Cremers--Kadelka Theorem. Section \ref{subsec:KRR} concerns a central limit theorem for the Kernel Ridge regression estimator. Finally, Section \ref{subsec:spde-general} presents two applications of our tools to the stochastic heat equation.

\subsection{Quantitative Cremers--Kadelka Theorem}
\label{subsec:Cremers--Kadelka}

A central topic in the theory of stochastic processes is identifying sufficient conditions for the convergence in distribution of a sequence of stochastic processes to a limiting one, in the topology induced by viewing the process as an element of a function space. Perhaps the most important procedure for obtaining such a result was proposed by Prokhorov in his fundamental paper \cite{Prokhorov:1956:CRP} and entails that, for the spaces $C([0,1])$ and $D([0,1])$, convergence in finite dimensional distribution, combined with tightness in the respective functional space, suffices to show the desired functional convergence. This was introduced (slightly adjusted) in the context of $L^p$ spaces by \cite{Grinblat:1976:LT}, and then extended to stochastic processes with paths in general Lusin spaces in \cite{CremersKadelkaLusin84}. In the context of $L^p(Y, \mathcal{A},\mu)$ spaces, where $(Y, \mathcal{A},\mu)$ is a $\sigma$-finite measure space, the condition for tightness was weakened in \cite{cremers1986weak}, who identified easily verifiable conditions for the convergence of a sequence of stochastic processes with sample paths in $L^p$ spaces.

In this section, we upgrade the important result of Cremers-Kadelka to its quantitative version, when the limiting distribution is a Gaussian random variable in $\clh = L^2(Y) \doteq L^2(Y, \mathcal{A},\mu)$ and the pre-limit can be expressed as a fixed Wiener chaos. Assuming the conditions that were provided therein, we prove that the conditions of our Theorem \ref{thm:abstract-fourth-moment-theorem} are automatically satisfied, thereby yielding a quantitative functional central limit theorem through an application of Theorem \ref{thm:quantitative-CLT}. In its original form, the (qualitative) theorem of Cremers-Kadelka is stated as Theorem 2 in \cite{cremers1986weak}.

Our assumptions are stronger in that the sequence lies in a fixed chaos. On the other hand, the CLT obtained is automatically quantitative. However, the fixed-chaos assumption is natural for particular nonlinear functionals of Gaussian observations, such as centered sums of squares and, more generally, Hermite partial-sum processes.

\begin{theorem}
\label{thm:KandC}
Suppose that $\xi_n$ and $\xi$ are jointly measurable processes whose paths define random variables in the separable Hilbert space $L^2(Y)$. Assume that, for a fixed integer $p\ge2$,
$\xi_n$ admits a fixed chaos representation
$\xi_n = \mathcal{I}_p(f_{p,n})$ with $f_{p,n} \in \mfh^{\odot p} \otimes L^2(Y)$, and that $\xi$ is a centered Gaussian random variable. Suppose further that the finite-dimensional distributions converge, i.e.,
\begin{equation}
    (\xi_n(s_1), \dots, \xi_n(s_q))^\top \xrightarrow{d}
    (\xi(s_1), \dots, \xi(s_q))^\top,
\end{equation}
in $\RR^{q}$ for $\mu^{\otimes q}$-almost every $(s_1,\dots,s_q)\in Y^q$ and for all $q \in \NN$.
Assume that either of the two following conditions is satisfied:
\begin{enumerate}[label=(\roman*)]
    \item \label{item:kadelka1}
    $\{ |\xi_n|^2, n\in\NN \}$ is $\PP\otimes\mu$ uniformly integrable.
    If $\mu(Y)=\infty$, assume also that there exist measurable sets $Y_m\uparrow Y$, with $\mu(Y_m)<\infty$, such that
    \begin{equation}\label{eq:CK-spatial-tail}
        \lim_{m\to\infty}\sup_{n\ge1}
        \EE\int_{Y\setminus Y_m}|\xi_n(s)|^2\,\mu(ds)=0.
    \end{equation}
    \item $\limsup_{n \to \infty} \EE \int_Y |\xi_n(s)|^2 \mu(ds) \le \EE \int_Y |\xi(s)|^2 \mu(ds)<\infty$.
    \label{item:kadelka2}
\end{enumerate}
Then, $\xi_n\xrightarrow{d}\xi$ in $L^2(Y)$. Moreover, the bounds of Theorem \ref{thm:quantitative-CLT} (see also Remark \ref{remark-quanti} \eqref{remark-quanti-1} for the case of fixed chaos) apply.
\end{theorem}

The spatial-tail condition in \ref{item:kadelka1} ensures that no $L^2$ mass may escape to spatial infinity, and is not necessary when $\mu(Y)<\infty$.

\begin{proof}
By the qualitative Cremers--Kadelka theorem, the assumptions yield the convergence in distribution
\begin{equation}\label{eq:weak-C-K}
    \xi_n\xrightarrow{d}\xi,\quad\text{in the topology of }L^2(Y).
\end{equation}
Furthermore, Theorem \ref{thm:abstract-fourth-moment-theorem} provides equivalent conditions for weak convergence given convergence of the covariance operators in trace-class norm. Therefore, it suffices to show $\|\clt_{\xi_n}-\clt_{\xi}\|_{\mathcal{S}_1(L^2(Y))}\to0$ to employ these equivalent conditions. Once these equivalent conditions hold, 
Remark \ref{remark-quanti} \eqref{remark-quanti-2} implies a quantitative CLT.
We infer convergence in trace-class norm by applying Lemma \ref{lem:trace-condition}.

We first show that, under either Condition \ref{item:kadelka1} or Condition \ref{item:kadelka2},
\begin{equation}\label{eq:CK-second-moment-convergence}
    \EE\|\xi_n\|_{L^2(Y)}^2
    \to 
    \EE\|\xi\|_{L^2(Y)}^2.
\end{equation}
Under Condition \ref{item:kadelka1}, let $Y_m$ be as in \eqref{eq:CK-spatial-tail}, taking $Y_m=Y$ when $\mu(Y)<\infty$. For fixed $m$ and $R>0$, the one-dimensional distributional convergence and the dominated convergence theorem give
\begin{equation}\label{eq:CK-truncated-moments}
    \EE\int_{Y_m}\big(|\xi_n(s)|^2\wedge R\big)\,\mu(ds)
    \to 
    \EE\int_{Y_m}\big(|\xi(s)|^2\wedge R\big)\,\mu(ds).
\end{equation}
Here we used that $x\mapsto x^2\wedge R$ is bounded and continuous and that $\mu(Y_m)<\infty$. Letting $R\to\infty$, uniform integrability gives convergence of the second moments integrated over $Y_m$. Letting $m\to\infty$ and using \eqref{eq:CK-spatial-tail} and $\EE\|\xi\|_{L^2(Y)}^2<\infty$ yields \eqref{eq:CK-second-moment-convergence}.
Under Condition \ref{item:kadelka2}, \eqref{eq:weak-C-K} and the Portmanteau theorem instead give
\[
\EE\|\xi\|^2_{L^2(Y)}
\le \liminf_{n\to\infty}\EE\|\xi_n\|^2_{L^2(Y)}
\le \limsup_{n\to\infty}\EE\|\xi_n\|^2_{L^2(Y)}
\le \EE\|\xi\|^2_{L^2(Y)},
\]
which again proves \eqref{eq:CK-second-moment-convergence}.
Since $\tr(\clt_{\xi_n})=\EE\|\xi_n\|^2_{L^2(Y)}$ and $\tr(\clt_{\xi})=\EE\|\xi\|^2_{L^2(Y)}$,
the convergence verifies Condition \ref{lem:trace-condition-item1} of Lemma \ref{lem:trace-condition}.

We turn to verifying Condition \ref{lem:trace-condition-item2} of Lemma \ref{lem:trace-condition}, i.e., the weak convergence of the covariance operators in the trace-class topology. We prove that for all $A\in\mathcal{L}(L^2(Y))$, as $n\to\infty$,
\begin{align}\label{eq:abc}
    \tr(A\clt_{\xi_n})
    = \EE\langle A\xi_n,\xi_n\rangle_{L^2(Y)}
    \to \EE\langle A\xi,\xi\rangle_{L^2(Y)}
    = \tr(A\clt_{\xi}).
\end{align}
Since $\xi_n\xrightarrow{d}\xi$ by \eqref{eq:weak-C-K}, and upon noticing that the map
$x\mapsto\langle Ax,x\rangle_{L^2(Y)}$ from $L^2(Y)$ to $\RR$ is continuous in the norm topology of $L^2(Y)$, the continuous mapping theorem says that
\begin{equation}\label{appl2-eq2}
    \langle A\xi_n,\xi_n\rangle_{L^2(Y)}\xrightarrow{d}
    \langle A\xi,\xi\rangle_{L^2(Y)}.
\end{equation}
To conclude \eqref{eq:abc} from \eqref{appl2-eq2}, we also need to show that the family $\{\langle A\xi_n,\xi_n\rangle_{L^2(Y)}\}_{n\in\NN}$ is uniformly integrable. To see this, note that for any $\varepsilon\in(0,1)$,
\begin{equation}\label{eq:abc-2}
\begin{aligned}
    \EE\big|\langle A\xi_n,\xi_n\rangle_{L^2(Y)}\big|^{1+\varepsilon}
    \le
    \left(\EE\big|\langle A\xi_n,\xi_n\rangle_{L^2(Y)}\big|^2
    \right)^{\frac{1+\varepsilon}{2}}
    \le
    \left(\|A\|_{\op}^2\EE\|\xi_n\|_{L^2(Y)}^4
    \right)^{\frac{1+\varepsilon}{2}},
\end{aligned}
\end{equation}
where we used H\"older's inequality. The last inequality follows since $A\in\mathcal{L}(\clh)$.
To continue bounding the final expression in \eqref{eq:abc-2}, note that \eqref{eq:CK-second-moment-convergence} implies $ M\doteq\sup_{n\ge1}\EE\|\xi_n\|_{L^2(Y)}^2<\infty$, so that for all $n\in\NN$, $\EE\|\xi_n\|_{L^2(Y)}^2\le M$.

Finally, using hypercontractivity of the $p$th order Wiener chaos on $\xi_n$ (Theorem 2.7.2 in \cite{nourdin2012normal}), there is a constant $c>0$ such that
\begin{align} \label{eq:jkek}
    \EE\|\xi_n\|_{L^2(Y)}^4
    \le c\big(\EE\|\xi_n\|_{L^2(Y)}^2\big)^2
    \le cM^2.
\end{align}
The Hilbert space-valued moment bound in \eqref{eq:jkek} follows by applying scalar hypercontractivity to the coordinates in an orthonormal basis, followed by Cauchy--Schwarz and Parseval's identity.
Then \eqref{appl2-eq2}, together with uniform integrability and Theorem 3.5 in \cite{billingsley1999convergence}, implies \eqref{eq:abc}.

Both conditions of Lemma \ref{lem:trace-condition} are therefore satisfied, which proves $\|\clt_{\xi_n}-\clt_{\xi}\|_{\mathcal{S}_1(L^2(Y))}\to0$.
\end{proof}

The following example gives a concrete Hermite-variation for which the fixed-chaos structure assumed in Theorem \ref{thm:KandC} arises. For these processes, weak convergence in $D([0,1])$ is already covered by the functional Breuer--Major theorem of \cite{nourdin2020functional}. The purpose of the present application is to obtain an estimate on the $d_2$ distance in $L^2([0,1])$.

\begin{example}\label{ex:CK-Hermite-processes}
Let $\{X_k\}_{k\in\mathbb{Z}}$ be a centered stationary Gaussian sequence with $\EE X_0^2=1$ and covariance function
$\rho(v)\doteq\EE[X_0X_v]$.
Fix $p\ge2$ and suppose that
\begin{equation}\label{eq:CK-BM-summability}
    \sum_{v\in\mathbb{Z}}|\rho(v)|^p<\infty.
\end{equation}
Let $H_p$ be the Hermite polynomial of order $p$, normalized by 
$\EE[H_p(G)^2]=p!$ for $G\sim\mathcal{N}(0,1)$, and define
\begin{equation}\label{eq:CK-Hermite-partial-sums}
    \xi_n(t)\doteq\frac1{\sqrt n}
    \sum_{k=1}^{\lfloor nt\rfloor}H_p(X_k),
    \qquad t\in[0,1].
\end{equation}
We refer to Chapter 4.1 in \cite{pipiras2017long} for more details on Hermite expansions.
We view $\xi_n$ as a random variable in $\clh=L^2([0,1])$.
The Breuer--Major theorem gives convergence of its finite-dimensional distributions to those of $\xi=\sigma_p B$, where $B$ is a standard Brownian motion and $\sigma_p^2\doteq p!\sum_{v\in\mathbb{Z}}\rho(v)^p\ge0;$ see \cite{breuer1983central,nourdin2020functional}.

We verify the remaining assumptions of Theorem \ref{thm:KandC}.
Realize $X_k=W(h_k)$ for an isonormal Gaussian process over $\mfh$, with
$\langle h_k,h_\ell\rangle_{\mfh}=\rho(k-\ell)$.
Let $v_{n,k}\in L^2([0,1])$ be the indicator function of $[k/n,1]$.
Since $H_p(X_k)=I_p(h_k^{\otimes p})$, we have
\begin{equation}\label{eq:CK-Hermite-kernels}
    \xi_n=\mathcal{I}_p(f_{n,p}),
    \qquad
    f_{n,p}\doteq\frac1{\sqrt n}
    \sum_{k=1}^n h_k^{\otimes p}\otimes v_{n,k}
    \in\mfh^{\odot p}\otimes L^2([0,1]).
\end{equation}
Moreover, since $\EE[H_p(X_k)H_p(X_\ell)]=p!\rho(k-\ell)^p$, 
\begin{align}
    \EE|\xi_n(t)|^2
    &=\frac{p!}{n}
    \sum_{k,\ell=1}^{\lfloor nt\rfloor}\rho(k-\ell)^p
    =p!\sum_{v\in\mathbb{Z}}
    \frac{(\lfloor nt\rfloor-|v|)_+}{n}\rho(v)^p
    \to  t\sigma_p^2,
    \label{eq:CK-Hermite-pointwise-variance}
\end{align}
where $a_+\doteq\max\{a,0\}$ and the convergence follows from
\eqref{eq:CK-BM-summability} and dominated convergence.
Also,
\[
    \sup_{n\ge1}\sup_{t\in[0,1]}\EE|\xi_n(t)|^2
    \le p!\sum_{v\in\mathbb{Z}}|\rho(v)|^p<\infty,
\]
implying Condition \ref{item:kadelka1}. A further application of dominated convergence gives
\[
    \EE\|\xi_n\|_{L^2([0,1])}^2
    =\int_0^1\EE|\xi_n(t)|^2\,dt
    \to \frac{\sigma_p^2}{2}
    =\EE\|\sigma_p B\|_{L^2([0,1])}^2.
\]
This verifies Condition \ref{item:kadelka2}, and Theorem \ref{thm:KandC} applies.

A concrete statistical example is obtained with $p=2$, i.e.,
$H_2(X_k)=X_k^2-1$ and by taking for the Gaussian series $\{X_k\}$ to be the stationary Gaussian autoregressive process
\[
    X_k=\theta X_{k-1}+\sqrt{1-\theta^2}\,\eta_k,
    \qquad |\theta|<1,
\]
where the innovations $\{\eta_k\}_{k\in\mathbb{Z}}$ are independent standard Gaussian random variables. Then \eqref{eq:CK-Hermite-partial-sums} is the cumulative centered sum of squares of the observations and the empirical variance estimator. For this example, we employ Remark \ref{remark-quanti}(\ref{remark-quanti-fixed-chaos-weighted}) and derive the bounds for the right-hand side of \eqref{eq:fixed-chaos-weighted-bound} to get an explicit $n^{-1/2}$ bound. Note that $\rho(v)=\theta^{|v|}$, Condition \eqref{eq:CK-BM-summability} holds, and
\begin{equation}
\label{ex:eq:example41-eq-sigma2}
    \sigma_2^2
    =2\sum_{v\in\mathbb{Z}}\theta^{2|v|}
    =\frac{2(1+\theta^2)}{1-\theta^2}.
\end{equation}

We first prove a bound on the contractions in \eqref{eq:fixed-chaos-weighted-bound}.
Choose the orthonormal basis of $\clh=L^2([0,1])$ given by $e_j(t)=\sqrt2\sin(\omega_jt)$, with $\omega_j=\pi(j-\tfrac12)$. The kernel representation \eqref{eq:CK-Hermite-kernels} gives \[ f_{n,2,j}\doteq\langle f_{n,2},e_j\rangle_\clh =\frac1{\sqrt n}\sum_{k=1}^n a_{n,j,k}h_k^{\otimes2}, \qquad a_{n,j,k}=\frac{\sqrt2}{\omega_j} \cos\!\left(\frac{\omega_jk}{n}\right), \] so that $|a_{n,j,k}|\leq C/j$ for some constant $C>0$. Then, \eqref{eq:fixed-chaos-weighted-eta} can be bounded as  \begin{align*} 
\eta_{n,j} &\leq\frac{C}{n^2j^4} \sum_{k,\ell,u,v=1}^n 
|\theta|^{|k-\ell|} |\theta|^{|u-v|} |\theta|^{|k-u|} |\theta|^{|\ell-v|}
\leq\frac{C}{nj^4} 
\left( \sum_{v\in\mathbb Z} |\theta|^{|v|} \right)^3
\leq\frac{C_\theta}{nj^4}
\end{align*} 
for some constant $C_\theta$ that only depends on $\theta$ and may change from below.
The second inequality follows by dropping $|\theta|^{|u-v|} \leq1$ and summing the remaining three factors successively. Since $\sum_{j\geq1}j^{-4/3}<\infty$, we obtain \begin{equation}\label{eq:CK-AR-contraction-rate} \frac12\left(\sum_{j=1}^{\infty}\eta_{n,j}^{1/3}\right)^{3/2} \leq\frac{C_\theta}{\sqrt n} \left(\sum_{j=1}^{\infty}j^{-4/3}\right)^{3/2} \leq\frac{C_\theta}{\sqrt n}. 
\end{equation} 

To bound the trace-class norm on the covariance operator difference in \eqref{eq:fixed-chaos-weighted-bound}, we recall the variance \eqref{ex:eq:example41-eq-sigma2} and $v_{n,k}$ from \eqref{eq:CK-Hermite-kernels} to 
extend $v_{n,k}$ by zero for $k>n$, and define 
\[ 
\clt_{\xi_n}=2\sum_{v\in\mathbb Z} |\theta|^{2|v|} K_{n,v},
\qquad
K_{n,v}\doteq\frac1n\sum_{k=1}^n v_{n,k+v}\otimes v_{n,k},\qquad v\geq0, \qquad K_{n,-v}=K_{n,v}^*,\quad v\geq1. \] 
Moreover, writing $u_s=\mathbf1_{[s,1]}$ for $s\in[0,1]$, \[ \clt_B=\int_0^1u_s\otimes u_s\,ds, \qquad K_{n,0}=\frac1n\sum_{k=1}^n u_{k/n}\otimes u_{k/n}. \] 
Then, for every $v\in\mathbb Z$, 
\begin{align}
\|K_{n,v}-K_{n,0}\|_{\cls_1(\clh)} &\leq\frac1n\sum_{k=1}^n \|v_{n,k+|v|}-v_{n,k}\|_\clh\|v_{n,k}\|_\clh \leq\sqrt{\frac{|v|}{n}}, \label{eq:op-trac-ex41-1}
\\
\|K_{n,0}-\clt_B\|_{\cls_1(\clh)} &\leq
2\sum_{k=1}^n\int_{(k-1)/n}^{k/n} \sqrt{k/n-s}\,ds \leq\frac{C}{\sqrt n},
\label{eq:op-trac-ex41-2}
\end{align}
where we used for \eqref{eq:op-trac-ex41-1} the rank-one identity $\|u\otimes w\|_{\cls_1(\clh)}=\|u\|_\clh\|w\|_\clh$ and invariance of the trace norm under adjoints.
Since $\|u_s\|_\clh\leq1$ and $\|u_s-u_t\|_\clh=|s-t|^{1/2}$, the trace-norm Riemann-sum estimate gives \eqref{eq:op-trac-ex41-2}. 
Then, \begin{equation}\label{eq:CK-AR-covariance-rate} \begin{aligned} \frac12\|\clt_{\xi_n}-\sigma_2^2\clt_B\|_{\cls_1(\clh)} &\leq\sum_{v\in\mathbb Z} |\theta|^{2|v|}  \|K_{n,v}-\clt_B\|_{\cls_1(\clh)} \leq\frac{C}{\sqrt n} \sum_{v\in\mathbb Z}(1+\sqrt{|v|}) |\theta|^{2|v|} \leq\frac{C_\theta}{\sqrt n}, 
\end{aligned} \end{equation} 
where the last series is finite because $|\theta|<1$. Combining \eqref{eq:CK-AR-contraction-rate} and \eqref{eq:CK-AR-covariance-rate} with the bound in \eqref{eq:fixed-chaos-weighted-bound}, we may conclude that there exists a finite constant $C_\theta$, depending only on $\theta$, such that \begin{equation}\label{eq:CK-AR-root-n} d_2(\xi_n,\sigma_2B)\leq\frac{C_\theta}{\sqrt n}, \qquad n\geq1. \end{equation}
\end{example}

\subsection{Central limit theorem for the Kernel Ridge Regression estimator}
\label{subsec:KRR}

In this section, we prove a functional CLT for the Kernel Ridge Regression
(KRR) estimator, viewed as a random element in its ambient reproducing
kernel Hilbert space (RKHS). The RKHS framework provides a principled way
to construct nonlinear estimators for complex systems, including dynamical
systems and general regression problems, by embedding data into a
high-dimensional feature space implicitly defined by a positive definite
kernel. This connects tools from functional analysis and statistical
learning theory, and underlies nonlinear methods such as support vector
machines and kernel principal component analysis; see
\cite{steinwart2008support} for a comprehensive introduction.

We consider a fixed-design regression model, i.e., a setting in which the
covariates \(x_1,\dots,x_n\) are deterministic and the responses satisfy
\begin{equation} \label{eq:reg-model}
Y_i = f_0(x_i) + \varepsilon_i, \qquad i=1,\dots,n,
\end{equation}
where the noise sequence \(\{\varepsilon_i\}_{i\ge 1}\) consists of
independent and identically distributed $\RR$-valued random variables.
The objective here is to estimate $f_0$ as an element in some RKHS.
More precisely, assume that
\begin{equation} \label{eq:noise-ass}
\EE[\varepsilon_i]=0, \qquad
\EE[\varepsilon_i^2]=\sigma^2<\infty.
\end{equation}
To estimate the function \(f_0\) in \eqref{eq:reg-model}, we employ a
regularized least-squares estimator, also known as KRR. Let
\((\mathcal{H},\langle\cdot,\cdot\rangle_{\mathcal{H}})\) be the separable
RKHS associated with a bounded positive definite kernel \(k\) on
\(\mathcal{X}\subset\mathbb{R}\), let \(K_x(\cdot)=k(\cdot,x)\) denote the
corresponding reproducing kernels, and suppose that $f_0\in\clh$.
For \(\lambda>0\), the KRR estimator is defined as the solution of
\begin{equation} \label{eq:KRR-OF}
\widehat f_{n,\lambda}
 \doteq
\arg\min_{f\in\mathcal{H}}
\Big\{
\frac1n\sum_{i=1}^n (Y_i-f(x_i))^2
+\lambda \|f\|_{\mathcal{H}}^2
\Big\}.
\end{equation}
As given in, e.g., \cite{smale2005shannon} or Theorem 1 in
\cite{scholkopf2001generalized}, by the representer theorem, the solution
to \eqref{eq:KRR-OF} has the explicit form
\begin{equation} \label{eq:KRR-explicit}
\widehat f_{n,\lambda}
    = \frac{1}{n}(\Gamma_n+\lambda I)^{-1} S_n^{*} Y.
\end{equation}
Here \(Y=(Y_1,\dots,Y_n)^\top\), and the sampling operator
\(S_n:\mathcal{H}\to\mathbb{R}^n\) is defined by
\(S_n f=(f(x_1),\dots,f(x_n))^\top\).
Its adjoint is denoted by \(S_n^*:\mathbb{R}^n\to\mathcal{H}\), and
\(\Gamma_n=\tfrac{1}{n}S_n^*S_n\) is the empirical covariance operator.
Similarly, we introduce
\(\varepsilon=(\varepsilon_1,\dots,\varepsilon_n)^\top\) such that
\eqref{eq:reg-model} can be equivalently written as
$Y=(f_0(x_1),\dots,f_0(x_n))^\top+\varepsilon$.

We show that, under suitable conditions on the design, the centered and
rescaled estimator satisfies a central limit theorem in the Hilbert space
$\clh$. Asymptotic normality for linear functionals of KRR estimators is
studied in \cite{tuo2024asymptotic}, while Gaussian approximations in RKHSs
and their use for inference are considered in \cite{singh2023kernel}.

The qualitative convergence below can also be obtained from classical
central limit theorems for independent Hilbert space-valued triangular
arrays; see, e.g., \cite{kundu2000central}. Our purpose here is to show how
Theorem~\ref{thm:abstract-fourth-moment-theorem-infinite-chaos} applies under the finite-variance
assumption \eqref{eq:noise-ass}, and how
Theorem~\ref{thm:quantitative-CLT} provides quantitative $d_2$-bounds
in terms of covariance operators, contractions, and chaos tails. In the fixed-design setting with fixed $\lambda>0$ considered here, this quantitative approximation applies to arbitrary
centered, independent and identically distributed errors with finite
variance, without requiring Gaussianity, boundedness, or moments of
order greater than two. It therefore extends the results of \cite{singh2023kernel} discussed above by accommodating error distributions outside their bounded-error setting.

For fixed $\lambda>0$, define
\begin{equation} \label{eq:KRR-fn}
F_n
\doteq \sqrt n
\big(\widehat f_{n,\lambda}-A_n\Gamma_n f_0\big)
= A_n\frac1{\sqrt n}S_n^*\varepsilon\in\clh,
\qquad
A_n\doteq(\Gamma_n+\lambda I)^{-1},
\end{equation}
and the population analog of \eqref{eq:KRR-explicit}
$f_{\lambda}\doteq(\Gamma+\lambda I)^{-1}\Gamma f_0$.

\begin{proposition}\label{lem:krr-clt-fixed-lambda}
Suppose \eqref{eq:reg-model} with independent and identically distributed errors that
satisfy \eqref{eq:noise-ass}. Suppose further that
$\sup_{x\in\mathcal{X}}\|K_x\|_{\clh}^2\le C_k<\infty$ and that
$\|\Gamma_n-\Gamma\|_{\mathcal{S}_1(\clh)}\to0$, as $n\to\infty$, for some
nonnegative trace-class operator $\Gamma$ on $\clh$. Then, letting
$\{F_n\}_n$ be as in \eqref{eq:KRR-fn},
\[
F_n\xrightarrow{d}Z_{\lambda}
\quad\text{in }\clh,
\]
where $Z_{\lambda}$ is a centered Gaussian random variable in $\clh$ with covariance
operator
\[
\clt_Z
=
\sigma^2(\Gamma+\lambda I)^{-1}
\Gamma(\Gamma+\lambda I)^{-1}.
\]
If, in addition,
$\sqrt n\|A_n\Gamma_n f_0-f_{\lambda}\|_{\clh}\to 0$,
then
$\sqrt n\big(\widehat f_{n,\lambda}-f_{\lambda}\big)\xrightarrow{d} Z_{\lambda}$
in $\clh$.
\end{proposition}

\begin{proof}
We first write the KRR fluctuation as an infinite Wiener--It\^o chaos
expansion. Recalling \eqref{eq:KRR-explicit}, we have
\begin{equation} \label{eq:An-fnl}
\widehat f_{n,\lambda}-A_n\Gamma_n f_0
=
A_n\frac1nS_n^*(Y-S_n f_0)
=
A_n\frac1nS_n^*\varepsilon.
\end{equation}
Note also that, for any $f\in\clh$ and
$c=(c_1,\dots,c_n)^\top\in\RR^n$, the reproducing property gives
\begin{align*}
\langle f,S_n^*c\rangle_{\clh}
&=
\sum_{i=1}^n f(x_i)c_i
=
\sum_{i=1}^n\langle f,K_{x_i}\rangle_{\clh}c_i
=
\Big\langle f,\sum_{i=1}^n c_iK_{x_i}\Big\rangle_{\clh},
\end{align*}
implying that $S_n^*c=\sum_{i=1}^n c_iK_{x_i}$.
Therefore, in view of \eqref{eq:An-fnl}, we can recast
\eqref{eq:KRR-fn} as
\begin{equation}\label{eq:KRR-weighted-sum}
F_n=\sum_{i=1}^n a_{n,i}\varepsilon_i,
\qquad
a_{n,i}\doteq A_n\frac{K_{x_i}}{\sqrt n}\in\clh.
\end{equation}

For the common marginal distribution $F_{\varepsilon}$ of $\{\varepsilon_i\}_{i\geq1}$, define $\psi \doteq F_{\varepsilon}^{-1}\circ\Phi$, where $F_{\varepsilon}^{-1}$ is the quantile function of $F_{\varepsilon}$ and $\Phi$ is the standard Gaussian distribution function. Then, if $G$ is standard Gaussian, $\psi(G)$ has the same distribution as $\varepsilon_1$; see (5.8.6) in \cite{pipiras2017long}.
Since the errors are independent and identically
distributed, and the quantities considered here depend only on their
joint law, we may work on a Gaussian probability space on
which
$$
\varepsilon_i=\psi(G_i),
\qquad
G_i=W(h_i),
\quad i\ge1,
$$
where $W$ is an isonormal Gaussian process over $\mfh=\ell^2(\NN)$
and $\{h_i\}_{i\ge1}$ is its canonical orthonormal basis. In
particular, the $G_i$'s are independent standard Gaussian random
variables.

Since $\varepsilon_i\in L^2(\Omega)$ and $\EE[\varepsilon_i]=0$, its
Wiener--It\^o chaos decomposition can be written as
\begin{equation}\label{eq:KRR-abstract-chaos}
\varepsilon_i
=
\sum_{r=1}^{\infty}I_r(g_{r,i})
\qquad\text{in }L^2(\Omega),
\end{equation}
for suitable kernels $g_{r,i}\in\mfh^{\odot r}$. 
Moreover, since $\varepsilon_i$ is measurable with respect to
$\sigma(G_i)=\sigma(W(h_i))$, each of its chaos projections belongs to
the Wiener chaos generated by the one-dimensional Gaussian subspace
$\operatorname{span}\{h_i\}$. Hence,
\begin{equation}\label{eq:KRR-kernel-subspace}
g_{r,i}\in \operatorname{span}\{h_i^{\odot r}\},
\qquad r\ge1.
\end{equation}
We now have that, since the $\varepsilon_i$'s are identically distributed and from the isometry and orthogonality of multiple Wiener--It\^o integrals,
\begin{equation}\label{eq:KRR-chaos-kernel-norm}
q_r
= r!\|g_{r,i}\|_{\mfh^{\otimes r}}^2, \quad \sum_{r=1}^{\infty}q_r
=\EE[\varepsilon_1^2]
=
\sigma^2.
\end{equation}

Combining \eqref{eq:KRR-weighted-sum} and
\eqref{eq:KRR-abstract-chaos}, we obtain the $\clh$-valued chaos
expansion
\begin{equation}\label{eq:KRR-infinite-chaos}
F_n
=
\sum_{r=1}^{\infty}\cli_r(f_{n,r}),
\qquad
f_{n,r}
\doteq
\sum_{i=1}^n g_{r,i}\otimes a_{n,i}
\in\mfh^{\odot r}\otimes\clh,
\end{equation}
where the series converge in $L^2(\Omega:\clh)$. In addition, by \eqref{eq:noise-ass} and
the uncorrelatedness of the $\varepsilon_i$'s, it holds that
$\clt_{F_n}=\sigma^2A_n\Gamma_nA_n$. To see this, choose
$u,v\in\clh$, and
\begin{align*}
\langle\clt_{F_n}u,v\rangle_{\clh}
&=
\frac1n\EE\big[
\langle\varepsilon,S_nA_nu\rangle_{\RR^n}
\langle\varepsilon,S_nA_nv\rangle_{\RR^n}
\big]
\\&=
\frac1n(S_nA_nu)^\top
\EE[\varepsilon\varepsilon^\top]S_nA_nv
\\
&=
\frac{\sigma^2}{n}
\langle S_nA_nu,S_nA_nv\rangle_{\RR^n}
=
\langle\sigma^2A_n\Gamma_nA_nu,v\rangle_{\clh}.
\end{align*}

We verify Conditions \ref{item1:thm-infinite-expansions-covar-oper}--\ref{item3:thm-infinite-expansions-tail} of
Theorem~\ref{thm:abstract-fourth-moment-theorem-infinite-chaos}.
We start with Condition \ref{item1:thm-infinite-expansions-covar-oper}. Write $A_0(\Gamma)=(\Gamma+\lambda I)^{-1}$. We decompose
\begin{equation}\label{eq:An-decomp}
\begin{aligned}
A_n\Gamma_nA_n-A_0(\Gamma)\Gamma A_0(\Gamma)
=
(A_n-A_0(\Gamma))\Gamma_nA_n
+
A_0(\Gamma)(\Gamma_n-\Gamma)A_n
+
A_0(\Gamma)\Gamma(A_n-A_0(\Gamma)).
\end{aligned}
\end{equation}
Taking the trace-class norm and using the ideal property
$\|XYZ\|_{\mathcal{S}_1(\clh)}
\le\|X\|_{\op}\|Y\|_{\mathcal{S}_1(\clh)}\|Z\|_{\op}$,
we obtain
\begin{equation}\label{eq:cov-bound}
\begin{aligned}
&
\|A_n\Gamma_nA_n-A_0(\Gamma)\Gamma A_0(\Gamma)
\|_{\mathcal{S}_1(\clh)}
\\
&\le
\left(
\sup_n\|\Gamma_n\|_{\mathcal{S}_1(\clh)}
\sup_n\|A_n\|_{\op}
+
\|\Gamma\|_{\mathcal{S}_1(\clh)}\|A_0(\Gamma)\|_{\op}
\right)
\|A_n-A_0(\Gamma)\|_{\op}
\\
&\qquad+
\sup_n\|A_n\|_{\op}\|A_0(\Gamma)\|_{\op}
\|\Gamma_n-\Gamma\|_{\mathcal{S}_1(\clh)}.
\end{aligned}
\end{equation}
The constants $\sup_n\|A_n\|_{\op}$,
$\|A_0(\Gamma)\|_{\op}$, and
$\sup_n\|\Gamma_n\|_{\mathcal{S}_1(\clh)}$ are finite because
$\|A_n\|_{\op}\le\lambda^{-1}$,
$\|A_0(\Gamma)\|_{\op}\le\lambda^{-1}$, and each $\Gamma_n$ is
trace-class with
\[
\|\Gamma_n\|_{\mathcal{S}_1(\clh)}
=
\tr(\Gamma_n)
=
\frac1n\sum_{i=1}^n\|K_{x_i}\|_{\clh}^2
\le C_k.
\]
To estimate $\|A_n-A_0(\Gamma)\|_{\op}$, we first use the resolvent identity $A_n-A_0(\Gamma)
=
A_n(\Gamma-\Gamma_n)A_0(\Gamma)$. Then, taking operator norms and using $\|A_n\|_{\op}\le\lambda^{-1}$
and $\|A_0(\Gamma)\|_{\op}\le\lambda^{-1}$ gives
\begin{equation}\label{eq:AnR0-diff}
\|A_n-A_0(\Gamma)\|_{\op}
\le
\lambda^{-2}\| \Gamma_n-\Gamma \|_{\mathcal{S}_1(\clh)}.
\end{equation}

For each $r\ge1$, let
\begin{equation}\label{eq:KRR-chaos-covariance}
F_{n,r}
\doteq
\cli_r(f_{n,r})
\qquad 
\clt_{F_{n,r}}
=
q_rA_n\Gamma_nA_n.
\end{equation}
To see the form of the covariance operators, recall that, by \eqref{eq:KRR-kernel-subspace}, the random variables $\cli_r(g_{r,i})$ are measurable with respect to the mutually independent Gaussian random variables $G_i$. Hence, for fixed $r\geq1$, the random variables ${\cli_r(g_{r,i})}_{i\geq1}$ are independent and centered and, by \eqref{eq:KRR-chaos-kernel-norm}, have common variance $q_r$.

Let $\{Z_r\}_{r\ge1}$ be independent centered Gaussian random
variables in $\clh$ with respective covariance operators $\clt_{Z_r}
\doteq
q_rA_0(\Gamma)\Gamma A_0(\Gamma)$. Then
$$
\|\clt_{F_{n,r}}-\clt_{Z_r}\|_{\mathcal{S}_1(\clh)}
=
q_r
\|A_n\Gamma_nA_n-A_0(\Gamma)\Gamma A_0(\Gamma)
\|_{\mathcal{S}_1(\clh)}
\to0
$$
for every $r\ge1$, which verifies Condition \ref{item1:thm-infinite-expansions-covar-oper}. Moreover, by
\eqref{eq:KRR-chaos-kernel-norm},
$$
\sum_{r=1}^{\infty}\clt_{Z_r}
=
\sigma^2A_0(\Gamma)\Gamma A_0(\Gamma)
=
\clt_Z,
\qquad
\|\clt_Z\|_{\mathcal{S}_1(\clh)}
\le
\sigma^2\lambda^{-2}\|\Gamma\|_{\mathcal{S}_1(\clh)}
<\infty,
$$
so Condition \ref{item1.1:thm-infinite-expansions-trace-sec-mom} also holds.

To verify the contraction Condition \ref{item2:thm-infinite-expansions-contr}, fix an orthonormal basis
$\{e_j\}_{j\in\NN}$ of $\clh$ and define
\begin{equation}\label{eq:def-fnj}
f_{n,r,j}
\doteq
\langle f_{n,r},e_j\rangle_{\clh}
=
\sum_{i=1}^n\alpha_{n,j,i}g_{r,i},
\qquad
\alpha_{n,j,i}
\doteq
\langle A_nK_{x_i}/\sqrt n,e_j\rangle_{\clh}.
\end{equation}
Fix $r\ge2$ and $m\in \{1,\dots,r-1\}$. Since
$g_{r,i}\in\operatorname{span}\{h_i^{\odot r}\}$ and the $h_i$'s
are orthonormal, it holds that $g_{r,i}\otimes_m g_{r,\ell}=0$ for $i \neq \ell$.
Moreover, the tensors
$g_{r,i}\otimes_mg_{r,i}$ corresponding to distinct $i$ are
orthogonal. Hence
\begin{align}
\|f_{n,r,j}\otimes_mf_{n,r,j}\|
_{\mfh^{\otimes(2r-2m)}}^2
&=
\sum_{i=1}^n
\alpha_{n,j,i}^4
\|g_{r,i}\otimes_mg_{r,i}\|
_{\mfh^{\otimes(2r-2m)}}^2
\label{eq:fnj-contraction-abstract}
\\&=
\sum_{i=1}^n
\alpha_{n,j,i}^4
\|g_{r,i}\|
_{\mfh^{\otimes r}}^4
\\&=
\left(\frac{q_r}{r!}\right)^2
\sum_{i=1}^n\alpha_{n,j,i}^4,
\end{align}
where the last identity follows from
\eqref{eq:KRR-chaos-kernel-norm}.
Using $\|A_n\|_{\op}\le\lambda^{-1}$ and
$\|K_{x_i}\|_{\clh}^2\le C_k$, we obtain
\begin{equation}\label{eq:KRR-coefficient-bounds}
\max_{1\le i\le n}|\alpha_{n,j,i}|
\le
\frac{\sqrt{C_k}}{\sqrt n \lambda},
\qquad
\sum_{i=1}^n\alpha_{n,j,i}^4
\le
\frac{C_k^2}{n\lambda^4}.
\end{equation}
It follows that
$$
\|f_{n,r,j}\otimes_mf_{n,r,j}\|
_{\mfh^{\otimes(2r-2m)}}^2
\le
\left(\frac{q_r}{r!}\right)^2
\frac{C_k^2}{n\lambda^4}
\to0
$$
for all $j\in\NN$, $r\ge2$, and
$m\in \{1,\dots,r-1\}$. This verifies Condition \ref{item2:thm-infinite-expansions-contr}.

It remains to verify the uniform tail Condition \ref{item3:thm-infinite-expansions-tail}. First, note
that
\begin{equation}\label{eq:KRR-weight-square-sum}
\sum_{i=1}^n \|a_{n,i}\|_{\clh}^2
=
\tr(A_n\Gamma_nA_n)
\le
\lambda^{-2}\tr(\Gamma_n)
\le
\frac{C_k}{\lambda^2}.
\end{equation}
By \eqref{eq:def-fnj}, the mutual orthogonality of the Gaussian
subspaces $\operatorname{span}{h_i}$, the isometry of multiple
Wiener--It\^o integrals, and Parseval's identity,
\begin{align}
\sum_{r=N+1}^{\infty}\sum_{j=1}^{\infty}
r! \|f_{n,r,j}\|_{\mfh^{\otimes r}}^2
&=
\sum_{r=N+1}^{\infty}
q_r
\sum_{j=1}^{\infty}\sum_{i=1}^n\alpha_{n,j,i}^2
\nonumber\\
&=
\sum_{r=N+1}^{\infty}
q_r
\sum_{i=1}^n \|a_{n,i}\|_{\clh}^2
\nonumber\\
&=
\tr(A_n\Gamma_nA_n)
\sum_{r=N+1}^{\infty}q_r
\le
\frac{C_k}{\lambda^2}
\sum_{r=N+1}^{\infty}q_r.
\label{eq:KRR-chaos-tail}
\end{align}
By \eqref{eq:KRR-chaos-kernel-norm}, the last quantity tends to zero
as $N\to\infty$, uniformly in $n$. Hence Condition \ref{item3:thm-infinite-expansions-tail} holds.

All the assumptions of
Theorem~\ref{thm:abstract-fourth-moment-theorem-infinite-chaos} are
therefore satisfied, and we obtain
$
F_n\xrightarrow{d}Z_{\lambda}$
 in $\clh$.

\end{proof}

\subsection{Approximation and ergodicity of the stochastic heat equation} \label{subsec:spde-general}

In this section, we consider the 1-dimensional stochastic heat equation, viewed as an evolution equation on 
$\clh = L^2([0,1],dx)$ with homogeneous Dirichlet boundary conditions. 
First, we study the spectral Galerkin approximation of its solutions, and the associated weak errors. 
Obtaining rates of convergence for weak errors is a well-known and challenging problem; see, e.g., \cite{ConJenKur19}. 
Here, we derive such rates for test functions $\phi$ belonging to the class $\clc_b^2(\clh)$, and for every approximation order $n$. Then, we study the problem of approximating the invariant measure when initiating the SPDE from an element of $L^2(\Omega:\clh)$. We also obtain an explicit exponential bound on the $d_2$ distance to the invariant measure. Note that, the bounds we obtain hold, in particular, for deterministic initial conditions.

Although the framework considered here is elementary, we believe that several generalizations are possible. 
For instance, with regard to the results of Section \ref{subsec:spde} below, one could consider joint approximations in time and space (e.g., Euler--Maruyama in time and Galerkin in space), 
Gaussian, nonzero, initial conditions, or even other semilinear SPDEs, as long as their law remains Gaussian on $\clh$. These extensions are outside the scope of the present article. For the linear equation considered here, Section \ref{subsubsec:spde-invariant} gives an explicit exponential rate of convergence in $d_2$.

The framework is as follows. Define the Laplace operator
\[
A u \doteq  \frac{d^2u}{dx^2}, 
\qquad D(A) \doteq  H^2(0,1) \cap H_0^1(0,1),
\]
where, for $p \in \NN$, $H^p(0,1) = \bigl\{  u \in L^2(0,1) : D^k u \in L^2(0,1) \text{ for } 0 \le k \le p  \bigr\}$ with $D^k u$ denoting the weak $k-$th derivative of $u$, and $H_0^1(0,1) \doteq \{u \in H^1(0,1): u(0) = u(1) = 0\}$.
Then $A$ is a self-adjoint, negative definite, unbounded operator on $\clh$, generating the strongly continuous semigroup $S(t) = e^{tA}, t \ge 0$.
We consider the stochastic evolution equation
\begin{equation} 
du(t) = A u(t) dt + Q^{1/2} dW_t, 
\qquad u(0) = u_0,
\label{eq:spde}
\end{equation}
where $W_t$ is a cylindrical Wiener process on $\clh$ and $Q \in \cls_1(\clh)$ is a symmetric, nonnegative, trace-class covariance operator; see Section 6 in \cite{da2002second} and Examples 10.17 and 10.20 in \cite{Lord_Powell_Shardlow_2014}, but note the slightly different terms arising in the latter due to considering $\clh = L^2(0,\pi)$. In the formulation \eqref{eq:spde}, we assume that $u_0\in L^2(\Om:\clh)$ and that $u_0$ and $W$ are independent.

The Dirichlet Laplacian admits the orthonormal eigenbasis and eigenvalues
\[
e_k(x) \doteq \sqrt{2}\sin(k\pi x), \quad A e_k \doteq -\lambda_k e_k, \quad \lambda_k \doteq (k \pi)^2
, \quad k\ge1,
\]
and we assume that the operator $Q$ admits the same eigenbasis, i.e., $Q e_k = q_k e_k$, where $q_k\ge0$ for all $k$. Since $Q \in \cls_1(\clh),$ it follows that  $\sum_{k=1}^\infty q_k < \infty$ and since $
\int_0^t \| S(t-s) Q^{1/2}\|_{\cls_2(\clh)}^2 ds < \infty$, the stochastic integral
is well-defined, giving rise to the mild solution of \eqref{eq:spde}
\begin{equation} \label{eq:mild-sol-def}
u(t) = S(t) u_0 +  \int_0^t S(t-s) Q^{1/2} dW_s, \qquad t \ge 0.
\end{equation}
When $\EE u_0 = 0$, the process $\{u(t)\}_{t\ge0}$ in \eqref{eq:mild-sol-def} is a centered random variable in $\clh$. When, in addition, $u_0 = 0$, the process is Gaussian.

Expanding $u(t) = \sum_{k=1}^\infty u_k(t) e_k$, where $u_k(t) \doteq \langle u(t), e_k\rangle_\clh$, we obtain the family of Ornstein--Uhlenbeck ordinary stochastic differential equations
\[
du_k(t) = -\lambda_k u_k(t) dt + \sqrt{q_k} d\beta_k(t), \quad u_k(0) = \langle u_0, e_k \rangle_\clh,
\]
where $\{\beta_k\}_{k \in \NN}$ are independent standard Brownian motions. It is a standard fact (e.g., Example 6.2.11 in \cite{da2002second}) that, if $\sum_{k\ge1} q_k/\lambda_k < \infty$, then $\clt_{u(t)} \in \cls_1(\clh)$, and, moreover
\begin{equation} 
\clt_{u(t)} = S(t) \clt_{u_0} S(t)^* + \int_0^t S(t-s) Q S(t-s)^* ds = S(t) \clt_{u_0} S(t)^* + \sum_{k=1}^\infty \frac{q_k}{2\lambda_k}\big(1 - e^{-2\lambda_k t}\big)  e_k \otimes e_k.
\label{eq:covariance}
\end{equation}

\subsubsection{Weak convergence rates of spectral Galerkin approximations}
\label{subsec:spde}

Here we consider the evolution equation $\eqref{eq:mild-sol-def}$ with initial conditions $u_0 = 0$. Recall the notation $E_n = \mathrm{span}\{e_1, \dots, e_n\}$ and denote again by $P_n: \clh \to E_n$ the orthogonal projection operator defined above \eqref{eq:summands3-trace}. The spectral Galerkin approximation of \eqref{eq:spde} (see, (10.56) of \cite{Lord_Powell_Shardlow_2014}) is the solution to the $E_n$-valued evolution equation:
\begin{equation}
du_n(t) = A_n u_n(t) dt + P_n Q^{1/2} dW_t, 
\qquad u_n(0)=0,
\end{equation}
where $A_n \doteq P_n A|_{E_n}$. It admits the mild solution given by 
\[
u_n(t) = \int_0^t e^{(t-s)A_n} P_n Q^{1/2} dW_s
        = \sum_{k=1}^n \sqrt{q_k}\int_0^t e^{-\lambda_k (t-s)} d\beta_k(s)  e_k.
\]
For all $n$, the process $\{u_n(t)\}_{t \ge 0}$ is also Gaussian in $E_n$ with covariance operator
\begin{equation} \label{eq:covar-oper-galerking}
\clt_{u_n(t)} = P_n \clt_{u(t)} P_n
        = \sum_{k=1}^n \frac{q_k}{2\lambda_k}\big(1 - e^{-2\lambda_k t}\big)  e_k \otimes e_k.
\end{equation}
We view the law of $\{u_n(t)\}$ as the centered, degenerate Gaussian measure on $\clh$ with covariance operator \eqref{eq:covar-oper-galerking}. Thus, the spectral Galerkin method corresponds to truncating the infinite-dimensional Gaussian field to its first $n$ spectral modes.

Fix $\phi \in \clc_b^2(\clh)$.
The \emph{weak error} of the spectral Galerkin approximation at the terminal time $T$ is defined as
\begin{equation} \label{eq:weak-error-def}
e^{w,\phi}_n(T)
  = \big|\EE[\phi(u(T))] - \EE[\phi(u_n(T))]\big|.
\end{equation}

In the result below, we translate our quantitative estimates for the weak approximation of Gaussian measures to a general result for the weak error of the Galerkin approximation at the terminal time $T$.

\begin{proposition}
Let $\phi \in \clc_b^2(\clh)$, such that $\|\phi\|_{\clc_b^2(\clh)} \le 1$. Then,
\begin{equation}
    e^{w,\phi}_n(T) \le \frac{1}{4} \sum_{k=n+1}^\infty \frac{q_k}{\lambda_k}\big(1 - e^{-2\lambda_k T}\big).
\end{equation}
\end{proposition}

\begin{proof}
    By Corollary \ref{cor1},
    \begin{equation}
        e^{w,\phi}_n (T) \le d_2(u(T), u_n(T)) \le \frac{1}{2}
    \left\lVert \clt_{u(T)}  - \clt_{u_n(T)} \right\rVert_{\mathcal{S}_1(\clh)}.
    \end{equation}
    Recalling the form of $\clt_{u(T)}$ and $\clt_{u_n(T)}$ in \eqref{eq:covariance} and \eqref{eq:covar-oper-galerking}, we infer that
    \begin{equation}
        \clt_{u(T)} - \clt_{u_n(T)}  = \sum_{k=n+1}^\infty \frac{q_k}{2\lambda_k}\big(1 - e^{-2\lambda_k T}\big)  e_k \otimes e_k,
    \end{equation}
    which is a non-negative, self-adjoint operator. Hence, 
    \begin{equation}
        \left\lVert \clt_{u(T)}  - \clt_{u_n(T)} \right\rVert_{\mathcal{S}_1(\clh)} = \sum_{k=n+1}^\infty \frac{q_k}{2\lambda_k}\big(1 - e^{-2\lambda_k T}\big),
    \end{equation}
    which concludes the proof.
\end{proof}

\subsubsection{Approximation of the invariant measure}
\label{subsubsec:spde-invariant}

Consider again the solution to the stochastic heat equation given in its mild formulation in \eqref{eq:mild-sol-def}, with an arbitrary initial condition $u_0$ in $L^2(\Om:\clh)$.
Set $\bar u_0\doteq u_0-\EE u_0$. In this case, $u_0$ admits the Wiener chaos expansion
\begin{equation}\label{eq:u0-init-cond}
    \bar u_0=u_0-\EE u_0=\sum_{r=1}^\infty \cli_r(f_r),
    \quad
    f_r=\sum_{i=1}^\infty f_{r,i}\otimes e_i
    \in\mfh^{\odot r}\otimes\clh;
\end{equation}
see, e.g., Lemma 4.1 in \cite{dukZou24}.

Recall that $u_0$ is independent of the cylindrical Wiener process $W$. Then,
\begin{equation}\label{eq:linear-part-I1}
    \int_0^t S(t-s)Q^{1/2}dW_s
    =\cli_1(h_t),
    \quad
    h_t=\sum_{i=1}^\infty h_{t,i}\otimes e_i
    \in\mfh\otimes\clh,
\end{equation}
where the stochastic convolution belongs to the first chaos of the common Gaussian process. Define
\[
    \bar u(t)\doteq u(t)-S(t)\EE u_0.
\]
Since $\EE u(t)=S(t)\EE u_0$, the process $\bar u(t)$ is centered. We first apply the abstract convergence theorem to $\bar u(t)$ along $t\in\NN$; the quantitative estimate for $u(t)$ below is valid for every $t\ge0$.
Then by combining \eqref{eq:mild-sol-def}, \eqref{eq:u0-init-cond}, and \eqref{eq:linear-part-I1},
\begin{equation}\label{eq:ut-wiener-chaos}
    u(t)
    =S(t) \EE u_0+\sum_{r=1}^\infty S(t)\cli_r(f_r)+\cli_1(h_t) =S(t) \EE u_0 +\cli_1(h_t+\widetilde f_{t,1})
      +\sum_{r=2}^\infty\cli_r(\widetilde f_{t,r}),
\end{equation}
where the new kernels are given by
\begin{equation}\label{eq:ut-kernels-tilde}
    \widetilde f_{t,r}\doteq
    \sum_{i=1}^\infty e^{-\lambda_i t}f_{r,i}\otimes e_i
    \in\mfh^{\odot r}\otimes\clh,
    \quad r\ge1.
\end{equation}
All covariance operators below are centered covariance operators; in particular, $\clt_{u(t)}=\clt_{\bar u(t)}$.
From the uncorrelatedness of the Wiener chaoses and the independence of the initial condition from the cylindrical Wiener process,
\begin{equation}\label{eq:covar-oper-ut}
    \clt_{u(t)}
    =\sum_{r=1}^\infty\clt_{\cli_r(\widetilde f_{t,r})}
     +\sum_{k=1}^\infty\frac{q_k}{2\lambda_k}
       \big(1-e^{-2\lambda_k t}\big)e_k\otimes e_k.
\end{equation}
Let $u_\infty$ be the centered Gaussian random variable in $\clh$ with covariance operator
\begin{equation}\label{eq:covar-oper-u-infty}
    \clt_{u_\infty}
    =\sum_{k=1}^\infty\frac{q_k}{2\lambda_k}e_k\otimes e_k.
\end{equation}
Its law is invariant for \eqref{eq:spde}. Indeed, if the initial condition has this law and is independent of $W$, then the mild solution is centered Gaussian and, by \eqref{eq:covariance}, its covariance is
\[
    S(t)\clt_{u_\infty}S(t)^*+\clt_{\cli_1(h_t)}
    =\clt_{u_\infty},\qquad t\ge0.
\]
For a qualitative statement in this setting, see, e.g., Theorem 8.20 in \cite{da2002second} or Theorem 5.2.3 in \cite{DalangSanzSole2026} for a closely related statement.
The following proposition verifies the assumptions of our infinite-chaos theorem for the centered process and gives an explicit rate of convergence in the $d_2$ distance for arbitrary initial conditions.

\begin{proposition}\label{prop:she-invariant}
    Let $\{u(t)\}_{t\ge0}$ be the solution to the stochastic heat equation with arbitrary initial condition $u_0\in L^2(\Om:\clh)$ independent of $W$, given in \eqref{eq:ut-wiener-chaos}. Let $u_\infty$ be the centered Gaussian random variable in $\clh$ with covariance operator \eqref{eq:covar-oper-u-infty}. Then, for every $t\ge0$,
    \begin{equation}\label{eq:she-invariant-d2-bound}
        d_2(u(t),u_\infty)
        \le e^{-\lambda_1 t}\| \EE u_0 \|_{\clh} +\frac{e^{-2\lambda_1 t}}{2}
        \left(\EE\|u_0- \EE u_0 \|_{\clh}^2
        +\tr(\clt_{u_\infty})\right).
    \end{equation}
    In particular, $u(t)\to u_\infty$ with regard to the $d_2$ distance (and so in distribution) as $t\to\infty$.
\end{proposition}

\begin{proof}
We verify the conditions of Theorem \ref{thm:abstract-fourth-moment-theorem-infinite-chaos} for the centered process $\bar u(t)$.
The summability $\sum_k q_k/\lambda_k<\infty$ implies that the limiting covariance operator in \eqref{eq:covar-oper-u-infty} belongs to $\cls_1(\clh)$, yielding Condition \ref{item1.1:thm-infinite-expansions-trace-sec-mom}.
We verify Condition \ref{item1:thm-infinite-expansions-covar-oper}. For $r\ge2$, we set $\clt_{Z_r}=0$, and so, from \eqref{eq:ut-wiener-chaos},
\begin{multline}
    \|\clt_{\cli_r(\widetilde f_{t,r})}\|_{\cls_1(\clh)}
=\tr(\clt_{\cli_r(\widetilde f_{t,r})})
    =\sum_{i=1}^\infty    \langle\clt_{\cli_r(\widetilde f_{t,r})}e_i,e_i\rangle_{\clh}\\
=\sum_{i=1}^\infty\EE|I_r(\widetilde f_{t,r,i})|^2
    =\sum_{i=1}^\infty r!e^{-2\lambda_i t}
    \|f_{r,i}\|_{\mfh^{\odot r}}^2
    \le r!e^{-2\lambda_1 t}
      \|f_r\|_{\mfh^{\odot r}\otimes\clh}^2\to0,
\end{multline}
as $t\to\infty$. For $r=1$, we set $\clt_{Z_1}=\clt_{u_\infty}$, use the independence of $u_0$ from $W$, \eqref{eq:ut-kernels-tilde}, \eqref{eq:covar-oper-ut}, and \eqref{eq:covar-oper-u-infty} to get
\begin{equation}
\begin{split}
    \|\clt_{\cli_1(\widetilde f_{t,1}+h_t)}
      -\clt_{u_\infty}\|_{\cls_1(\clh)}
    &=\|\clt_{\cli_1(\widetilde f_{t,1})}
       +\clt_{\cli_1(h_t)}-\clt_{u_\infty}\|_{\cls_1(\clh)}\\
    &\le\|\clt_{\cli_1(\widetilde f_{t,1})}\|_{\cls_1(\clh)}
      +\left\|\sum_{k=1}^\infty\frac{q_k}{2\lambda_k}
       e^{-2\lambda_k t}e_k\otimes e_k\right\|_{\cls_1(\clh)}\\
    &\le e^{-2\lambda_1 t}\|f_1\|_{\mfh\otimes\clh}^2
      +e^{-2\lambda_1 t}\sum_{k=1}^\infty\frac{q_k}{2\lambda_k}\to0.
\end{split}
\end{equation}
The first quantity in the last line goes to $0$ with the same considerations as before. The convergence of the second quantity follows by the summability of the series.

For Condition \ref{item2:thm-infinite-expansions-contr}, fix $r\ge2$, $i\in\NN$, and $\ell\in\{1,\dots,r-1\}$. Then
\begin{equation}
    \|\widetilde f_{t,r,i}\otimes_\ell\widetilde f_{t,r,i}
      \|_{\mfh^{\otimes2(r-\ell)}}
    \le e^{-2\lambda_i t}\|f_{r,i}\|_{\mfh^{\otimes r}}^2\to0,
\end{equation}
as $t\to\infty$.
Finally, since
\[
    \sum_{r=1}^\infty r!\|f_r\|_{\mfh^{\otimes r}\otimes\clh}^2
    =\EE\|\bar u_0\|_{\clh}^2
    =\EE\|u_0- \EE u_0 \|_{\clh}^2<\infty,
\]
it follows that
\[
\begin{split}
    \sup_{t\in\NN}\sum_{r\ge N+1}
      r!\|\widetilde f_{t,r}\|_{\mfh^{\otimes r}\otimes\clh}^2
    =\sup_{t\in\NN}\sum_{r\ge N+1}\sum_{i=1}^\infty
      e^{-2\lambda_i t}r!\|f_{r,i}\|_{\mfh^{\otimes r}}^2\le\sum_{r\ge N+1}r!\|f_r\|_{\mfh^{\otimes r}\otimes\clh}^2\to0,
\end{split}
\]
which verifies Condition \ref{item3:thm-infinite-expansions-tail}.
Theorem \ref{thm:abstract-fourth-moment-theorem-infinite-chaos} therefore yields $d_2(\bar u(t),u_\infty)\to0$ along $t\in\NN$.

To obtain \eqref{eq:she-invariant-d2-bound}, fix $t\ge0$ and write $Y_t\doteq\cli_1(h_t)$.
Every test function in the unit ball defining $d_2$ has Lipschitz constant at most one. Hence, since $u(t)-\bar u(t)=S(t) \EE u_0$ is deterministic,
\begin{equation}\label{eq:she-invariant-mean-shift}
    d_2(u(t),\bar u(t))
    \le\|S(t) \EE u_0 \|_{\clh}
    \le e^{-\lambda_1 t}\| \EE u_0 \|_{\clh}.
\end{equation}

For $\phi\in\clc_b^2(\clh)$ with $\|\phi\|_{\clc_b^2(\clh)}\le1$, Taylor's formula and the bound on the second Fr\'echet derivative give
\begin{equation}\label{eq:jknoe}
\begin{split}
    |\phi(y+v)-\phi(y)-D_F\phi(y)[v]|
    \le\int_0^1(1-s)|D_F^2\phi(y+sv)[v,v]|\,ds
    \le\frac12\|v\|_{\clh}^2,
    \quad y,v\in\clh.
\end{split}
\end{equation}
Since $\bar u(t)=Y_t+S(t)\bar u_0$, we obtain
\begin{align}
    d_2(\bar u(t),Y_t)
    &=\sup_{\phi:\|\phi\|_{\clc_b^2(\clh)}\le1}
      |\EE[\phi(\bar u(t))]-\EE[\phi(Y_t)]|\nonumber\\
    &=\sup_{\phi:\|\phi\|_{\clc_b^2(\clh)}\le1}
      \left|\EE\left[\phi(Y_t+S(t)\bar u_0)-\phi(Y_t)
      -D_F\phi(Y_t)[S(t)\bar u_0]\right]\right|\nonumber\\
    &\le\frac12\EE\|S(t)\bar u_0\|_{\clh}^2.
    \label{eq:she-initial-data-taylor}
\end{align}
For the second line, we used the fact that $\bar u_0$ is centered and independent of $Y_t$, so that
$\EE[D_F\phi(Y_t)[S(t)\bar u_0]]=0$; for the last line, we used \eqref{eq:jknoe}.

On the other hand, Corollary \ref{cor1} and \eqref{eq:covar-oper-u-infty} yield
\begin{align}
    d_2(Y_t,u_\infty)
    &\le\frac12\|\clt_{Y_t}-\clt_{u_\infty}\|_{\cls_1(\clh)}
     =\frac14\sum_{k=1}^\infty\frac{q_k}{\lambda_k}e^{-2\lambda_k t}.
    \label{eq:she-convolution-invariant}
\end{align}
Combining \eqref{eq:she-invariant-mean-shift}, \eqref{eq:she-initial-data-taylor}, and \eqref{eq:she-convolution-invariant} with the triangle inequality for $d_2$ and Parseval's identity gives
\[
\begin{split}
    d_2(u(t),u_\infty)
    &\le\|S(t) \EE u_0 \|_{\clh}
      +\frac12\sum_{k=1}^\infty e^{-2\lambda_k t}
       \left(\EE|\langle u_0-\EE u_0,e_k\rangle_{\clh}|^2
       +\frac{q_k}{2\lambda_k}\right)\\
    &\le e^{-\lambda_1 t}\| \EE u_0 \|_{\clh}
      +\frac{e^{-2\lambda_1 t}}2
       \left(\EE\|u_0- \EE u_0 \|_{\clh}^2
       +\tr(\clt_{u_\infty})\right),
\end{split}
\]
since $\lambda_k\ge\lambda_1$ for all $k\ge1$. Both terms on the right-hand side tend to zero as $t\to\infty$, proving the claimed convergence for all real times.
\end{proof}

\section{Auxiliary Lemmas}

\begin{lemma}
  \label{lem:4}
  Let $h \in \clc^j_b(\clh)$ and suppose $\| h \|_{\clc_b^j(\clh)} \leq 1$, $j=1,2$. Then, the (Stein solution) operator $f_h$ defined in \eqref{eq:Stein-solution} admits $j$ Fr\'echet derivatives. Moreover, 
\begin{equation}
  \label{eq:8}
  \sup_{x \in \clh} \| \Deriv^j_F f_h(x) \|_{\op(\cll(\clh^{\otimes j}:\RR))}
  \leq
  \frac{1}{j}, \qquad j = 1,2.
\end{equation}
\end{lemma}

\begin{proof}
We begin by conjecturing the form of the Fr\'echet derivative $D_F f_h(x)$. We then verify, using the definition, that this is indeed the Fr\'echet derivative of $f_h$ whenever $h\in\clc_b^1(\clh)$.

First, set $g_u: \clh \to \clh, g_u(x) = e^{-u}x + \sqrt{1-e^{-2u}} y$, for $u \geq 0$ and fixed $y \in \clh$. Then, $D_F g_u(x)[z] = e^{-u}z$, i.e., for all $u \ge 0$, this is the multiplication operator associated to the constant function $e^{-u}$.
Since $h$ is Fr\'echet differentiable and by the chain rule for Fr\'echet derivatives (see p.\ 337 in \cite{Lang}), we get
\begin{align*}
&D_F\Big( h (g_u(x))\Big)[z]
= \Big( (D_F h)(g_u(x)) \Big) [D_F g_u(x)[z]]
\\&= \Big( (D_F h)(g_u(x)) \Big)[e^{-u}z] 
= e^{-u} \Big( (D_F h)(g_u(x)) \Big)[z].
\end{align*}
The last identity follows since $D_F h(x)[z]$ is linear in $z$.

By exchanging integration and Fr\'echet differentiation, this leads us to defining the candidate derivative
\begin{align}
  D_F f_h(x)[z]
  &\doteq \int_0^\infty e^{-u}\int_\clh
  (D_Fh)(e^{-u}x+\sqrt{1-e^{-2u}}y)[z]\,\gamma_Z(dy)\,du,
  \label{eq:Dfh-form}
\end{align}
We formally verify \eqref{eq:Dfh-form} by using the definition of the Fr\'echet derivative and showing
\begin{equation} \label{eq:frech}
    \lim_{||z||_{\clh} \to 0} \frac{|f_h(x+z) - f_h(x) -D_F f_h(x)[z]|}{||z||_{\clh}} = 0.
\end{equation}
First, observe that
\begin{multline} \label{eq:128}
    |f_h(x+z) - f_h(x) -D_F f_h(x)[z]| \\\le \int_0^\infty \int_{\clh} \bigg| h(e^{-u}(x+z) + \sqrt{1-e^{-2u}} y)  - h(e^{-u}x + \sqrt{1-e^{-2u}} y) \\
    - (D_F h)(e^{-u}x + \sqrt{1-e^{-2u}} y)[e^{-u}z]\bigg| \gamma_Z(dy) du.
\end{multline}
By the definition of the Fr\'echet differentiability of $h$ (cf. \eqref{eq:frech} with $h$ replacing $f_h$, $e^{-u}x + \sqrt{1-e^{-2u}} y$ replacing $x$, and $e^{-u} z$ replacing $z$), it follows that, $\gamma_Z \otimes du$ a.s.,
\begin{multline}
    \lim_{\|z\|_{\clh} \to 0} \frac{1}{\|z\|_{\clh}} \left(  \bigg| h(e^{-u}(x+z) + \sqrt{1-e^{-2u}} y)  - h(e^{-u}x + \sqrt{1-e^{-2u}} y)  \right. \\ \left.
    - (D_F h)(e^{-u}x + \sqrt{1-e^{-2u}} y)[e^{-u}z]\bigg|  \right) =0. \label{eq:129}
\end{multline}
Recalling $g_u$, we rewrite \eqref{eq:129} as
\begin{align*}
    &
    | h(g_u(x+z))  - h(g_u(x)) - (D_Fh)(g_u(x))[e^{-u}z]| 
    \\ &\le 
    | h(g_u(x+z))  - h(g_u(x))| 
        + | (D_Fh)(g_u(x))[e^{-u}z]| 
    \\ &\le 
    \|h \|_{\Lip} \|e^{-u} z\|_{\clh} 
    + \sup_{x \in \clh} \|(D_Fh)(x)\|_{\op(\cll(\clh:\RR))} \|e^{-u}z\|_{\clh} 
    \\ &\le 
    e^{-u} \left( \|h \|_{\Lip} \| z\|_{\clh} + \sup_{x \in \clh} \|(D_F h)(x)\|_{\op(\cll(\clh:\RR))} \|z\|_{\clh} \right)
      \le 2 e^{-u}  \|z\|_{\clh}.
\end{align*}
The second inequality follows by the fact that $h$ is Lipschitz since $h \in \clc_b^1(\clh)$, and the last by the norm in \eqref{eq:norm-def-cb2}. 

The quantity in the last line is integrable with regard to $du \otimes \gamma_Z$. Consequently, by combining \eqref{eq:128}, \eqref{eq:129}, and the dominated convergence theorem, \eqref{eq:frech} follows. Thus $f_h$ is Fr\'echet differentiable with derivative \eqref{eq:Dfh-form}.

Then the bound in \eqref{eq:8} for $j = 1$ follows from the form of $D_F f_h(x)$ in \eqref{eq:Dfh-form} and the respective bound for $h$ since
\begin{multline}
    \sup_{x \in \clh} \|D_F f_h(x)\|_{\op(\cll(\clh:\RR))} \le \int_0^\infty e^{-u} \int_{\clh} \sup_{x \in \clh} \|(D_F h)(e^{-u}x + \sqrt{1-e^{-2u}} y)\|_{\op(\cll(\clh:\RR))} \gamma_Z(dy) du \\
    \le \sup_{x \in \clh} \|D_F h(x)\|_{\op(\cll(\clh:\RR))} \le 1.
\end{multline}
Similarly, the candidate for the second derivative is given by
\begin{equation}\label{eq:D2fh-form}
\begin{split}
D_F^2 f_h(x)[z_1,z_2]
\doteq \int_0^\infty e^{-2u}\int_\clh
(D_F^2h)(e^{-u}x+\sqrt{1-e^{-2u}}y)[z_1,z_2] \gamma_Z(dy)\,du,
\quad z_1,z_2\in\clh.
\end{split}
\end{equation}
Given $h \in \clc_b^2(\clh)$ and $\| h \|_{\clc_b^2(\clh)} \leq 1$, the same arguments as for the first derivative then show that $f_h$ is twice Fr\'echet differentiable with second derivative $D_F ^2 f_h$, as well as the bound in \eqref{eq:8} for $j = 2$. 
\end{proof}

\begin{remark} \label{rmk:frechet-vs-gross-dif}
    \begin{enumerate}
        \item We note that, in fact, $h \in \clc_b^2(\clh)$ is a strong condition. For example, in important separable Banach spaces such as the space of continuous trajectories $C([0,1])$, there are no non-trivial, Fr\'echet differentiable functions with bounded support; see the discussion in Section 5.4 of \cite{Bogachev1998}. Despite this, it turns out that the class $\clc_b^2(\clh)$ is broad enough to metrize weak convergence whenever $\clh$ is a Hilbert space; see \cite{BasBurCamPec25}. \label{rmk:frechet-vs-gross-dif-item1}
        \item In the statement above, we do not claim that $f_h \in \clc_b^2(\clh)$. This is because the claim that $\sup_{x \in \clh} |f_h(x)| < \infty$ is not immediate, even if $h \in \clc_b^2(\clh)$. Nonetheless, it turns out that it is not required for the aims of this article.
    \end{enumerate}
\end{remark}

\begin{lemma} \label{le:analogue-Th-2.6}
Let $F,G\in \mathbb{D}^{1,2}(\clh) \subset L^2(\Omega:\clh)$, and set
$\Gamma(F,G) = \langle D_M F, D_M G\rangle_{\mathfrak H}$. Then, 
    \begin{enumerate}[label=(\roman*)]
    \item     \label{item1:Dirichlet-Th-2.6} $\Gamma(\cdot,\cdot)$ is bilinear, almost surely positive (i.e., $\langle \Gamma(F,F) u, u \rangle_{\clh} \ge 0$), and 
\begin{equation}
  \Gamma(F,G) = (\Gamma(G,F))^*.
\end{equation}
    \item     \label{item2:Dirichlet-Th-2.6} Let $\varphi,\psi:\clh \to \clh$ be continuously Fréchet-differentiable such that, for some constant $C$,
    \[
    \sup_{x \in \clh} \|\varphi(x) \|_{\clh} \le C, \; \sup_{x \in \clh} \|D_F \varphi(x) \|_{\op} \le C, \quad \sup_{x \in \clh} \|\psi(x) \|_{\clh} \le C,\; \sup_{x \in \clh} \|D_F \psi(x) \|_{\op} \le C.
    \]
Then, for $F,G\in \mathbb{D}^{1,2}(\clh)$,
\begin{equation}
      \Gamma(\varphi(F),\psi(G))
      =
      \nabla_F \psi(G) \Gamma(F,G)(\nabla_F \varphi(F))^*.
\end{equation}
    \item     \label{item3:Dirichlet-Th-2.6}
For all $F,G \in \mathrm{dom}(-L)$,
\begin{equation}
      \mathbb E\big[ \tr_{\clh}\Gamma(F,G)\big]
  = \mathbb E\big[ \langle D_M F, D_M G\rangle_{\mathfrak H \otimes \clh}\big]
  = - \mathbb E\langle LF,  G\rangle_\clh
  = - \mathbb E\langle F,  LG\rangle_\clh.
\end{equation}

\item \label{item1:lem:Gamma-covariance}
Let $F$ admit the finite chaos expansion in \eqref{eq:finite-chaos-expansions} (without the dependence on $n$). Then, $\Gamma(F,-L^{-1}F) \in L^1(\Omega: \mathcal S_1(\clh))$ and
\begin{equation} \label{item1:lem:Gamma-covariance:eq:positivity-exp}
    \EE\big[ \Gamma(F,-L^{-1}F)\big] = \clt_F, \quad \text{i.e., }   \big\langle \EE\big[ \Gamma(F,-L^{-1}F)\big] u, v \big\rangle_{\clh}
    =
    \EE\big[ \langle F,u\rangle_{\clh} \, \langle F,v\rangle_{\clh} \big], 
\end{equation}
for all $u,v \in \clh$, where $\clt_F$ is the covariance operator of $F$.

\item \label{item2:lem:Gamma-covariance}
Let $F$ admit the single chaos expansion \eqref{eq:seq-multiple-fixed-chaos} (without the dependence on $n$). Then,
\begin{equation} \label{eq:Gamma-positive-fixed-chaos}
    \Gamma(F,-L^{-1}F)
    = \frac{1}{p} \big\langle D_M F, D_M F \big\rangle_{\mfh} = \frac{1}{p} \Gamma(F,F).
\end{equation}
Moreover, $\Gamma(F,-L^{-1}F)$ is an a.s. nonnegative (positive semidefinite) operator on $\clh$.
    \end{enumerate}
\end{lemma}

\begin{proof}
\textit{Proof of \ref{item1:Dirichlet-Th-2.6}:}
Bilinearity is immediate, and so is a.s. positivity upon noticing that $\Gamma(F,F) = \| D_M F \|^2_{\mfh} \in \cll(\clh:\clh)$. For symmetry, fix $u,v \in \clh$. Then, it follows from
    \begin{align}
    \langle \Gamma(F,G)u,v\rangle_\clh
    &=
    \langle \langle D_M F, D_M G \rangle_{\mfh} u,v\rangle_\clh
    \\&=
    \langle \langle D_M F, u\rangle_\clh ,\langle v, D_M G \rangle_\clh \rangle_{\mfh}
    \\&=
    \langle \langle v, D_M G \rangle_\clh, \langle D_M F, u\rangle_\clh \rangle_{\mfh}
=\langle \Gamma(G,F)v,u\rangle_\clh  \label{eq:lemma-gamma-u-v},
    \end{align}
by symmetry of the inner products $\langle \cdot, \cdot \rangle_{\mfh} $ and $\langle \cdot, \cdot \rangle_{\clh} $.

\textit{Proof of \ref{item2:Dirichlet-Th-2.6}:}
By (4.10) of Lemma 4.7. in \cite{kruse2014strong} (with $q=0, p =2$), the Malliavin chain rule for $\clh$-valued random variables gives
\begin{equation}
    D_M (\varphi(F))=\nabla_F \varphi(F) D_M F \in \mfh \otimes \clh,\qquad D_M (\psi(G))=\nabla_F \psi(G) D_M G \in \mfh \otimes \clh,
\end{equation}
where $\nabla_F \varphi(F)$ and $\nabla_F \psi(G)$ are elements of $\cll(\clh:\clh)$. Then, 
\begin{align*}
\langle \Gamma(\varphi(F),\psi(G))u, v\rangle_\clh
&=
\big\langle \langle D_M  \varphi(F),u\rangle_\clh,  \langle D_M  \psi(G),v\rangle_\clh\big\rangle_{\mathfrak H}
\\
&=
\big\langle \langle \nabla_F \varphi(F)D_M F,u\rangle_\clh,  \langle \nabla_F \psi(G)D_M G,v\rangle_\clh\big\rangle_{\mathfrak H}
\\
&=
\big\langle \langle D_M F,(\nabla_F \varphi(F))^* u\rangle_\clh,  \langle D_M G,(\nabla_F \psi(G))^* v\rangle_\clh\big\rangle_{\mathfrak H}
\\
&=
\langle \Gamma(F,G)(\nabla_F \varphi(F))^* u, (\nabla_F \psi(G))^* v\rangle_\clh
\\
&=
\langle \nabla_F \psi(G) \Gamma(F,G)(\nabla_F \varphi(F))^* u, v\rangle_\clh,
\end{align*}
which yields
\begin{equation}
\Gamma(\varphi(F),\psi(G))
=\nabla_F \psi(G) \Gamma(F,G) (\nabla_F \varphi(F))^*.    
\end{equation}

\textit{Proof of \ref{item3:Dirichlet-Th-2.6}:}
Using that $\delta$ is the adjoint of $D_M $ \eqref{eq:2.24inPoincare}, we have that
\begin{equation} \label{eq:proof-item3-al1}
\mathbb E\langle D_M F, D_M G\rangle_{\mathfrak H\otimes \clh}=\mathbb E\langle \delta D_M F, G\rangle_\clh.
\end{equation}
By \eqref{eq:L-def} (see also Lemma 2.4. in \cite{vidotto2025functional}), we have $\delta D_M F=-LF$. This combined with \eqref{eq:proof-item3-al1}, then gives 
\begin{equation} \label{eq:proof2.6-1}
\mathbb E\langle D_M F, D_M G\rangle_{\mathfrak H\otimes \clh}=\mathbb E\langle \delta D_M F, G\rangle_\clh
=- \mathbb E\langle LF, G\rangle_\clh = - \EE \langle F, L G \rangle_\clh.
\end{equation}
To identify the left-hand side with $\mathbb E[\tr_{\clh}\Gamma(F,G)]$, fix an orthonormal basis $\{ e_j \}_{j \in \NN}$ of $\clh$ and note that
\begin{equation}
\tr_{\clh}\Gamma(F,G)
=\sum_{j=1}^{\infty}\langle \Gamma(F,G)e_j,e_j\rangle_\clh
=\sum_{j=1}^{\infty}\langle \langle D_MF,e_j\rangle_\clh,  \langle D_MG,e_j\rangle_\clh\rangle_{\mathfrak H}
=\langle D_M F, D_M G\rangle_{\mathfrak H\otimes \clh},
\end{equation}
Hence, using \eqref{eq:proof2.6-1}, we get
\begin{equation}
\mathbb E\big[\tr_{\clh} \Gamma(F,G)\big]
=\mathbb E\big[\langle D_M F,D_M G\rangle_{\mathfrak H\otimes \clh}\big]
=- \mathbb E\langle LF, G\rangle_\clh.
\end{equation}
Symmetry of $L$ gives also $- \mathbb E\langle F,LG\rangle_\clh$.

\textit{Proof of \ref{item1:lem:Gamma-covariance}:}
Since $F$ has a finite chaos expansion, it is automatically in $\mathbb{D}^{1,2}(\clh)$. Fix $u,v \in \clh$ and define the real-valued random variables
\[
F_u \doteq \langle F,u\rangle_{\clh} ,
\qquad
F_v \doteq \langle F,v\rangle_{\clh}.
\]
The calculations performed in \eqref{eq:lemma-gamma-u-v} show that
\begin{align*}
\big\langle \EE\big[ \Gamma(F,-L^{-1}F)\big] u, v \big\rangle_{\clh}
= \EE\big[
\big\langle 
\langle D_M F, u\rangle_{\clh}, 
\langle D_M (-L^{-1} F), v\rangle_{\clh}
\big\rangle_{\mfh}
\big]
= \EE\big[ \langle D_M F_u, D_M (-L^{-1} F_v) \rangle_{\mfh} \big],
\end{align*}
where the last line follows by the linearity of $D_M$ and $L^{-1}$. In the last formula, the operators $D_M$ and $L^{-1}$ acting on real-valued random variables are interpreted in the classical sense, e.g., as in Proposition 2.7.4 and Definition 2.8.10 of \cite{nourdin2012normal}.

By the classical real-valued identity for $F_u,F_v$ with finite chaos expansions (see, e.g., the display below (6.2.2) in \cite{nourdin2012normal}),
\[
\EE\big[ \langle D_M F_u, -D_M L^{-1} F_v \rangle_{\mfh} \big]
= \EE[F_u F_v]
= \EE\big[ \langle F,u\rangle_{\clh} \, \langle F,v\rangle_{\clh} \big].
\]
In the last two displays above, we have slightly abused the notation, by not distinguishing between the Hilbert- and real-valued versions of the operators $D_M$ and $L^{-1}$. Thus \eqref{item1:lem:Gamma-covariance:eq:positivity-exp} holds. Finally, the Hilbert--Schmidt factorization of the Malliavin derivatives, as used above for the cross-terms, gives
\[
\EE\|\Gamma(F,-L^{-1}F)\|_{\cls_1(\clh)}
\leq
\bigl(\EE\|D_MF\|_{\mfh\otimes\clh}^2\bigr)^{1/2}
\bigl(\EE\|D_ML^{-1}F\|_{\mfh\otimes\clh}^2\bigr)^{1/2}
<\infty.
\]
The last inequality holds since $F$ has a finite chaos expansion. Thus $\Gamma(F,-L^{-1}F)\in L^1(\Omega:\cls_1(\clh))$, without any positivity assumption on this operator.

\textit{Proof of \ref{item2:lem:Gamma-covariance}:}
 If $F$ lives in the $p$-th chaos, then $L^{-1}F = -\frac{1}{p} F$, so
\[
\Gamma(F,-L^{-1}F)
= \left\langle D_M F, - D_M L^{-1}F \right\rangle_{\mfh}
= \frac{1}{p} \big\langle D_M F, D_M F \big\rangle_{\mfh} = \frac{1}{p} \Gamma(F,F)
\]
That this operator is a.s. nonnegative follows from item \ref{item1:Dirichlet-Th-2.6}.
\end{proof}

For the lemma below, we refer to \cite{janson2015higher} for notions of higher-order moments of Hilbert space-valued random variables. In particular, we mention that, if $Z$ is an $\clh$-valued random variable such that $\EE \|Z\|_{\clh}^k<\infty$, then the weak $k$-th moment defined by $\EE ( x_1^*(Z) \dots x_k^*(Z)) \in \RR$ exists, where $k \in \NN$ and $x_1^*,\dots x_k^* \in \clh^*$.

\begin{lemma} \label{lemma:equiv-weak-moment-all-bases}
Let $k=2,4$ and $\{F_n\}_{n \in \NN}$ be a sequence of centered, $\mathcal H$-valued random variables and $Z$ be a centered, $\mathcal H$-valued random variable such that $\EE \|Z\|_{\clh}^k < \infty$ and $\EE \|F_n\|_{\clh}^k < \infty$. The following are equivalent:  
    \begin{enumerate}[label=(\roman*)]
    \item the $k$-th weak moments of $\{F_n\}$ converge to the $k$-th weak moments of $Z$, i.e., $\EE ( x^*_1(F_n)\dots x^*_k(F_n)) \to \EE ( x^*_1(Z)\dots x^*_k(Z))$, for all $x_1^*,\dots,x_k^* \in \clh^*$.
    \label{item1:momentcond}
    \item $\EE \left[ \langle F_n, e_i \rangle_\clh^k\right] \to \EE \left[ \langle Z, e_i \rangle_\clh^k\right]$, for all $i \in \NN$ and for all orthonormal bases $\{e_i\}_{i \in \NN}$.
    \label{item2:momentcond}
    \end{enumerate}
\end{lemma}

\begin{proof}
By the Riesz representation theorem, every $x^*\in\mathcal H^*$ can be identified with an element $x \in \clh$ so that
$x^*(h)=\langle h,x\rangle_{\clh}$ for all $h\in\mathcal H$. Define the quadratic and quartic forms
\begin{equation}
    Q_n(u)\doteq\EE \left[ \langle F_n,u\rangle_\clh^2\right], 
    \qquad R_n(u)\doteq\EE \left[ \langle F_n,u\rangle_\clh^4\right],
\end{equation}
and similarly $Q,R$ with $F_n$ replaced by $Z$.

\textit{\ref{item1:momentcond} $\Rightarrow$ \ref{item2:momentcond}:}
Fix an orthonormal basis $\{e_i\}_{i \in \NN}$ and choose $x_1^*=\cdots=x_k^*=\langle\cdot,e_i\rangle_\clh$. Then by (i),
\begin{equation}
   \EE \left[ \langle F_n,e_i\rangle^k_\clh\right]  
   \to  
   \EE \left[ \langle Z,e_i\rangle^k_\clh\right],
   \qquad k=2,4,
\end{equation}
which is exactly \ref{item2:momentcond}.

\textit{\ref{item2:momentcond} $\Rightarrow$ \ref{item1:momentcond}:}
Identify as above $x_k^* = x_k \in \clh, k = 2,4$. First, Condition \ref{item2:momentcond} for all bases means that for any unit vector $u\in\mathcal H$ we can choose a basis containing $u$, and hence
\begin{equation} \label{eq:conv-Q}
Q_n(u)\to Q(u),\qquad R_n(u)\to R(u).    
\end{equation}
Furthermore, any element $u \in \clh$ can be written as a scalar times a unit vector $u = \|u\|_{\clh} \frac{u}{\|u\|_{\clh}}$. Then the same convergences \eqref{eq:conv-Q} hold for any $u \in \clh$ .

For $k=2$, for any $x_1^*,x_2^* \in\mathcal H^*$,
\begin{equation}
    \EE x_1^*(F_n) x_2^*(F_n) = \EE\langle F_n,x_1\rangle_\clh \langle F_n,x_2 \rangle_\clh
    = \tfrac14\Big(Q_n(x_1+x_2)-Q_n(x_1-x_2)\Big).
\end{equation}
Since $Q_n\to Q$ pointwise, we obtain
\begin{equation}
    \EE \langle F_n,x_1 \rangle_\clh \langle F_n,x_2\rangle_\clh
    \to \EE\langle Z,x_1 \rangle_\clh \langle Z,x_2 \rangle_\clh = \EE x_1^*(Z) x_2^*(Z).
\end{equation}

For $k=4$, let $x_i^* \in \clh^*, i = 1,2,3,4$ and denote their respective representations $x_i \in \clh, i = 1,2,3,4$. Then,
\begin{equation}
M^{(4)}_n(x_1,x_2,x_3,x_4)\doteq \EE  \Big[\prod_{j=1}^4 \langle F_n,x_j\rangle_\clh \Big].
\end{equation}
Using the general polarization identity as in equation (3) of \cite{nelson2007probability}, we obtain
\begin{equation}
M^{(4)}_n(x_1,x_2,x_3,x_4)
=\frac{1}{4!}\sum_{r=1}^4(-1)^{4-r}  \sum_{\substack{S\subset\{1,2,3,4\}\\ |S|=r}}
R_n \Big(\sum_{j\in S} x_j\Big).
\end{equation}
Since $R_n(\cdot)\to R(\cdot)$ pointwise on $\mathcal H$, we may pass to the limit to get
\begin{equation}
M^{(4)}_n(x_1,x_2,x_3,x_4) \to
\frac{1}{4!}\sum_{r=1}^4(-1)^{4-r}  \sum_{\substack{S\subset\{1,2,3,4\}\\ |S|=r}}
R \Big(\sum_{j\in S} x_j\Big)
=\EE \Big[\prod_{j=1}^4 \langle Z,x_j\rangle_\clh \Big],
\end{equation}
which concludes the proof.
\end{proof}

The next lemma states sufficient conditions for convergence in trace-class norm and goes back to \cite{arazy1981more}; see the theorem after the appendix therein. Although the results there are stated for operators on $\ell^2$, they hold true for operators on any separable Hilbert space by leveraging the isometric isomorphism between separable Hilbert spaces.

\begin{lemma}[Arazy, 1981] \label{lem:trace-condition}
Let $\clh$ be a separable Hilbert space and let $\{\clt_{F_n}\}_{n\in\mathbb{N}},\clt_Z\in \mathcal{S}_1(\clh)$ be non-negative trace-class operators. Assume
\begin{enumerate}
  \item $\tr(\clt_{F_n})\to \tr(\clt_Z)$,
  \label{lem:trace-condition-item1}
  \item $\clt_{F_n}\to \clt_Z$ weakly in $\mathcal{S}_1(\clh)$, i.e.,
  $\tr(A \clt_{F_n})\to \tr(A \clt_Z)$ for all $A\in\mathcal{L}(\clh)$.
  \label{lem:trace-condition-item2}
\end{enumerate}
Then,
\[
  \bigl\|\clt_{F_n}-\clt_Z\bigr\|_{\mathcal{S}_1(\clh)}  \xrightarrow[n\to\infty]{}  0 .
\]
\end{lemma}

For additional conditions ensuring convergence of covariance operators in trace-class norm, we refer to \cite{kubrusly1986convergence}.

\paragraph{Acknowledgments:}
The authors are grateful to Giovanni Peccati for drawing their attention to issues in prior literature, including their own work. They also thank Amarjit Budhiraja and Chiranjib Mukherjee for valuable professional advice and guidance. Finally, they are grateful for two anonymous referee reports, which have improved the quality and readability of the manuscript. Notably, one reviewer's comments led to the removal of the nondegeneracy assumption for the covariance of the target Gaussian measures.

PZ's research was funded by the DFG under Germany's Excellence Strategy EXC 2044–390685587, Mathematics Münster: Dynamics–Geometry–Structure.

\small
\bibliographystyle{acm}
\bibliography{mybibliography}

\end{document}